\documentclass[12pt,pdftex,a4paper]{amsart}

\selectfont

\usepackage{qtree}
\usepackage{forest}
\usepackage{mathtools} % small fixes & extras over amsmath

\usepackage{caption}

\usepackage{etoolbox}
\usepackage{epic,eepic,ecltree}
\usepackage{graphicx}
\usepackage{tikz}
\usepackage{tikz-cd}

\usetikzlibrary{decorations.text}
\usepackage{amssymb,color, euscript, enumerate}
\usepackage{amsthm}
\usepackage{amsmath}
\usepackage{braket}
\usepackage{amscd}
\usepackage{txfonts}
\usepackage{comment}
\usepackage{bm}
\usepackage{booktabs}
\usepackage{array}

\usepackage{wrapfig}
\usepackage{graphicx}

\usepackage{amscd}
\numberwithin{equation}{section}

\makeatletter
\let\old@tocline\@tocline
\let\section@tocline\@tocline
\newcommand{\subsection@dotsep}{4.5}
\newcommand{\subsubsection@dotsep}{4.5}
\patchcmd{\@tocline}
  {\hfil}
  {\nobreak
     \leaders\hbox{$\m@th
        \mkern \subsection@dotsep mu\hbox{.}\mkern \subsection@dotsep mu$}\hfill
     \nobreak}{}{}
\let\subsection@tocline\@tocline
\let\@tocline\old@tocline

\patchcmd{\@tocline}
  {\hfil}
  {\nobreak
     \leaders\hbox{$\m@th
        \mkern \subsubsection@dotsep mu\hbox{.}\mkern \subsubsection@dotsep mu$}\hfill
     \nobreak}{}{}
\let\subsubsection@tocline\@tocline
\let\@tocline\old@tocline

\let\old@l@subsection\l@subsection
\let\old@l@subsubsection\l@subsubsection

\def\@tocwriteb#1#2#3{%
  \begingroup
    \@xp\def\csname #2@tocline\endcsname##1##2##3##4##5##6{%
      \ifnum##1>\c@tocdepth
      \else \sbox\z@{##5\let\indentlabel\@tochangmeasure##6}\fi}%
    \csname l@#2\endcsname{#1{\csname#2name\endcsname}{\@secnumber}{}}%
  \endgroup
  \addcontentsline{toc}{#2}%
    {\protect#1{\csname#2name\endcsname}{\@secnumber}{#3}}}%

\newlength{\@tocsectionindent}
\newlength{\@tocsubsectionindent}
\newlength{\@tocsubsubsectionindent}
\newlength{\@tocsectionnumwidth}
\newlength{\@tocsubsectionnumwidth}
\newlength{\@tocsubsubsectionnumwidth}
\newcommand{\settocsectionnumwidth}[1]{\setlength{\@tocsectionnumwidth}{#1}}
\newcommand{\settocsubsectionnumwidth}[1]{\setlength{\@tocsubsectionnumwidth}{#1}}
\newcommand{\settocsubsubsectionnumwidth}[1]{\setlength{\@tocsubsubsectionnumwidth}{#1}}
\newcommand{\settocsectionindent}[1]{\setlength{\@tocsectionindent}{#1}}
\newcommand{\settocsubsectionindent}[1]{\setlength{\@tocsubsectionindent}{#1}}
\newcommand{\settocsubsubsectionindent}[1]{\setlength{\@tocsubsubsectionindent}{#1}}

\renewcommand{\l@section}{\section@tocline{1}{\@tocsectionvskip}{\@tocsectionindent}{}{\@tocsectionformat}}%
\renewcommand{\l@subsection}{\subsection@tocline{2}{\@tocsubsectionvskip}{\@tocsubsectionindent}{}{\@tocsubsectionformat}}%
\renewcommand{\l@subsubsection}{\subsubsection@tocline{3}{\@tocsubsubsectionvskip}{\@tocsubsubsectionindent}{}{\@tocsubsubsectionformat}}%
\newcommand{\@tocsectionformat}{}
\newcommand{\@tocsubsectionformat}{}
\newcommand{\@tocsubsubsectionformat}{}
\expandafter\def\csname toc@1format\endcsname{\@tocsectionformat}
\expandafter\def\csname toc@2format\endcsname{\@tocsubsectionformat}
\expandafter\def\csname toc@3format\endcsname{\@tocsubsubsectionformat}
\newcommand{\settocsectionformat}[1]{\renewcommand{\@tocsectionformat}{#1}}
\newcommand{\settocsubsectionformat}[1]{\renewcommand{\@tocsubsectionformat}{#1}}
\newcommand{\settocsubsubsectionformat}[1]{\renewcommand{\@tocsubsubsectionformat}{#1}}
\newlength{\@tocsectionvskip}
\newcommand{\settocsectionvskip}[1]{\setlength{\@tocsectionvskip}{#1}}
\newlength{\@tocsubsectionvskip}
\newcommand{\settocsubsectionvskip}[1]{\setlength{\@tocsubsectionvskip}{#1}}
\newlength{\@tocsubsubsectionvskip}
\newcommand{\settocsubsubsectionvskip}[1]{\setlength{\@tocsubsubsectionvskip}{#1}}

\patchcmd{\tocsection}{\indentlabel}{\makebox[\@tocsectionnumwidth][l]}{}{}
\patchcmd{\tocsubsection}{\indentlabel}{\makebox[\@tocsubsectionnumwidth][l]}{}{}
\patchcmd{\tocsubsubsection}{\indentlabel}{\makebox[\@tocsubsubsectionnumwidth][l]}{}{}

\newcommand{\@sectypepnumformat}{}
\renewcommand{\contentsline}[1]{%
  \expandafter\let\expandafter\@sectypepnumformat\csname @toc#1pnumformat\endcsname%
  \csname l@#1\endcsname}
\newcommand{\@tocsectionpnumformat}{}
\newcommand{\@tocsubsectionpnumformat}{}
\newcommand{\@tocsubsubsectionpnumformat}{}
\newcommand{\setsectionpnumformat}[1]{\renewcommand{\@tocsectionpnumformat}{#1}}
\newcommand{\setsubsectionpnumformat}[1]{\renewcommand{\@tocsubsectionpnumformat}{#1}}
\newcommand{\setsubsubsectionpnumformat}[1]{\renewcommand{\@tocsubsubsectionpnumformat}{#1}}
\renewcommand{\@tocpagenum}[1]{%
  \hfill {\mdseries\@sectypepnumformat #1}}

\let\oldappendix\appendix
\renewcommand{\appendix}{%
  \leavevmode\oldappendix%
  \addtocontents{toc}{%
    \protect\settowidth{\protect\@tocsectionnumwidth}{\protect\@tocsectionformat\sectionname\space}%
    \protect\addtolength{\protect\@tocsectionnumwidth}{2em}}%
}
\makeatother

\makeatletter
\settocsectionnumwidth{2em}
\settocsubsectionnumwidth{2.5em}
\settocsubsubsectionnumwidth{3em}
\settocsectionindent{1pc}%
\settocsubsectionindent{\dimexpr\@tocsectionindent+\@tocsectionnumwidth}%
\settocsubsubsectionindent{\dimexpr\@tocsubsectionindent+\@tocsubsectionnumwidth}%
\makeatother

\settocsectionvskip{10pt}
\settocsubsectionvskip{0pt}
\settocsubsubsectionvskip{0pt}
    
\settocsectionformat{\bfseries}
\settocsubsectionformat{\mdseries}
\settocsubsubsectionformat{\mdseries}
\setsectionpnumformat{\bfseries}
\setsubsectionpnumformat{\mdseries}
\setsubsubsectionpnumformat{\mdseries}

\let\oldtableofcontents\tableofcontents
\renewcommand{\tableofcontents}{%
  \vspace*{-\linespacing}% Default gap to top of CONTENTS is \linespacing.
  \oldtableofcontents}

\newcommand{\msl}{\mathrm{sl}}

\newcommand{\zee}{\bar{\zeta}}

\newcommand{\Harm}{\mathrm{Harm}}

\newcommand{\Embh}{\mathrm{Mfld}_{2,\mathrm{HS}}^{\mathrm{CO}}}

\newcommand{\Emb}{\mathrm{Mfld}_{2,\mathrm{Emb}}^{\mathrm{CO}}}

\newcommand{\Embcd}{\mathrm{Mfld}_{d,\mathrm{emb}}^{\mathrm{CO}}}

\newcommand{\Embc}{\mathrm{Mfld}_2^{\mathrm{CO}}}

\newcommand{\Disk}{\mathrm{Disk}_2^{\mathrm{CO}}}

\newcommand{\Diskc}{\mathrm{Disk}_{2,\mathrm{emb}}^{\mathrm{CO}}}

\newcommand{\Hilb}{{\mathrm{Hilb}}}
\newcommand{\Ind}{{\mathrm{Ind}}}

\newcommand{\fG}{\mathbb{G}}
\newcommand{\bB}{\mathbb{B}}

\newcommand{\Hom}{{\mathrm{Hom}}}

\newcommand{\dn}{{\frac{d-2}{2}}}

\newcommand{\Cc}{C_c^\infty}

\newcommand{\cC}{{\mathcal C}}

\newcommand{\N}{\mathbb{N}}
\newcommand{\Z}{\mathbb{Z}}
\newcommand{\R}{\mathbb{R}}
\newcommand{\C}{\mathbb{C}}

\newcommand{\Om}{\Omega}

\newcommand{\Conf}{\mathrm{Conf}}

\newcommand{\sod}{\mathrm{so}(d+1,1)}

\newcommand{\alg}{\text{alg}}

\newcommand{\std}{{\text{std}}}

\newcommand{\va}{\bm{1}}

\newcommand{\id}{{\mathrm{id}}}

\newcommand{\z}{{\bar{z}}}

\newcommand{\w}{{\bar{w}}}

\newcommand{\m}{{\rho}}

\newcommand{\pa}{{\partial}}

\newcommand{\Vect}{{\underline{\text{Vect}}_\C}}

\newcommand{\cN}{{\mathcal{N}}}

\newcommand{\al}{\alpha}
\newcommand{\ep}{\epsilon}
\newcommand{\be}{\beta}

\newcommand{\ze}{\zeta}

\newcommand{\om}{\omega}

\newcommand{\si}{\sigma}

\newcommand{\hf}{\hat{h}}
\newcommand{\cl}{\mathrm{cl}}

\newcommand{\fh}{{\hat{\mathfrak{h}}}}

\newcommand{\ft}{\frac{1}{2}}

\newcommand{\ba}{\bar{a}}

\newcommand{\Ld}{{\overline{L}}}

\newcommand{\Sym}{\mathrm{Sym}}

\newcommand{\Hf}{{\mathcal{H}}}

\newcommand{\End}{\mathrm{End}}

\newcommand{\CE}{\mathbb{CE}_2}
\newcommand{\CEc}{\mathbb{CE}_{2}^\mathrm{emb}}
\newcommand{\CEcd}{\mathbb{CE}_{d}^\mathrm{emb}}
\newcommand{\CEh}{\mathbb{CE}_2^{\mathrm{HS}}}

\newcommand{\CEd}{\mathbb{CE}_d}
\newcommand{\CEf}{\mathbb{CE}}

\newcommand{\J}{\mathcal{J}}

\newcommand{\hotimes}{\hat{\otimes}}

\newcommand{\bD}{\mathbb{D}}

\newcommand{\bs}{\hat{S}}

\newcommand{\norm}[1]{\left\lVert #1 \right\rVert}

\newtheorem{thm}{Theorem}[section]
\newtheorem{dfn}[thm]{Definition}
\newtheorem{lem}[thm]{Lemma}
\newtheorem{prop}[thm]{Proposition}
\newtheorem{cor}[thm]{Corollary}
\newtheorem{rem}[thm]{Remark}

\theoremstyle{remark}

\newtheorem{mainthm}{Theorem}

\begin{document}

\begin{center}
{{\LARGE \bf Bergman space, Conformally flat 2-disk operads and affine Heisenberg vertex algebra}
} \par \bigskip

\renewcommand*{\thefootnote}{\fnsymbol{footnote}}
{\normalsize
Yuto Moriwaki \footnote{email: \texttt{moriwaki.yuto (at) gmail.com}}
}
\par \bigskip
{\footnotesize Interdisciplinary Theoretical and Mathematical Science Program (iTHEMS)\\
Wako, Saitama 351-0198, Japan}

\par \bigskip
\end{center}

\noindent

\vspace{5mm}

\begin{center}
\textbf{\large Abstract}
\end{center}
In this paper we consider the operad of holomorphic disk embeddings 
of the unit disk $\mathbb D \subset \mathbb C$.
We introduce a suboperad $\mathbb{CE}_2^{HS}$ defined by square-integrability conditions and show that the symmetric algebra
$\mathrm{Sym} A^{2}(\mathbb D)$ of the Bergman space carries a natural $\mathbb{CE}_2^{HS}$-algebra structure.
Conformally flat factorization homology with coefficients in $\mathrm{Sym} A^{2}(\mathbb D)$ then yields metric-dependent invariants of
two-dimensional Riemannian manifolds.
Moreover, $\mathrm{Sym} A^{2}(\mathbb D)$ is identified with the ind-Hilbert space completion of the affine Heisenberg vertex operator algebra.

\vspace{12mm}

%
%この論文では我々は単位円盤 $\bD=\{z\in \C \mid |z|<1\}$の正則埋め込みのなすoperadの二乗可積分性で特徴付けられる部分operad $\CEh$を導入する。$\CEh(1)$は unit disk 上の Grunsky operator が Hilbert-Schmidt になる正則単射のなすモノイドである。我々は
%Bergman空間の対称積が$\CEh$代数の構造を持つことを示す。このとき$\Sym(A^2(\bD))$は左kan拡張により、2次元リーマン多様体の計量に依存した不変量を定める。 $\Sym(A^2(\bD))$はaffine heisenberg VOAのind-Hilbert spaceとしての完備化になっており、
%共形平坦2-disk代数とunitary full VOA との間の関係を提案する。
%
%\vspace{3mm}

%In this paper we study the operad of holomorphic embeddings of the unit disk
%$\mathbb{D}=\{z\in\mathbb C\mid |z|<1\}$.
%We introduce a suboperad $\mathbb{CE}_2^{HS}$, characterized by square-integrability conditions.
%Its unary part $\mathbb{CE}_2^{HS}(1)$ is the monoid of univalent maps on $\mathbb{D}$ whose Grunsky operator is Hilbert-Schmidt.
%We prove that the symmetric algebra of the Bergman space carries a $\mathbb{CE}_2^{HS}$-algebra structure,
%whose left Kan extension yields metric-dependent invariants of two-dimensional Riemannian manifolds.
%Moreover, $\Sym A^{2}(\mathbb{D})$ is identified with an ind-Hilbert space completion of the affine Heisenberg vertex operator algebra.

\tableofcontents

\clearpage

\begin{center}
\textbf{\large Introduction}
\end{center}

The local structure of a two-dimensional chiral conformal field theory is encoded by a \emph{vertex operator algebra} (VOA) with  holomorphic product \cite{BPZ,B1,FLM}
\begin{align}
Y(\bullet,z):V \otimes V \rightarrow V((z)),\qquad 
Y(a,z)b = \sum_{n\in\Z}a(n)b\, z^{-n-1}.
\label{intro_product_z}
\end{align}
Such local data can be globalized to Riemann surfaces of arbitrary genus via conformal blocks, using methods from algebraic analysis and algebraic geometry \cite{TUY,Zh,DGT}.
The Monster moonshine phenomenon is one manifestation of this local-to-global principle in conformal field theory \cite{CN,B2}.

%The local structure of a two-dimensional chiral conformal field theory is encoded by a {\it vertex operator algebra} (VOA) equipped with a ``holomorphic product'' \cite{B1,FLM}
%\begin{align}
%Y(\bullet,z):V \otimes V \rightarrow V((z)),\qquad Y(a,z)b = \sum_{n\in\Z}a(n)b\, z^{-n-1}.
%\label{intro_product_z}
%\end{align}
%The product \eqref{intro_product_z} is a formal power series satisfying a certain algebraic identity. 
%Such formal (local) data are lifted to global data on Riemann surfaces of arbitrary genus via algebraic analysis methods using conformal blocks \cite{TUY,Zh,DGT}.
%%By methods of algebraic analysis, the conformal blocks of Tsuchiya-Ueno-Yamada
%%% and the chiral homology of Beilinson--Drinfeld, 
%% such formal (local) data can be lifted to global data on Riemann surfaces of arbitrary genus \cite{TUY,Zh,DGT}.
%The Monster moonshine phenomenon was one manifestation of this local-to-global principle in conformal field theory.

On the other hand, a two-dimensional non-chiral conformal field theory in general possesses a ``real-analytic product'' \cite{MS,FRS,HK,M1}
\begin{align*}
Y(-,z,\z):F \otimes F \rightarrow F((z,\z,|z|^\R)), \qquad Y(a,z,\z)b=\sum_{{r,s\in\R}} a(r,s)b\, z^{-r-1}\z^{-s-1},
%\label{intro_vertex}
\end{align*}
and VOAs appear 
as two canonical subalgebras, corresponding to the holomorphic and anti-holomorphic parts of $F$, through the exceptional isomorphism in two dimensions,
\begin{align}
\mathrm{so}(3,1)_\C \cong \mathrm{sl}_2\C\oplus \mathrm{sl}_2\C
\label{intro_exception}
\end{align}
(see \cite[Proposition 3.12]{M1}).
In dimensions $d \geq 3$, the Lie algebra $\mathrm{so}(d+1,1)$ does not split as in \eqref{intro_exception}.
Accordingly, conformal field theories in higher dimensions, as well as non-chiral full conformal field theories in $d=2$, are typically described in terms of Hilbert spaces, distributions, probability measures, or von Neumann algebras
%and not holomorphic functions
 \cite{LM,AGT,AMT,DKRV,BKLR}.
 In such settings, one cannot expect to describe the local-to-global principle by purely algebro-geometric methods.
Therefore, a framework that captures this principle without relying on {holomorphicity} is required.
%Conformal field theories in dimension $d \geq 3$, where \eqref{intro_exception} does not exist, as well as non-chiral full conformal field theories  in $d=2$, are described in terms of Hilbert spaces, distributions, probability measures or von Neumann algebras \cite{LM,AGT, AMT,DKRV,BKLR}.

As a tool for this purpose, we introduced {\it conformally flat factorization homology} in \cite{MLeft}.
Conformally flat factorization homology is a conformal Riemannian-geometric variation of  factorization homology \cite{Lurie2,AF1}, and its input data are given by {\it conformally flat $d$-disk algebras}, which form an ind-Hilbert space refinement of Costello-Gwilliam {\it factorization algebras} \cite{CG1,CG2}.

The first aim of this paper is to show that the completion of the affine Heisenberg vertex operator algebra as an ind-Hilbert space carries the structure of a conformally flat $2$-disk algebra structure.
We realize this structure on the symmetric algebra of the Bergman space $A^2(\bD)$ using the factorization algebra associated with the conformal Laplacian \cite{Mfactorization}.
%二次元のカイラル共形場理論の局所的な構造は正則な積を持つ代数
%\begin{align}
%Y(\bullet,z):V \otimes V \rightarrow V((z)),\qquad Y(a,z)b = \sum_{n\in\Z}a(n)b\, z^{-n-1},
%\label{intro_product_z}
%\end{align}
%頂点作用素代数によって記述される\cite{B1,FLM}。積\eqref{intro_product_z}はある代数的恒等式 (Borcherds identity) を満たす形式的級数であるが、
%こうした形式的(局所的)なデータは代数解析/代数幾何的な手法により土屋上野山田の共形ブロックによって、一般の種数のリーマン面上の大域的なデータへと持ち上げられる\cite{TUY,BD,Zhu}。
%モンスター・ムーンシャイン現象はこうした共形場理論の局所大域的な対応の一つの現れであった。
%ただし2次元共形場理論は一般には実解析的な積
%\begin{align*}
%Y(-,z,\z):F \otimes F \rightarrow F((z,\z,|z|^\R)), \qquad Y(a,z,\z)b=\sum_{\substack{r,s\in\R\\r-s \in\Z}} a(r,s)b\, z^{-r-1}\z^{-s-1}
%%\label{intro_vertex}
%\end{align*}
%を持ち\cite{HK}、VOA は$2$次元における例外的な同型
%\begin{align}
%\mathrm{so}(3,1)_\C \cong \mathrm{sl}_2\C\oplus \mathrm{sl}_2\C\label{intro_exception}
%\end{align}
%を通じて、$F$の正則と反正則な部分からなる二つのcanonical な部分代数として現れる\cite[Proposition 3.12]{M1}。
%\eqref{intro_exception}の存在しない$d \geq 3$の共形場理論や、$d=2$であってもカイラルでない full 共形場理論は実解析的な関数やヒルベルト空間、超関数、確率測度、作用素環を用いて記述されるため\cite{AMT,DKRV,KRV,GKRV1,BKLR}、代数解析を使って局所大域的な対応を記述することは望めない。そこで正則性を用いずに共形場理論の局所大域的な対応を記述する枠組みが必要である。
%
%こうした一般のCFTの局所大域的な対応を記述するための道具として、我々は\cite{MLeft}において共形平坦因子化ホモロジーを導入した。
%共形平坦因子化ホモロジーは、Lurie や AFの因子化ホモロジー\cite{Lurie2,AF1}の共形リーマン幾何的なバリエーションであり、このホモロジー理論の input data はCostello-Gwilliam の因子化代数の refinementである 共形平坦d-disk algebra によって与えられる。
%
%
%この論文の第一の目的は affine Heisenberg vertex operator algebra の ind-Hilbert space としての完備化に共形平坦 d-disk algebra の構造が入ることを
%Bergman 空間$A^2(\bD)$の対称積$\Sym A^2(\bD)$と共形ラプラシアンに付随する因子化代数\cite{Mfactorization}を用いて証明することである。

\begin{table}[h]
\centering
\renewcommand{\arraystretch}{1.15}
\setlength{\tabcolsep}{8pt}
\begin{tabular}{@{}p{0.25\linewidth} p{0.3\linewidth} p{0.4\linewidth}@{}}
\toprule
 & \textbf{Vertex algebra} & \textbf{$\CEh$-algebra} \\
\midrule
Vector space
  & $M(0) = \mathrm{Ind}_{\hf_+}^{\hf}\C \va$
  & $\bigoplus_{p=0}^\infty \Sym^p A^2(\bD)$
\\
translation
  & $Y(h(-n-1)\va,\zeta)\va$
  & $\frac{(n+1)z^n}{(1-\zeta z)^{n+2}} \in A^2(\bD)$
%$1$-particle states
%  & $h(-n-1)\va$
%  & $(n+1)z^n$.
\\
$p$-particle states
  & $h(-k_1-1)\cdots h(-k_p-1)\va$
  & $\sqrt{p!}\bs^p((k_1+1)z^{k_1} \otimes \cdots \otimes (k_p+1)z^{k_p})$.
\\
%Dilation by $r\in (0,1]$
%  & $r^{L(0)}$
%  & $r^p f(rz_1,\dots,rz_p)$
%\\
vertex operator
  & $Y(r^{L(0)}a,\zeta)s^{L(0)}b$
  & $\m_{B_{\zeta,r},B_{s,0}}(a,b)$
\\
\bottomrule
\end{tabular}
\label{tab:dictionary}
\end{table}

%「係数環」であるBergman空間 $\Sym A^2(\bD)$ を左カン拡張することで2次元リーマン多様体の計量に依存した不変量が構成される。
%標語的に言えば「一般に unitary (full) VOA の完備化は計量に依存したホモロジー理論の係数環である」と期待され、
%共形場理論の正則性に依存しない局所大域的な対応を与える。
%
%この論文のもう一つの目的は2次元の共形幾何を記述する $2$-disk operad として、
%\cite{MLeft}で導入した$\CE$を拡大したより精密な operad $\CEh$を導入することである。
%後で議論するように$\CEh$はTakhtanjan-TeoによるHilbert manifold を用いた universal Teihmuller 空間の精密化と密接に関係している。
%共形ブロックの理論では頂点作用素代数からリーマン面の「代数幾何的な」 moduli 空間の上の sheaf を与えるが\cite{DGT}、こうした関係性は我々の枠組みではunitary full VOA から「Hilbert 解析的な」 Teihmuller 空間の理論が得られると期待される。
%
%以下では、まず頂点作用素代数の立場から共形平坦$2$-disk algebra とその有界性の問題がどのように現れるかを述べ、その後 Bergman 空間の例について説明する。

Conformally flat factorization homology with coefficients in $\Sym A^{2}(\bD)$
defines a metric-dependent invariant of two-dimensional Riemannian manifolds.
Heuristically, the completion of a unitary (full) vertex operator algebra should provide such coefficients,
thereby yielding a local-to-global principle for conformal field theory that does not rely on holomorphicity.

%
%
%By taking the left Kan extension with
%the conformally flat $2$-disk algebra $\Sym A^2(\bD)$, one obtains a metric-dependent invariant of two-dimensional Riemannian manifolds.
%Heuristically, the completion of a unitary (full) vertex operator algebra should provide the coefficient system for conformally flat factorization homology thus giving a local-to-global principle in conformal field theory beyond holomorphic methods.

%From this construction, one obtains in particular a metric-dependent invariant of two-dimensional Riemannian manifolds
%$\mathrm{H}_{\mathrm{CF}}((M,g),\Sym A^2(\bD))$ with coefficients in the Bergman space.
%In general, we expect that the completion of a unitary (full) VOA provides the coefficient system of the metric-dependent homology theory, thereby giving a local-to-global principle in conformal field theory that does not rely on holomorphicity.

Another aim of this paper is to introduce a refined operad $\CEh$, extending the conformally flat 2-disk operad $\CE$ introduced in \cite{MLeft},
and show that $\Sym A^2(\bD)$ inherits a $\CEh$-algebra structure.
As will be discussed later, $\CEh$ is closely related to the refinement of the universal Teichmüller space as a Hilbert manifold due to  Takhtajan-Teo \cite{TT}.
In the theory of conformal blocks, a vertex operator algebra gives rise to a sheaf over the algebro-geometric moduli space of Riemann surfaces \cite{DGT}.
In contrast, within our framework we expect that the completion of a unitary full VOA gives rise to a {Hilbert-analytic analogue of conformal blocks} over Teichm\"uller space.
%we expect that the completion of a unitary full VOA gives rise to a Hilbert-analytic framework for Teichm\"{u}ller space.

In what follows, we first explain how the issue of (un)boundedness arises from the viewpoint of vertex operator algebras, and 
then illustrate the operad $\CEh$ and its action in the Bergman space example.
%
%In what follows, we first explain 
%the issue of (un)boundedness arises from the perspective of vertex operator algebras, and then illustrate the definition of the operad $\CEh$ and the construction in the example of the Bergman space.

\vspace{3mm}

\noindent
\begin{center}
{\bf 0.1. Hilbert space completions of VOAs and (un)boundedness}
\end{center}

%
%
%\vspace{3mm}
%
%
%\textbf{Hilbert space completions of vertex operator algebras and (un)boundedness}

A bridge between VOAs and algebraic QFT (the formulation of QFT in terms of von Neumann algebras) was constructed by Carpi-Kawahigashi-Longo-Weiner, who considered the Hilbert space completion $H_V$ of a unitary VOA \cite{CKLW}.
This line of work was subsequently developed in joint work with Adamo and Tanimoto, leading to a relation between non-chiral unitary full VOAs and axiomatic QFT \cite{AMT}.
%This line of work was subsequently developed further by the author and Adamo-Tanimoto into a relation between non-chiral unitary full VOAs and axiomatic QFT \cite{AMT}.
%こうした試みは著者とAdamo-Tanimotoにより
%non-chiral な unitary full VOA と axiomatic QFT (ヒルベルト空間とその上の非有界作用素) との間の関係に発展した\cite{AMT}。
Let $V$ be a unitary (full) VOA and let $a_i \in V$ be quasi-primary vectors. 
If $1>|\ze_1|>|\ze_2|>\cdots>|\ze_n|>0$, then the product of vertex operators converges absolutely and defines a vector in 
$H_V$:
%the Hilbert space $H_V$:
\begin{align}
Y(a_1,\ze_1,\zee_1)\cdots Y(a_n,\ze_n,\zee_n)\va \in H_V
\label{eq_comp}
\end{align}
(see \cite[Section 3.3]{AMT}).
Physically, the expression \eqref{eq_comp} corresponds to a state obtained by inserting $a_i\in V$ at the points $\ze_i \in \bD=\{|z|<1\}$, and the Hilbert space $H_V$ may thus be viewed as the space of states on the unit disk $\bD$.
However, as we shall see below, $Y(\bullet,z,\z)$ is in general unbounded with respect to the norm on $H_V$, and this unboundedness, which is ubiquitous in quantum field theory, creates substantial analytic difficulties \cite{GW}.

In this paper we consider the affine Heisenberg vertex operator algebra $M(0)$ generated by $h \in M(0)_1$.
Its (analytically continued) two-point function is given by
\begin{align*}
(\va, Y(L(-1)^n h,\zeta_1)Y(L(-1)^m h,\zeta_2)\va ) 
= (-1)^n \frac{(n+m+1)!}{(\zeta_1-\zeta_2)^{n+m+2}},
\end{align*}
and even after normalization using $\norm{L(-1)^n h}=\sqrt{n+1}\,n!$, one encounters divergence through the binomial coefficient
$\binom{n+m}{n}\frac{1}{(\zeta_1-\zeta_2)^{n+m+2}}$.

\begin{wrapfigure}{r}{0.18\textwidth}
  \centering
 \vspace{-0.8\baselineskip} % 必要なら微調整
  \includegraphics[width=0.18\textwidth]{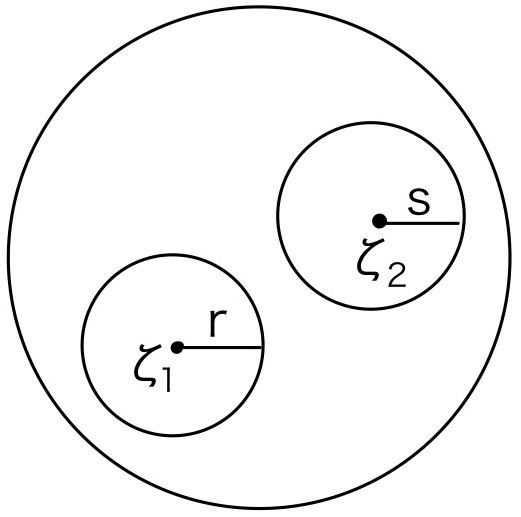}
%  \caption{A configuration in the operad $\CE$.}
  \label{fig_intro}
  \vspace{-0.8\baselineskip}
\end{wrapfigure}
Let $B_r(\zeta)$ denote the open disk of radius $r$ centered at $\zeta$. 
Our key observation is the following. 
Choose $1>r>0$ and $1>s>0$ such that
\begin{align}
\overline{B_r(\zeta_1)} \cap \overline{B_s(\zeta_2)}=\emptyset,
\quad\text{and}\quad 
B_r(\zeta_1),B_s(\zeta_2)\subset \bD.
\label{intro_geom_sep}
\end{align}
Consider the radius-normalized two-point function
\begin{align*}
\left\langle 
\va,
Y\!\left(r^{L(0)} \frac{L(-1)^n h}{\norm{L(-1)^n h}},\zeta_1\right)
Y\!\left(s^{L(0)} \frac{L(-1)^m h}{\norm{L(-1)^m h}},\zeta_2\right)
\va
\right\rangle
\sim 
\binom{n+m}{n}
\frac{r^n s^m}{(\zeta_1-\zeta_2)^{n+m}}.
\end{align*}
From \eqref{intro_geom_sep} we obtain
\[
\sigma
=
\frac{r}{|\zeta_1-\zeta_2|}
+
\frac{s}{|\zeta_1-\zeta_2|}
<1,
\]
which implies $
\binom{n+m}{n}
\frac{r^n s^m}{(\zeta_1-\zeta_2)^{n+m}}
\le
\sigma^{\,n+m}$.
Hence the radius-normalized two-point function decays exponentially. 
By contrast, when the two disks are tangent ($\sigma=1$), the two-point function becomes unbounded.
Generalizing the geometric configuration \eqref{intro_geom_sep}, we consider the operad of holomorphic disk embeddings
\begin{align}
\CEc(n)
=
\left\{
\sqcup_{i=1}^n \phi_i:\sqcup_n \bD \to \bD \mid
\text{injective holomorphic maps}
\right\}.
\label{intro_hol_emb}
\end{align}
For $B_r(\zeta)\subset\bD$, the holomorphic map
\begin{align}
B_{\zeta,r}:\bD \to \bD,
\qquad
B_{\zeta,r}(z)=rz+\zeta
\label{intro_trans}
\end{align}
defines an element $(B_{\zeta_1,r},B_{\zeta_2,s})\in\CEc(2)$ corresponding to the configuration \eqref{intro_geom_sep}. 
Note that the naive definition $\CEc$ also includes configurations with $\sigma=1$, where the (normalized) two-point function is unbounded.

In dimension $d \geq 3$, 
%these observations were extended to $d\ge3$.
a necessary and sufficient condition for the boundedness of the two-point function associated with conformal open embeddings of disks was established \cite[Theorem 2.31]{MLeft}, leading to a refined operad $\CEd$.
In dimension $d=2$, the boundedness condition is more subtle, as we shall explain below.

\vspace{3mm}

\noindent
\begin{center}
{\bf 0.2. Bergman space, Hilbert--Schmidt conditions, and Teichm\"uller spaces}
\end{center}

%\textbf{The Bergman space, Hilbert--Schmidt conditions, and the universal Teichm\"uller space}

The Bergman space $A^2(\bD)$ is the Hilbert space of square-integrable holomorphic functions on the unit disk $\bD=\{z\in\C \mid |z|<1\}$,
\begin{align*}
A^2(\bD) = L^2\left(\bD,{\pi}^{-1}d^2x\right) \cap \mathrm{Hol}(\bD),
\end{align*}
which is a closed subspace of $L^2\left(\bD,{\pi}^{-1}d^2x\right)$ with orthonormal basis $\{\sqrt{n+1}z^n\}_{n \geq 0}$.
Let $\Sym\, A^2(\bD)=\bigoplus_{p=0}^\infty \Sym^p A^2(\bD)$ be the algebraic direct sum.
For each finite $N$, the space $\bigoplus_{p=0}^N \Sym^p A^2(\bD)$ is a Hilbert space, whereas $\Sym A^2(\bD)$ is not complete; thus it defines an object of the ind category of Hilbert spaces $\Ind\Hilb$ (see \cite[Appendix B]{MLeft}).

Since $\norm{h(-n-1)\va}=\sqrt{n+1}$ and $\norm{h(-1)^p\va}=\sqrt{p!}$, the linear map
\begin{align}
\Psi:
h(-k_1-1)\cdots h(-k_p-1)\va
\mapsto \sqrt{p!}\bs^p((k_1+1)z^{k_1} \otimes \cdots \otimes (k_p+1)z^{k_p})
\label{intro_Psi}
\end{align}
defines an isometry from the affine Heisenberg VOA $M(0)$ to $\Sym A^2(\bD)$ (see Lemma \ref{lem_Psi}).
Here $\bs^p$ denotes the symmetrization operator.

In \cite{Mfactorization}, using the factorization algebra associated with the (conformal) Laplacian, we constructed an algebra over the naive operad $\CEc$ in the category of $\C$-vector spaces.
As discussed above, configurations in $\CEc$ do not in general define bounded operators.
We therefore introduce a suboperad $\CEh\subset \CEc$ on which the corresponding actions are bounded.
First, the submonoid $\CEh(1) \subset \CEc(1)$ is defined by
\begin{align}
\CEh(1) = \{\phi \in \CEc(1) \mid F_\phi(z,w) \in L^2(\bD \times \bD) \}.
\label{intro_HS}
\end{align}
Here $F_\phi(z,w)$ is the holomorphic function on $\bD \times \bD$ given by
\begin{align}
F_\phi(z,w) = \frac{\phi'(z)\phi'(w)}{(\phi(z)-\phi(w))^2} -\frac{1}{(z-w)^2}.
\label{intro_F_def}
\end{align}
The function $F_\phi(z,w)$ is a cocycle corresponding to the central charge: it describes how the two-dimensional Green's function $-(2\pi)^{-1}\log |z-w|$ changes under a local conformal transformation $\phi$, and it arises from the factorization algebra associated with the conformal Laplacian in \cite{Mfactorization}.
For $n \geq 2$, we set
\begin{align*}
\CEh(n) = \left\{(\phi_1,\dots,\phi_n)\in \CEh(1)^n\mid \phi_i(\bD)\cap \phi_j(\bD)=\emptyset \text{ and }G_{\phi_i,\phi_j}(z,w) \in L^2(\bD \times \bD)  \right\}.
\end{align*}
Here $G_{\phi_1,\phi_2}(z,w)$ is also a holomorphic function on $\bD\times \bD$,
\begin{align}
G_{\phi_1,\phi_2}(z,w) =  \frac{\phi_1'(z)\phi_2'(w)}{(\phi_1(z)-\phi_2(w))^2},
\label{intro_G}
\end{align}
introduced via factorization algebras in \cite{Mfactorization}.
Then $\CEh$ is a suboperad of $\CEc$ (Proposition~\ref{prop_monoid_HS}).
Using the contraction defined by integrating the cocycle \eqref{intro_F_def}
together with the pushforward along $\phi$, we construct a monoid homomorphism (Theorem \ref{thm_semigroup_quantum})
\begin{align}
\rho_1:\CEh(1) \rightarrow \End\,\left(\Sym\, A^2(\bD)\right),
\label{intro_monoid}
\end{align}
%and boundedness is ensured by $F_\phi(z,w) \in L^2(\bD \times \bD)$.
Here, for boundedness, it is essential that $F_\phi(z,w) \in L^2(\bD \times \bD)$.
Higher operations
\begin{align*}
\rho_n:\CEh(n) \rightarrow \Hom\,\left(\Sym\, A^2(\bD)^{\otimes n}, \Sym\, A^2(\bD)\right)
\end{align*}
are defined similarly by $F_{\phi_i}(z,w)$ and $G_{\phi_i,\phi_j}(z,w)$ in $L^2(\bD\times \bD)$.
%In this way we obtain the following :
\begin{mainthm}[Theorem~\ref{thm_CF}, Theorem~\ref{thm_CF_algebra}, and Corollary~\ref{cor_simple}]\label{thm_Bergman}
The ind-Hilbert space $\Sym A^2(\bD)$ inherits the structure of a $\CEh$-algebra.
The restriction on the suboperad $\CE\subset \CEh$ yields a symmetric monoidal functor $\Disk \rightarrow \Ind \Hilb$ satisfying \cite[Definition 1.32]{MLeft}, and its left Kan extension along $\Disk \rightarrow \Embc$,
\begin{align}
\mathrm{Lan}_{\Sym A^2(\bD)}: \Embc \rightarrow \Ind \Hilb,
\label{intro_Lan}
\end{align}
is symmetric monoidal.
Moreover, $\Sym A^2(\bD)$ has no proper closed ideal as a $\CE$-algebra.
\end{mainthm}

The category $\Embc$ is the category of germs of compact Riemannian manifolds (possibly with boundary) in dimension $2$ with conformal open embeddings \cite{MLeft}.
The left Kan extension in the theorem,
\begin{align*}
\mathrm{H}_{\mathrm{CF}}((M,g), \Sym A^2(\bD)) = \mathrm{Lan}_{\Sym A^2(\bD)}(M,g),
\end{align*}
defines a metric-dependent invariant of manifolds with coefficients in $\Sym A^2(\bD)$.

Let $F_\phi(z,w) = \sum_{n,m \geq 0}d_{n,m}z^nw^m$ be the Taylor expansion of $F_\phi$. 
Then the infinite matrix on $\ell^2(\N)$ given by 
$\left(\frac{d_{n,m}}{\sqrt{(n+1)(m+1)}}\right)_{n,m \geq 0}$
is called the \emph{Grunsky operator}.
The condition $F_\phi(z,w) \in L^2(\bD \times \bD)$ in \eqref{intro_HS} is equivalent to the Grunsky operator being Hilbert--Schmidt,
\begin{align*}
\sum_{n,m \geq 0} \frac{1}{(n+1)(m+1)}|d_{n,m}|^2 <\infty.
\end{align*}
The Hilbert--Schmidt property of the Grunsky operator characterizes the {\it Weil--Petersson class} and plays a central role in the 
Hilbert manifold refinement of the universal Teichm\"uller space
introduced by Takhtajan and Teo 
\cite{TT}. It is also essential in the refinement of the bordism category of Riemann surfaces in \cite{RSS}.
In this way, the Hilbert-Schmidt condition provides a key analytic ingredient in the Weil-Petersson geometry of Teichm\"uller theory
(see \cite{SS}).

Factorization homology is closely related to topological field theory \cite{Lurie3,Sc,CS,AF2}.
We expect that by suitably defining a ``Hilbert'' category $\Embh$ of two-dimensional Riemannian manifolds that contains the operad $\CEh$ as a full monoidal subcategory, and by taking the left Kan extension, one can construct Segal's functorial CFT \cite{Segal}.

%
%
%$F_\phi(z,w) = \sum_{n,m \geq 0}d_{n,m}z^nw^m$ を$F_\phi$のテイラー展開とする。このとき$l^2(\N)$上の無限次元行列$\left(\frac{d_{n,m}}{\sqrt{(n+1)(m+1)}}\right)_{n,m \geq 0}$は Grunsky operator と呼ばれる。
%\eqref{intro_HS}の条件、$F_\phi(z,w) \in L^2(\bD \times \bD)$は Grunsky operator が Hilbert-Schmidt operator になること
%\begin{align*}
%\sum_{n,m \geq 0} \frac{1}{(n+1)(m+1)}|d_{n,m}|^2 <\infty
%\end{align*}
%と同値である。
%Grunsky operator の Hilbert-Schmidt 性は、Weil-Petersson class として特徴づけられ Takhtanjan-Teo の universal Teichm\"{u}ller 空間の Hilbert 多様体としての refinement において重要な役割を果たす\cite{TT} (see also \cite{Shen})。
%これは\cite{RSS}における リーマン面の bordism category の 関数解析的なrefinement でも重要であり、このようにHilbert-Schmidt 性 (Weil-Petersson class)は、Hilbert 多様体として refine された Teichm\"{u}ller 空間論の解析的な基礎をなす。
%
%因子化ホモロジーは位相的な場の量子論と密接に関係しているが\cite{Lurie3, Sc, CS, AF2}、
%%2次元リーマン多様体の germ のなす圏 $\Emb$のこうした refinement を考えることで、
%operad $\CEh$を full monoidal subcategory として持つ''Hilbert 的な''二次元のリーマン多様体の圏$\Embh$を適切に定義し、その左カン拡張\eqref{eq_intro_left}をすることで、Segal の functorial CFT \cite{Segal} を構成できるのではないかと期待している。
%
%
\vspace{3mm}

\noindent
\begin{center}
{\bf 0.3. Unitary vertex operator algebras and analytic conformal blocks}
\end{center}

Finally, we explain the relationship between $\Sym A^2(\bD)$ and the affine Heisenberg VOA.
Let $\zeta \in \bD$ and $1>r,s>0$ satisfy
\begin{align}
B_r(\zeta),\; B_s(0) \subset \bD
\quad \text{ and }\quad 
\overline{B_{r}(\zeta)} \cap \overline{B_{s}(0)} = \emptyset.
\end{align}
Then $(B_{\zeta,r},B_{0,s}) \in \CEh(2)$ (see \eqref{intro_trans}).
Under the isometry $\Psi:M(0) \rightarrow \Sym A^2(\bD)$, for any $a,b \in M(0)$ we prove the identity
\begin{align}
\m_{(B_{\zeta,r}, B_{0,s})}(\Psi(a),\Psi(b))
= \Psi\!\left(Y(r^{L(0)} a,z)s^{L(0)} b \big|_{z=\zeta}\right)
\label{intro_compare}
\end{align}
in $\Sym A^2(\bD)$ (Theorem~\ref{thm_identification}).
This may be viewed as a refinement, at the Hilbert space level, of a relationship between holomorphic factorization algebras and vertex operator algebras \cite{CG1} (see also \cite{Bru,Vic,Nis}).
Moreover, via the relationship between unitary full VOAs and axiomatic QFT developed in \cite{AMT}, it provides a bridge between factorization algebras and functional-analytic formulations of quantum field theory.
%また AMT における unitary full VOA と axiomatic QFT との関係を通じて、因子化代数と関数解析的な場の量子論の定式化を結びつける。
Many examples of unitary (full) VOAs
%vertex operator algebras 
are known, and we expect that, by considering the completion of vertex operators as in \eqref{intro_compare}, one can produce a wide range of examples of $\CEh$-algebras and metric-dependent invariants of two-dimensional Riemannian manifolds.

The monoid action \eqref{intro_monoid} constructed explicitly in this paper should be closely related, via conformal welding, to the Henriques-Tener monoid action \cite{HT1,HT2}, which is defined by exponentiating the Virasoro action. We leave this connection for future work.

Finally, we comment on the relationship with non-chiral CFT.
For a holomorphic map $\phi:\bD \rightarrow \bD$, the assignment
\begin{align*}
\J(\phi)(z) = \overline{\phi(\z)}
\end{align*}
defines an outer automorphism $\J$ of the operad $\CEh$ (Definition~\ref{def_conjugate}),
and hence one can define the $\CEh$-algebra obtained by twisting $\Sym A^2(\bD)$ by $\J$.
This corresponds to the anti-holomorphic affine Heisenberg vertex operator algebra $\overline{M(0)}$ \cite{M1}.
Moreover,
\begin{align*}
\Sym\, A^2(\bD) \otimes \Sym\, A^2(\bD)^{\J}
\end{align*}
naturally carries a $\CEh$-algebra structure (this is the full CFT arising naturally from the factorization algebra associated with the two-dimensional Laplacian; see \cite{Mfactorization}).

In this way, the notion of a $\CEh$-algebra does not rely on the holomorphic splitting coming from the exceptional isomorphism \eqref{intro_exception}, and the left Kan extension \eqref{intro_Lan} may be regarded as a real-analytic analogue of the local-to-global principle described by conformal blocks.

The organization of this paper is as follows.
In Section~\ref{sec_pre_CE}, we review from \cite{MLeft} the definition of the operad $\CEc$ and define the outer automorphism $\J$.
In Section~\ref{sec_pre_Bergman}, we recall the definition of the Bergman space and basic properties.
In Section~\ref{sec_pre_HS}, we define our suboperad $\CEh \subset \CEc$.
In Section~\ref{sec_monoid}, we construct the monoid homomorphism \eqref{intro_monoid}, and in Section~\ref{sec_wick} we define
a geometric Wick contraction.
Based on these, in Section~\ref{sec_construction} we endow $\Sym A^2(\bD)$ with a $\CEh$-algebra structure.
In Section~\ref{sec_vertex}, we review the affine Heisenberg vertex operator algebra $M(0)$ and its unitary structure, and in Section~\ref{sec_dictionary} we prove the correspondence \eqref{intro_compare} between $M(0)$ and $\Sym A^2(\bD)$.

\vspace{3mm}

\noindent
\begin{center}
{\bf Acknowledgements}
\end{center}

I express my gratitude to Yoh Tanimoto and Maria Stella Adamo for valuable discussions on 
unitary full vertex operator algebras.
%,
%to Mayuko Yamashita for discussions on topological quantum field theory,
%to Masahito Yamazaki and Slava Rychkov for discussions on physical aspects,
%and to Hiro-Lee Tanaka, Takumi Maegawa, and Vladimir Sosnilo for discussions related to category theory.
%I also wish to express my gratitude to Yuji Tachikawa, 
%Masaki Natori, Naruki Masuda, Tomohiro Asano, Masahiro Futaki,
%Hiroshi Ooguri, Tomoyuki Arakawa, Shintaro Yanagida, Atsushi Matsuo, Hiroshi Yamauchi, Yoshihisa Saito, Toshiro Kuwabara for valuable discussions and comments. 
This work is supported by Grant-in Aid for Early-Career Scientists (24K16911).

\vspace{3mm}
\begin{center}
\textbf{\large Notations}
\end{center}

\vspace{3mm}
We will use the following notations:
\begin{itemize}
%\item[$\text{[}n\text{]}$:] the finite set $\{1,2,\dots,n\}$ for $n \geq 1$.
\item[$\bD$:]$=\{z\in\C \mid |z|<1 \}$, the unit disk in $\C$.
\item[$\CEc$:] an operad of holomorphic disk embeddings, \S \ref{sec_pre_CE}.
\item[$\fG$:]$=\{e^{i\theta}\frac{z-a}{1-\bar{a}z}\}_{\theta\in\R,a \in \bD}$ the maximal subgroup of $\CEc(1)$, \S \ref{sec_pre_CE}.
\item[$\CE$:] a suboperad of $\CEc$ which extend to a neighborhood, \S \ref{sec_pre_HS}.
\item[$\CEh$:] a suboperad of $\CEc$ satisfying a square-integrability condition, \S \ref{sec_pre_HS}.
%\item[$\CEh$:] 二乗可積分性を持つ正則 disk 埋め込みのなす operad
\item[$\Vect$:] the category of $\C$-vector spaces.
\item[$\Hilb$:] the category of Hilbert spaces with bounded linear operators.
\item[$\Ind \Hilb$:] the category of ind-objects in $\Hilb$.
\item[$\otimes$:] the algebraic tensor product of vector spaces.
\item[$\hotimes$:] the tensor product of Hilbert spaces.
\item[$A^2(\bD)$:] the Bergman space, \S \ref{sec_pre_Bergman}.
\item[$d\mu = \frac{1}{\pi} dx_1dx_2$:] the normalized Lebesgue measure, \S \ref{sec_pre_Bergman}.
\item[$E_a(z)$:]$ = \frac{1}{(1-az)^2}$ for $a\in \bD$, \S \ref{sec_pre_Bergman}.
\end{itemize}

\vspace{5mm}

\section{Preliminary}\label{sec_pre}
%\ref{sec_pre_CE}章では\cite{MLeft}で導入した operad $\CEc$の定義を振り返る。
%\ref{sec_pre_Bergman}では二乗可積分な正則関数のなすヒルベルト空間であるBergman空間とその再生核ヒルベルト空間としての性質を振り返る。
%\ref{sec_pre_HS}では、$\CEc$の suboperadである $\CEh$を定義し、\cite{MLeft}で導入した別のsuboperad $\CE \subset \CEc$との関係を調べる。
In Section~\ref{sec_pre_CE}, we recall the definition of the operad $\CEc$ introduced in \cite{MLeft}.  
In Section~\ref{sec_pre_Bergman}, we review the Bergman space, the Hilbert space of square-integrable holomorphic functions, and its structure as a reproducing kernel Hilbert space.  
In Section~\ref{sec_pre_HS}, we define the suboperad $\CEh \subset \CEc$ and examine its relationship with another suboperad $\CE \subset \CEc$ introduced in \cite{MLeft}.

\subsection{Conformally flat 2-disk operad and its complex conjugate}\label{sec_pre_CE}

In this section, following \cite[Section 1]{MLeft}, we review the definition of the conformally flat $2$-disk operad $\CEc$ and show that the complex conjugation defines an automorphism of the operad.
% We also recall some elementary results on M\"{o}bius transformations.

Let $\Emb$ be the category whose objects are two-dimensional oriented Riemannian manifolds without boundary and whose morphisms are orientation-preserving conformal open embeddings.
Let $g_\std=dx_1^2+dx_2^2$ be the Riemannian metric on $\R^2$.
Then the unit disk
\begin{align*}
\bD=\{z\in \C \mid |z|<1\}
\end{align*}
equipped with $g_\std$ is an object of $\Emb$.
We simply write $\bD$ for the pair $(\bD,g_\std)$.
Let $\Diskc$ be the full monoidal subcategory of $\Emb$ generated by $\bD$.
The objects of $\Diskc$ consist of finite disjoint unions of disks, $\sqcup_n \bD$ ($n \geq 0$).
We define
\begin{align*}
\CEc(n) = \Hom_{\Emb}(\sqcup_n \bD,\bD)\qquad (n \geq 0)
\end{align*}
to be the operad of conformal embeddings of disks.
For a symmetric monoidal category $\cC$, there is a one-to-one correspondence between algebras over this operad, namely $\CEc$-algebras in $\cC$, and symmetric monoidal functors $\Diskc \rightarrow \cC$ (see \cite[Proposition 1.11]{MLeft}).
Since, in two dimensions, orientation-preserving conformal maps are equivalently holomorphic maps with nonvanishing derivative, we have
\begin{align*}
\CEc(1) = \{f: \bD \rightarrow \bD \mid \text{$f$ is an injective holomorphic map} \},
\end{align*}
and $\CEc(1)$ becomes an (infinite-dimensional) monoid under composition.
Moreover, $\CEc(n)$ is given by conformal open embeddings with pairwise disjoint images,
\begin{align*}
\CEc(n) = \{(f_1,\dots,f_n) \in \CEc(1)^n \mid f_i(\bD) \cap f_j(\bD) =\emptyset, \text{ for } i\neq j\}.
\end{align*}
Here $\CEc(0)$ corresponds to the unique map from the empty set, namely $*=(\emptyset \rightarrow {\bD})$.
By Schwarz's lemma, the maximal subgroup $\CEc(1)^\times = \{\text{$f$ is invertible in }\CEc(1)\}$
coincides with the subgroup of M\"{o}bius transformations
\begin{align*}
\mathrm{PSU}(1,1) = \left\{z \mapsto e^{i\theta}\frac{z-\al}{1-\bar{\al}z}\right\}_{\theta\in\R,|\al|<1},
\end{align*}
see \cite[Proposition 1.22]{MLeft}, which we denote by $\fG$.

%この章では\cite[Section 1]{MLeft}に従い conformally flat 2-disk operad $\CEc$の定義を振り返り、複素共役の作用が operad の自己同型を定めることを見る。また一次分数変換に関する初等的な結果を振り返る。
%
%$\Emb$を対象が二次元の向きづけられたリーマン多様体で射が向きを保つ共形開埋め込みのなす圏とする。
%$g_\std=dx_1^2+dx_2^2$を$\R^2$上のリーマン計量する。このとき the unit disk
%\begin{align*}
%\bD=\{z\in \C \mid |z|<1\}
%\end{align*}
%equipped with $g_\std$ is an object of $\Emb$. この論文を通じてdisk 上のリーマン計量は$g_\std$を考え、組$(\bD,g_\std)$を単に$\bD$とかく。
%$\Diskc$を$\bD$から生成される $\Emb$の full monoidal subcategory とする。$\Diskc$の対象は 有限個のdisk の disjoint unions $\sqcup_n \bD$ ($n \geq 0$)からなる。
%また
%\begin{align*}
%\CEc(n) = \Hom_{\Emb}(\sqcup_n \bD,\bD)\qquad (n \geq 0)
%\end{align*}
%を disk の共形埋め込みのなすoperad とする。
%対称モノイダル圏$\cC$に対して、operad の代数、$\CEc$-algebra in $\cC$ と、対称monoidal 関手$\Diskc \rightarrow \cC$には一対一対応がある (see \cite[Proposition ]{MLeft})。
%二次元において向きを保つ共形写像と微分がゼロでない正則関数を考えることは同値であるから
%\begin{align*}
%\CEc(1) = \{f: \bD \rightarrow \bD \mid \text{$f$ is an injective holomorphic map} \},
%\end{align*}
%であり、$\CEc(1)$は合成により(無限次元の) monoid になる。また$\CEc(n)$は像が disjoint な open embeddings
%\begin{align*}
%\CEc(n) = \{(f_1,\dots,f_n) \in \CEc(1)^n \mid f_i(\bD) \cap f_j(\bD) =\emptyset, \text{ for } i\neq j\}
%\end{align*}
%によって与えられる。ここで$\CE(0)$は 空集合からの唯一の射$*=(\emptyset \rightarrow {\bD})$ に対応する。
%またシュワルツの補題から最大部分群$\CEc(1)^\times = \{f \in \CEc(1) \mid \text{$f$ is invertible in }\CEc(1)\}$
%は、一次分数変換
%\begin{align*}
%\mathrm{PSU}(1,1) = \left\{z \mapsto e^{i\theta}\frac{z-\al}{1-\bar{\al}z}\right\}_{\theta\in\R,|\al|<1},
%\end{align*}
%と一致する \cite[?]{MLeft}. 

The operad structure on $\CEc$, 
$\circ_i:\CEc(n) \times \CEc(m) \rightarrow \CEc(n+m-1)$,
is given by the operation of composing disk embeddings in the $i$-th disk ($i \in \{1,\dots,n\}$),
\begin{align*}
(f_1,\dots,f_n) \circ_i (g_1,\dots,g_m)
= (f_1,\dots,f_{i-1},f_ig_1,\dots,f_ig_m,f_{i+1},\dots,f_n).
\end{align*}
Here the identity map $(\id_{\bD}:\bD \rightarrow \bD)\in \CEc(1)$ serves as the unit, and the composition with $\CEc(0)$,
\begin{align*}
(f_1,\dots,f_n) \circ_i * 
= (f_1,\dots,f_{i-1},f_{i+1},\dots,f_n),
\end{align*}
is the operation obtained by forgetting the $i$-th disk embedding, regarded as an element of $\CEc(n-1)$.
%The pullback of the Riemannian metric by a holomorphic function $f$ satisfies
%\begin{align*}
%f^*(g_\std) = |f'|^2 g_\std.
%\end{align*}
%Hence the conformal factor of a holomorphic map is given by the absolute value $|f'|$ of its holomorphic derivative.
%
%$\CEc$上の operad 構造$\circ_i:\CEc(n) \times \CEc(m) \rightarrow \CEc(n+m-1)$は、$i$番目の disk に対して、disk の埋め込みを合成する operation ($i \in \{1,\dots,n\}$)
%\begin{align*}
%(f_1,\dots,f_n) \circ_i (g_1,\dots,g_m) = (f_1,\dots,f_{i-1},f_ig_1,\dots,f_ig_m,f_{i+1},\dots,f_n)
%\end{align*}
%によって与えられる。ここで恒等写像$(\id_{\bD}:\bD \rightarrow \bD)\in \CE(1)$が unit であり、$\CE(0)$との合成
%\begin{align*}
%(f_1,\dots,f_n) \circ_i * = (f_1,\dots,f_{i-1},f_{i+1},\dots,f_n)
%\end{align*}
%は$i$番目の disk embedding を忘れて$\CE(n-1)$と思う operation である。
%
%正則関数$f$によるリーマン計量の引き戻しは
%\begin{align*}
%f^*(g_\std) = |f'|^2 g_\std
%\end{align*}
%を満たす。よって正則写像の共形因子は正則微分の絶対値$|f'|$で与えられる。
Let $\phi \in \CEc(1)$. Set 
\begin{align*}
\J(\phi)(z) = \overline{\phi(\z)}
\end{align*}
for any $z\in \bD$. Then, $\J:\CEc(1) \rightarrow \CEc(1)$ is a monoid homomorphism.
Define a sequence of maps
\begin{align*}
\J:\CEc(n) \rightarrow \CEc(n),\quad (\phi_1,\dots,\phi_n) \mapsto \J(\phi_1,\dots,\phi_n)=(\J(\phi_1),\dots,\J({\phi}_n)).
\end{align*}
Then, $\J$ is an automorphism of the operad $\CEc$. 
\begin{dfn}\label{def_conjugate}
For any $\CEc$-algebra $A$, given by
\begin{align*}
\rho_n:\CEc(n) \rightarrow \mathrm{Hom}(A^{\otimes n},A),\qquad (n \geq 0)
\end{align*}
a \textbf{ complex conjugate} of $A$ is a $\CEc$-algebra defined on the same object $A$ with the multiplications $(\rho_n \circ \J)_{n \geq 0}$.
We denote it by $A^\J$.
\end{dfn}

\begin{rem}
The assignment $\phi \mapsto \J(\phi)$ can be defined more generally for a conformal diffeomorphism $g:\bD \rightarrow \bD$, not necessarily orientation-preserving. If $g \in \mathrm{PSU}(1,1)$ the corresponding twist is given by an inner automorphism, and it is easy to see that $A^g$ is isomorphic to $A$. Hence, by Schwarz's lemma, the above twist provides, essentially, the only nontrivial twist.
\end{rem}

The biholomorphic automorphism group of the disk, $\CEc(1)^\times \cong \mathrm{PSU}(1,1)$, satisfies the following important identity:
%
%\begin{rem}
%$\phi \mapsto \J(\phi)$はより一般に$g:\bD \rightarrow \bD$なる向きを保つとは限らない conformal な diffeomorphism に対して定義することができる。ただし$g \in \mathrm{PSU}(1,1)$となる場合は、inner automorphism による twist であり、$A^g$が$A$と同型であることは簡単に分かる。よってSchwartz の補題から上記の twist が本質的には唯一の非自明な twist を与える。
%\end{rem}
%
%disk の双正則 automorphism $\CE(1)^\times \cong \mathrm{PSU}(1,1)$ は次の重要な恒等式を満たす:
\begin{prop}\label{prop_conf_identity}
For any $\phi(z) = e^{i\theta}\frac{z-\al}{1-\bar{\al}z} \in \mathrm{PSU}(1,1)$ ($\theta \in\R$, $|\al|<1$),
\begin{align*}
 \frac{{\phi'(z)}\overline{\phi'(w)}}{(1- \phi(z)\overline{\phi(w)})^2}=\frac{1}{(1-z\bar{w})^2}
\end{align*}
holds for any $z,w \in \bD$.
Moreover, $\J(\phi)(z) =  e^{-i\theta}\frac{z-\bar{\al}}{1-{\al}z} \in \mathrm{PSU}(1,1)$.
%\begin{align*}
%{|1-z\bar{\al}|^2}\left(1-\left|\frac{z-\al}{1-z\bar{\al}}\right|^2\right) 
%%&= (1-z\bar{\al}-\z \al + |z|^2|\al|^2 )- (|z|^2+|\al|^2 -z\bar{\al}-\z\al)\\
%&=(1-|z|^2)(1-|\al|^2)
%\end{align*}
%holds.
%Moreover, for any $\phi \in \fG$,
%\begin{align*}
% (1-2(\phi(x),\phi(y))+|\phi(x)|^2|\phi(y)|^2) =\Om_\phi(x)\Om_{\phi}(y) (1-2(x,y)+|x|^2|y|^2).
%\end{align*}
\end{prop}
\begin{proof}
Since $\phi'(z) =e^{i\theta}\frac{(1-|\al|^2)}{(1-\bar{\al}z)^2} $, we have
\begin{align*}
\frac{{\phi'(z)}\overline{\phi'(w)}}{(1- \phi(z)\overline{\phi(w)})^2}&=\frac{(1-|\al|^2)^2}{((1-\bar{\al}z)(1-\al\bar{w})-(z-\al)(\bar{w}-\bar{\al}))^2}=\frac{1}{(1-z\bar{w})^2}.
\end{align*}
\end{proof}

\subsection{Bergman spaces and their tensor products}\label{sec_pre_Bergman}
In this section, we recall some basic results on the Bergman space.
The material in this section is standard; for a more detailed account of the Bergman space, see for example \cite{HKZ}.

%この章では Bergman 空間に関する基本的な結果を振り返る.
%この章の内容は標準的であり、Bergman 空間についてのより詳しい解説はたとえば\cite{}を参照されたい。
%Let $\bD$ be the unit disk $\bD = \{z \in \C \mid |z|<1\}$
%and 
For $p \geq 1$, let $\mathrm{Hol}(\bD^p)$ be the space of holomorphic functions on $\bD^p \subset \C^p$,
and $d\mu = \frac{1}{\pi^p}d^2z_1 \dots d^2 z_p = \frac{1}{\pi^p}dx_1\cdots dx_{2p}$ a normalized Lebesgue measure on $\bD^p$.
Set
\begin{align*}
A^2(\bD^p)&= L^2(\bD^p,d\mu) \cap \mathrm{Hol}(\bD^p),\\
&=\left\{f\in \mathrm{Hol}(\bD^p) \mid \int_{\bD^p} |f(z_1,\dots,z_p)|^2 d\mu <\infty \right\},
\end{align*}
called a \textbf{Bergman space}, consisting of square-integrable holomorphic functions on $\bD^p$.
It is well-known that the Bergman space is a closed subspace of $L^2(\bD^p,d\mu)$, and thus, a Hilbert space.
Denote by $(-,-)_B$ the inner product on $A^2(\bD^p)$.
Assume $p=1$.
Since
\begin{align*}
(z^n,z^m)_B &= \int_{\bD}\overline{z^n}z^m d\mu = \int_{\bD}r^{n+m}e^{(-n+m)i\theta}r \frac{d\theta}{\pi} dr= \frac{1}{n+1}\delta_{n,m},
\end{align*}
$\{e_n(z)= {\sqrt{n+1}}z^n\}_{n \geq 0}$ is an orthonormal basis of $A^2(\bD)$.
For any $a \in \bD$, 
\begin{align*}
E_a(z)&=\sum_{n \geq 0}\overline{e_n(\bar{a})}e_n(z)=
\sum_{n \geq 0}(n+1){a}^nz^n=\frac{1}{(1- {a}z)^2} \in H(\bD)
\end{align*}
is square-integrable and $E_a(z) \in A^2(\bD)$.
Moreover, for any $f= \sum_{n\geq 0} c_n e_n(z) \in A^2(\bD)$,
\begin{align}
(E_a(z),f)_B &= \int_\bD \overline{\frac{1}{(1- {a}z)^2}}f(z) d\mu= \sum_{n \geq 0} c_n e_n(\bar{a})=f(\bar{a}).\label{eq_evaluation}
\end{align}
Hence, taking the inner product with $E_a$ corresponds to evaluation at $\bar{a}$.
%よって$E_a$との内積は$\bar{a}$を代入する写像に対応する。
By \eqref{eq_evaluation}, we have:
\begin{prop}
\label{prop_dense}
The subspace spanned by $\{E_a\}_{a\in \bD}\subset A^2(\bD)$ is dense in $A^2(\bD)$, and 
\begin{align*}
(E_a(z),E_b(z))_B = \frac{1}{(1- \bar{a}b)^2},
\end{align*}
holds for any $a,b\in\bD$.
\end{prop}
\begin{lem}\label{lem_kernel_derivative}
For any $n \geq 0$ and $a\in \bD$, $\pa_a^n E_a(z) \in A^2(\bD)$. Moreover, for any $f\in A^2(\bD)$,
\begin{align*}
(\pa_a^n E_a(z), f)_B = (\pa_z^n f)(\ba).
\end{align*}
\end{lem}
\begin{proof}
Since the pole of $E_a(z)$ is located at $z = a^{-1} \in \C \setminus \overline{\bD}$, $\pa_a^n E_a(z) \in A^2(\bD)$. Since
\begin{align*}
\pa_a^n E_a(z)&= \frac{(n+1)!z^n}{(1-az)^{2+n}}
=\sum_{k \geq 0}\frac{(n+k+1)!}{k!} a^k z^{n+k},
\end{align*}
for any $f(z)=\sum_{k\geq 0}c_k z^k$ we have,
\begin{align*}
(\pa_a^n E_a(z),f) 
= \sum_{k \geq 0} \frac{(n+k)!}{k!}c_{n+k} \bar{a}^k 
= (\pa_z^n f)(\ba).
\end{align*}
%$E_a(z)$の極は$z= a^{-1} \in \C \setminus \overline{\bD}$にあることから、$\pa_a^n E_a(z) \in A^2(\bD)$は明らかである。$f(z)=\sum_{k\geq 0}c_k z^k$に対して
%\begin{align*}
%\pa_a^n E_a(z)&= \frac{(n+1)!z^n}{(1-az)^{2+n}}
%=\sum_{k \geq 0}\frac{(n+k+1)!}{k!} a^k z^{n+k}
%\end{align*}
%より
%\begin{align*}
%(\pa_a^n E_a(z),f) = \sum_{k \geq 0} \frac{(n+k)!}{k!}c_{n+k} \bar{a}^k = (\pa_z^n f)(\ba).
%\end{align*}
\end{proof}

For any $p \geq 1$, let $A^2(\bD)^{\hotimes p}$ be the tensor product of Hilbert spaces,
which is the Hilbert space completion of the algebraic tensor product.
Let $\mathrm{Sym}^p A^2(\bD)$ is the closed subspace of $A^2(\bD)^{\hotimes p}$
defined by the image of the (completed) orthogonal projection
\begin{align}
\bs^p:A^2(\bD)^{\hotimes p} \rightarrow A^2(\bD)^{\hotimes p},\quad (f_1\otimes \dots  \otimes f_n)\mapsto
\frac{1}{p!}\sum_{\si \in S_p} f_{\si(1)}\otimes \dots  \otimes f_{\si(p)}.\label{eq_symetrizer}
\end{align}
Set
\begin{align*}
H^k&=\bigoplus_{p = 0}^k \mathrm{Sym}^p A^2(\bD),\\
\Sym(A^2(\bD))& = \bigoplus_{p \geq 0} \mathrm{Sym}^p A^2(\bD),
\end{align*}
the {algebraic} direct sum. 
Here $\Sym^0 A^2(\bD)$ is a one-dimensional Hilbert space, and throughout this paper we denote its normalized basis vector by $\va$.
%ここで$\Sym^0 A^2(\bD)$は一次元のヒルベルト空間であり、この論文を通じてその正規直交基底を$\va$とかく。
Note that for any $k\geq 0$, $H^k$ is a Hilbert space, while $\Sym(A^2(\bD))$ is not complete.
%The purpose of this paper is to define a conformally flat $d$-disk algebra structure on $\Sym(A^2(\bD))$.
%この論文の目的は$\Sym(A^2(\bD))$上に、共形平坦d-disk 代数の構造を定義することである。
Let $f_1,\dots,f_p \in A^2(\bD)$. Then, by
\begin{align}
f_1 (z_1)f_2(z_2) \cdots  f_p(z_p) \in \mathrm{Hol}(\bD^p) \cap L^2(\bD^p),
\label{eq_Bergman_tensor}
\end{align}
we have a linear map $A^2(\bD)^{\otimes p} \rightarrow A^2(\bD^p)$, which is isometric. Hence, we have:
\begin{lem}\label{lem_tensor_isomorphism}
For any $p \geq 1$, $A^2(\bD)^{\hotimes p} \rightarrow A^2(\bD^p)$ defined by \eqref{eq_Bergman_tensor} is an isometric isomorphism of Hilbert spaces. Moreover, the image of $\Sym^p(A^2(\bD))$ is given by
\begin{align*}
\{ F(z_1,\dots,z_n) \in A^2(\bD^p) \mid F(z_1,\dots,z_n)=F(z_{\si(1)},\dots,z_{\si(n)}) \text{ for any }\si \in S_n  \}.
\end{align*}
\end{lem}

\subsection{Harmonic cocycles  and Hilbert-Schmidt conditions}\label{sec_pre_HS}
Motivated by the study of factorization algebras associated with the conformal Laplacian \cite[Section 2.3]{Mfactorization}, we consider, for $\phi \in \CEc(1)$,
\begin{align}
F_\phi(z,w) = \frac{\phi'(z)\phi'(w)}{(\phi(z)-\phi(w))^2} - \frac{1}{(z-w)^2}
\end{align}
and, for $(\phi_1,\phi_2) \in \CEc(2)$,
\begin{align*}
G_{\phi_1,\phi_2}(z,w) = \frac{\phi_1'(z)\phi_2'(w)}{(\phi_1(z)-\phi_2(w))^2}.
\end{align*}
Then $F_\phi(z,w)$ extends holomorphically across the diagonal $z=w$, and defines a holomorphic function on $\bD \times \bD$ \cite[Proposition 2.10]{Mfactorization}.
Moreover, since $\phi_1(\bD) \cap \phi_2(\bD) = \emptyset$, the function $G_{\phi_1,\phi_2}(z,w)$ also defines a holomorphic function on $\bD \times \bD$.
A straightforward computation shows the following:

\begin{prop}\label{prop_FG_equality}
For any $f_1,f_2 \in \CEc(1)$ and $(g_1,g_2) \in \CEc(2)$,
\begin{align}
F_{f_1 \circ f_2}(z,w) &= F_{f_1}(f_2(z),f_2(w))f_2'(z)f_2'(w) + F_{f_2}(z,w)\label{eq_F_cocycle}\\
G_{f\circ g_1,f\circ g_2}(z,w) &= F_{f}(g_1(z),g_2(w))g_1'(z)g_2'(w) + G_{g_1,g_2}(z,w)\label{eq_G_cocycle1}\\
G_{g_1 \circ f_1,g_2\circ f_2}(z,w)&=G_{g_1,g_2}(f_1(z),f_2(w))f_1'(z)f_2'(w).
\label{eq_G_cocycle2}
\end{align}
\end{prop}

%As observed in \cite{MLeft} for $d \geq 3$, the naive operad $\CEc$ does not allow one to define a conformally flat $d$-disk algebra structure on $\Sym(A^2(\bD))$.
%\cite{MLeft}で$d \geq 3$において観測されたように、naive なoperad$\CEc$の元に対して、対応する operad の積は$\Sym(A^2(\bD))$上で well-defined 二ならない。
As observed in \cite{MLeft} for $d \geq 3$, for elements of the naive operad $\CEc$, the corresponding operadic multiplication is not well-defined on $\Sym(A^2(\bD))$.
As we will see later, in order to define an algebra structure on $\Sym(A^2(\bD))$, it is essential that these functions be square-integrable, 
more precisely,
\begin{align}
F_\phi(z,w), G_{\phi_1,\phi_2}(z,w) \in A^2(\bD \times \bD).
\label{eq_square}
\end{align}
In this section, we introduce a suboperad $\CEh$ of $\CEc$ satisfying \eqref{eq_square}.
We also examine its relationship with the suboperad $\CE$ of $\CEc$ introduced in \cite{MLeft}.

%
%
%共形ラプラシアンに付随する因子化代数の研究 \cite{Mfactorization} に動機を受けて, 我々は$\phi \in \CEc(1)$に対して
%\begin{align}
%F_\phi(z,w) = \frac{\phi'(z)\phi'(w)}{(\phi(z)-\phi(w))^2} - \frac{1}{(z-w)^2}
%\end{align}
%and $\phi_1,\phi_2 \in \CEc(2)$に対して
%\begin{align*}
%G_{\phi_1,\phi_2}(z,w) = \frac{\phi_1'(z)\phi_2'(w)}{(\phi_1(z)-\phi_2(w))^2}
%\end{align*}
%を考える。このとき$F_\phi(z,w)$は$z=w$上でも正則になり、$\bD \times \bD$上の正則関数を定めることが分かる \cite[]{Mfactorization}. 
%また$G_{\phi_1,\phi_2}(z,w)$も$\phi_1(\bD) \cap \phi_2(\bD) = \emptyset$であることから、$\bD \times \bD$上の正則関数を定める。簡単な計算により次が成り立つ:
%\begin{prop}\label{prop_FG_equality}
%For any $f_1,f_2 \in \CEc(1)$ and $(g_1,g_2) \in \CEc(2)$,
%\begin{align}
%F_{f_1 \circ f_2}(z,w) &= F_{f_1}(f_2(z),f_2(w))f_2'(z)f_2'(w) + F_{f_2}(z,w)\label{eq_F_cocycle}\\
%G_{f\circ g_1,f\circ g_2}(z,w) &= F_{f}(g_1(z),g_2(w))g_1'(z)g_2'(w) + G_{g_1,g_2}(z,w)\label{eq_G_cocycle1}\\
%G_{g_1 \circ f_1,g_2\circ f_2}(z,w)&=G_{g_1,g_2}(f_1(z),f_2(w))f_1'(z)f_2'(w).
%\label{eq_G_cocycle2}
%\end{align}
%\end{prop}

%後で見るように、$\Sym(A^2(\bD))$上に代数構造を定義するには、これらの関数が square-integrable であること, すなわち
%\begin{align}
%F_\phi(z,w), G_{\phi_1,\phi_2}(z,w) \in A^2(\bD \times \bD)\label{eq_square}
%\end{align}
%が本質的である。この章では \eqref{eq_square}を満たす$\CEc$-operad の suboperad $\CEh$を導入する。
%また、その\cite{}で導入された $\CEc$の suboperad $\CE$ との関係を調べる。

Let $\phi \in \CEc(1)$. Then, $F_\phi(z,w)$ is a holomorphic function on $\bD \times \bD$, but not always extends onto a continuous function on $\overline{\bD}\times \overline{\bD}$ and $F_\phi(z,w) \notin A^2(\bD \times \bD)$. In fact, set
\begin{align*}
F_\phi(z,w) = \sum_{n,m \geq 0}f_{n,m}e_n(z)e_m(w).
\end{align*}
Then, the condition $F_\phi(z,w) \in L^2(\bD\times \bD)$ is equivalent to $\sum_{n,m \geq 0} |f_{n,m}|^2 <\infty$.
We set
\begin{align*}
\CEh(1) = \{\phi \in \CEc(1) \mid F_\phi(z,w) \in L^2(\bD \times \bD) \}.
\end{align*}
The following proposition is well known (see, for example, \cite[Proposition 2.10]{Mfactorization}).
%このときに$F_\phi(z,w) \in L^2(\bD\times \bD)$であることと、$\sum_{n,m \geq 0} |d_{n,m}|^2 <\infty$は同値である。
%ここで
%\begin{align*}
%\CEh(1) = \{\phi \in \CEc(1) \mid F_\phi(z,w) \in L^2(\bD \times \bD) \}
%\end{align*}
%とおく。次の命題はよく知られている (see for example \cite[]{Mfactorization}).
\begin{prop}\label{prop_cocycle_vanish}
Let $\phi \in \CEc(1)$. Then, $F_\phi(z,w)=0$ if and only if $\phi$ is a linear fractional transformation, that is, there is $\begin{pmatrix}
   a & b \\
   c & d
\end{pmatrix} \in \mathrm{SL}_2\C$ such that $\phi(z) = \frac{az+b}{cz+d}$.
\end{prop}
%\begin{proof}
%For $\phi(z) = \frac{az+b}{cz+d}$, it is easy to see that
%\begin{align*}
%\left(\frac{az+b}{cz+d}-\frac{aw+b}{cw+d}\right) = \frac{1}{(cz+d)(cw+d)}(z-w).
%\end{align*}
%Hence, $F_\phi=0$. Conversely, assume $F_\phi=0$.
%It is easy to see that $\lim_{w \to z} F_\phi(z,w)=\frac{1}{6}S_{\phi}$,
%where $S_\phi$ is the Schwarzian derivative of $\phi$. Hence, $F_\phi=0$ implies that $S_\phi=0$. Hence, $\phi$ is a linear fractional transformation.
%\end{proof}
By Proposition~\ref{prop_cocycle_vanish}, the set $\CEh(1)$ contains $\fG$; in particular, $\id_\bD \in \CEh(1)$.
%命題\ref{prop_cocycle_vanish}より$\CEh(1)$は$\mathrm{PSU}(1,1)$を含み、とくに$\id_\bD \in \CEh(1)$である。このとき次が分かる:
\begin{prop}\label{prop_monoid_HS}
For any $\phi_1, \phi_2 \in \CEh(1)$, $\norm{F_{\phi_1 \circ \phi_2}(z,w)}_{L^2} \leq \norm{F_{\phi_1}(z,w)}_{L^2}+\norm{F_{\phi_2}(z,w)}_{L^2}$.
In particular, $\CEh(1)$ is a submonoid of $\CEc(1)$.
\end{prop}
\begin{proof}
Let $\phi_1, \phi_2 \in \CEh(1)$. Then, by Proposition \ref{prop_FG_equality} \eqref{eq_F_cocycle}, we have
\begin{align*}
\norm{F_{\phi_1 \circ \phi_2}(z,w)}_{L^2} \leq \norm{F_{\phi_1}(\phi_2(z),\phi_2(w))\phi_2'(z)\phi_2'(w)}_{L^2}+\norm{F_{\phi_2}(z,w)}_{L^2}.
\end{align*}
Hence, the assertion follows from
\begin{align}
\begin{split}
&\int_{\bD \times \bD} |F_{\phi_1}(\phi_2(z),\phi_2(w))\phi_2'(z)\phi_2'(w)|^2 d\mu(z)d\mu(w)\\
 &=
\int_{\phi_2(\bD) \times \phi_2(\bD)} |F_{\phi_1}(z',w')|^2 d\mu(z')d\mu(w')
\leq \int_{\bD \times \bD} |F_{\phi_1}(z',w')|^2 d\mu(z')d\mu(w') = \norm{F_{\phi_1}(z,w)}_{L^2}^2.
\end{split}
\label{eq_cocycle_L2}
\end{align}
\end{proof}
%
%\begin{rem}
%行列$\{d_{nm}\}_{n,m \geq 0}$は geometric function theory において Grunsky matrix と呼ばれる重要な matrix であり、
%$\sum_{n,m} |d_{nm}|^2<\infty$であることと、行列が Hilbert-Schmidt operator であることは同値である。
%Takhtajan-Teo と Shen の独立した研究では、ある条件の下で $\phi$ に付随する Grunsky matrix が 
%Hilbert-Schmidt operator であることと、$\phi$が Weil-Petersson class であることと同値であることが示された (\cite[Chapter 2.2]{TT} and \cite[Theorem 3]{Shen})。
%Takhtajan-Teo はこうした Weil-Petersson class に入る正則写像たちを元に、
%？？？が可能な universal Teichm\"{u}ller 空間の理論を構築した。
%こうした Takhtajan-Teo の universal Teichm\"{u}ller 空間の理論、および \cite{} ?? は、
%我々の operad $\CEh$によって定義される \todo
%
%まで拡張されたより適切な圏 $\Embc$ の定式化と、その factorization homology 理論の universal Teichm\"{u}ller 空間との関係や Segal CFT
%のと関係を示唆する。こうした研究は将来の仕事とする。
%\end{rem}
%
In general, for $n \geq 0$, we set
\begin{align*}
\CEh(n) = \{(\phi_1,\dots,\phi_n) \in \CEh(1)^n \cap \CEc(n) \mid G_{\phi_i,\phi_j}(z,w) \in L^2(\bD \times \bD) \text{ for any }i\neq j \}.
\end{align*}

\begin{prop}\label{prop_CE_HS}
$\CEh$ is a suboperad of $\CEc$.
\end{prop}
\begin{proof}
Let $(f_1,\dots,f_n) \in \CEh(n)$ and $(g_1,\dots,g_m) \in \CEh(m)$.
Then
$(f_1,\dots,f_n) \circ_i (g_1,\dots,g_m) = (f_1,\dots,f_{i}g_1,\dots,f_ig_m,\dots,f_n)$
is in $\CEc(n+m-1)$.
By Proposition~\ref{prop_monoid_HS}, we have $f_i g_j \in \CEh(1)$.
Hence, it suffices to show that $G_{f_ig_k,f_ig_l}(z,w)$ and $G_{f_j,f_ig_k}(z,w)$ belong to $L^2(\bD \times \bD)$, and this follows from Proposition~\ref{prop_FG_equality} and the similar computation as in \eqref{eq_cocycle_L2}.
\end{proof}

%一般に$n \geq 0$に対して、
%\begin{align*}
%\CEh(n) = \{(\phi_1,\dots,\phi_n) \in \CEh(1)^n \cap \CEc(n) \mid G_{\phi_i,\phi_j}(z,w) \in L^2(\bD \times \bD) \text{ for any }i\neq j \}
%\end{align*}
%とおく。
%
%\begin{prop}\label{prop_CE_HS}
%$\CEh$は$\CEc$の suboperad である。
%\end{prop}
%\begin{proof}
%Let $(f_1,\dots,f_n) \in \CEh(n)$ and  $(g_1,\dots,g_m) \in \CEh(m)$とする。このとき
%$(f_1,\dots,f_n) \circ_i (g_1,\dots,g_m) = (f_1,\dots,f_{i}g_1,\dots,f_ig_m,\dots,f_n) \in \CEh(n+m-1)$.
%ここで命題\ref{prop_monoid_HS}より$f_i g_j \in \CEh(1)$である。
%そこで$G_{f_ig_k,f_ig_l}(z,w), G_{f_j,f_ig_k} \in L^2(\bD \times \bD)$ を示せば十分であるが、これは
% Proposition \ref{prop_FG_equality}
%and \eqref{eq_cocycle_L2}と同様の計算から従う。
%\end{proof}
%
%
%最後に$\CEh$と、\cite{}で導入された operad $\CE$の関係を調べる。まず初めに\cite{}に従い$\CE$の定義を簡単に振り返る。

Finally, we examine the relationship between $\CEh$ and the operad $\CE$ introduced in \cite[Section 1.1]{MLeft}.
To begin, we briefly review the definition of $\CE$.
Recall that $\CE(1)$ is a subset of $\CEc(1)$ consisting of $f \in \CEc(1)$ such that
\begin{itemize}
\item
there is a open neighborhood $U$ of $\overline{\bD}$ and injective holomorphic map $\tilde{f}:U \rightarrow \C$ which satisfies $\tilde{f}|_\bD = f$.
\end{itemize}

The following lemma is proved in \cite[Proposition 1.6]{MLeft}, but we give an explicit proof:
\begin{lem}\label{lem_CE2}
Let $f,g \in \CE(1)$. Then, $f \circ g \in \CE(1)$. In particular, $\CE(1)$ is a submonoid of $\CEc$.
\end{lem}
\begin{proof}
It suffices to show that there is an extension of $f\circ g$.
Let $\tilde{f}:U \rightarrow \C$ and $\tilde{g}:V \rightarrow \C$ be the extensions.
Since $\tilde{g}(\overline{\bD}) = \overline{g(\bD)} \subset \overline{\bD}$ and $\overline{\bD} \subset U$,
$\tilde{g}^{-1}(U) \cap V$ is an open neighborhood of $\overline{\bD}$.
Then, $\tilde{f}|_{U \cap \tilde{g}(V)} \circ \tilde{g}: \tilde{g}^{-1}(U) \cap V \rightarrow \C$
is an injective holomorphic map. Hence, the assertion holds.
\end{proof}

Fo any $n \geq 0$, set
\begin{align*}
\CE(n)=\{(f_1,\dots,f_n) \in \CE(1)^n \mid \overline{f_i(\bD)}\cap \overline{f_j(\bD)}=\emptyset\text{ for }i\neq j \}.
\end{align*}
Here, $\overline{f_i(\bD)}$ is the closure of $f_i(\bD) \subset \C$. Then, $\CE$ is a suboperad of $\CEc$ (see \cite[Proposition 1.6]{MLeft}).
\begin{prop}\label{prop_sub_CEh}
As suboperads of $\CEc$, $\CE \subset \CEh$. In particular, for any $f \in \CE(1)$ and $(g_1,\dots,g_n) \in \CE(n)$,
\begin{align*}
F_{f}(z,w), \; G_{g_i,g_j}(z,w) \in L^2(\bD \times \bD)\qquad \text{ for any }i\neq j.
\end{align*}
Moreover, both $\CE$ and $\CEh$ are stable under the automorphism $\J$ of $\CEc$.
%$\CEc$の suboperads として$\CE \subset \CEh$が成り立つ。とくに任意の$f \in \CE(1)$ and $(g_1,\dots,g_n) \in \CE(n)$に対して、
%\begin{align*}
%F_{f}(z,w), G_{g_i,g_j}(z,w) \in L^2(\bD \times \bD)
%\end{align*}
%が $i \neq j$に対して成り立つ。
\end{prop}
\begin{proof}
Let $\tilde{f}:U \rightarrow \C$ be the extension.
Then, 
\begin{align*}
\tilde{F}_{\tilde{f}}(z,w) = \frac{\tilde{f}'(z)\tilde{f}'(w)}{(\tilde{f}(z)-\tilde{f}(w))^2} - \frac{1}{(z-w)^2}
\end{align*}
is a holomorphic function on $U \times U$. 
In particular, $F_f(z,w)$ extends to a continuous function on $\overline{\bD} \times \overline{\bD}$. Hence $F_{f}(z,w) \in L^2(\bD \times \bD)$.
Moreover, since $\overline{f_i(\bD)}\cap \overline{f_j(\bD)}=\emptyset$, 
\begin{align*}
\delta_{ij} = \inf_{z \in \overline{f_i(\bD)}, w \in \overline{f_j(\bD)}} |z-w|
\end{align*}
satisfies $\delta_{ij}>0$. Therefore,
\begin{align*}
\int_{\bD^2}|G_{g_i,g_j}(z,w)|^2d\mu \leq \frac{1}{\delta_{ij}^2} \int_{\bD^2}|g_i'(z)g_j'(w)|^2 d\mu =\frac{1}{\delta_{ij}^2} \mathrm{vol}(g_i(\bD) \times g_j(\bD))
\end{align*}
and it follows that $G_{g_i,g_j}(z,w) \in L^2(\bD \times \bD)$.
The last claim is clear.
%とくに$F_f(z,w)$は$\overline{\bD} \times \overline{\bD}$上に連続関数として延長される。よって$F_{f}(z,w) \in L^2(\bD \times \bD)$。
%また$\overline{f_i(\bD)}\cap \overline{f_j(\bD)}=\emptyset$であるから、
%\begin{align*}
%\delta_{ij} = \inf_{z \in \overline{f_i(\bD)}, w \in \overline{f_j(\bD)}} |z-w|
%\end{align*}
%は$\delta_{ij}>0$を満たす。よって
%\begin{align*}
%\int_{\bD^2}|G_{g_i,g_j}(z,w)|^2d\mu \leq \frac{1}{\delta_{ij}^2} \int_{\bD^2}|g_i'(z)g_j'(w)|^2 d\mu =\frac{1}{\delta_{ij}^2} \mathrm{vol}(g_i(\bD) \times g_j(\bD))
%\end{align*}
%より、$G_{g_i,g_j}(z,w) \in L^2(\bD \times \bD)$が従う。
\end{proof}

\section{Conformally flat 2-disk algebra on the Bergman space}
In this section, we show that $\Sym A^2(\bD)$ carries a $\CEh$-algebra structure.
This provides a completion, in the ind-Hilbert space setting, of the factorization algebra 
associated with the conformal Laplacian studied in \cite{Mfactorization}.

\subsection{Classical and quantum monoid action}\label{sec_monoid}

In this section, we construct a monoid homomorphism
$\CEh(1)\rightarrow \mathrm{End}(\Sym A^2(\bD))$.

Let $\phi \in \CEc$ and $f \in A^2(\bD)$. Since
\begin{align}
\int_\bD |\phi'(z)f(\phi(z))|^2 dz^2 = \int_{\phi(\bD)} |f(z)|^2 dz^2 \leq \int_{\bD} |f(z)|^2 dz^2,\label{eq_contraction_D}
\end{align}
\begin{align}
T_\phi: A^2(\bD) \rightarrow A^2(\bD),\quad f(z)\mapsto \phi'(z)\,(f\circ \phi)(z)
\label{eq_contraction}
\end{align}
is a bounded linear map with operator norm at most $1$.
For $\phi_1,\phi_2 \in \CEc$, we have
$T_{\phi_1}T_{\phi_2}(f)(z)
= (\phi_2'\circ \phi_1)(z)\,\phi_1'(z)\,(f \circ \phi_2 \circ \phi_1)(z)
= T_{\phi_2 \circ \phi_1}(f)(z)$.
Hence, 
$T:\CEc \rightarrow B(A^2(\bD))$
defines an anti-homomorphism of monoid.
We denote by $T_\phi^*:A^2(\bD) \rightarrow A^2(\bD)$ the adjoint of $T_\phi$.

%\begin{align}
%T_\phi: A^2(\bD) \rightarrow A^2(\bD),\quad f(z)\mapsto \phi'(z) (f\circ \phi)(z)
%\label{eq_contraction}
%\end{align}
%は、bounded linear map であり、operator norm は1以下である。
%ここで$\phi_1,\phi_2 \in \fS$に対して、
%$
%T_{\phi_1}T_{\phi_2}(f)(z) = (\phi_2'\circ \phi_1)(z) \phi_1'(z) (f \circ\phi_2 \circ \phi_1)(z)
%= T_{\phi_2 \circ \phi_1}(f)(z)$ であるから、
%$T:\fS \rightarrow B(A^2(\bD))$
%は semigroup の anti-homomorphism である。
%$T_\phi^*:A^2(\bD) \rightarrow A^2(\bD)$を$T_\phi$の随伴作用素とする。
%すなわち
%\begin{align*}
%(T_\phi^* f,g)_B = (f,T_\phi g)_B
%\end{align*}
%が任意の$f,g \in A^2(\bD)$に対して成り立つ。
For any $a\in \bD$, since $
(T_\phi^* E_a,f)_B = (E_a,T_\phi f)_B =
 (E_a,\phi' f\circ \phi )_B = \phi'(\ba) f(\phi(\ba))$, we have
\begin{align}
T_\phi^*(E_a) = \overline{\phi'(\ba)} E_{\overline{\phi(\ba)}}.\label{eq_T_star}
\end{align}
Define
$\rho_\cl:\CEc \rightarrow \bB(A^2(\bD))$
by
\begin{align*}
\rho_\cl(\phi) = T_{\J(\phi)}^*.
\end{align*}
Since $T_{\phi_1 \phi_2}^* = (T_{\phi_2}T_{\phi_1})^* = T_{\phi_1}^*T_{\phi_2}^*$
and $\J$ is an automorphism of $\CEc(1)$,
we have:
\begin{prop}\label{prop_monoid_classical}
The map $\rho_\cl:\CEc(1) \rightarrow \bB(A^2(\bD))$ is a monoid homomorphism
 satisfying
\begin{align*}
\rho_\cl(\phi)(E_a(z)) =\phi'(a)E_{\phi(a)}
\end{align*}
for any $a\in \bD$ and $\phi \in \CEc(1)$.
Moreover, the restriction $\rho_\cl|_{\fG}:\fG \rightarrow \bB(A^2(\bD))$ is a strongly continuous unitary representation.
\end{prop}
\begin{proof}
It suffices to prove the last claim. Let $\phi \in \fG$. Since for any $a,b \in \bD$, by Proposition \ref{prop_dense} and Proposition \ref{prop_conf_identity},
\begin{align*}
(T_{\J(\phi)}^* E_a,T_{\J(\phi)}^* E_b)
&= \frac{\overline{\phi'(a)}{\phi'(b)}}{(1- \overline{\phi(a)}{\phi(b)})^2}
= \frac{1}{(1-\ba{b})^2}
= (E_a,E_b),
\end{align*}
and since $\{E_a\}_{a\in \bD} \subset A^2(\bD)$ spans a dense subspace, it follows that $T_\phi^*$ is unitary.
Hence $\rho_\cl|_{\fG}$ is a unitary representation.
For $E_a \in A^2(\bD)$, it is easy to see that $\rho_\cl(g)E_a \to E_a$ as $g \to 1$.
Since the subspace spanned by $\{E_a\}_{a\in \bD}$ is dense, strong continuity follows.
%
%
%
%最後の主張を示せば十分。
%Let $\phi \in \fG$. Since for any $a,b \in \bD$, by Proposition \ref{prop_dense} and Proposition \ref{prop_conf_identity},
%\begin{align*}
%(T_{\J(\phi)}^* E_a,T_{\J(\phi)}^* E_b)&= \frac{\overline{\phi'(a)}{\phi'(b)}}{(1- \overline{\phi(a)}{\phi(b)})^2}=\frac{1}{(1-\ba{b})^2} = (E_a,E_b).
%\end{align*}
%Since $\{E_a\}_{a\in \bD} \subset A^2(\bD)$ spans a dense subspace, $T_\phi^*$ is unitary.
%よって$\rho_\cl|_{\fG}$は unitary 表現である。
%$E_a \in A^2(\bD)$に対して、$\rho_\cl(g)E_a \to E_a$ for $g \to 1$は簡単に分かる。$\{E_a\}_{a\in \bD}$で張られる空間が dense であることから、強連続性は従う。
\end{proof}

Let $a \in \bD$ and $1>r>0$ satisfy $B_r(a) \subset \bD$.
Then,
\begin{align}
B_{a, r}(z) = rz+a \label{eq_def_B}
\end{align}
is an element of $\CE(1)$.
\begin{lem}\label{lem_translation_basis}
For any $n \geq 0$,
\begin{align*}
\rho_\cl(B_{a,r})(\sqrt{n+1}e_n(z)) =\frac{r^{n+1}}{n!} \pa_a^n E_a(z).
\end{align*}
\end{lem}
\begin{proof}
For any $f\in A^2(\bD)$, by $f(rz+a)=\sum_{k \geq 0} \frac{1}{k!}f^{(k)}(a) (rz)^k$,
\begin{align*}
(\rho_\cl(B_{a, r})(\sqrt{n+1}e_n(z)),f)
&=(T_{\J(B_{a, r})}^*(\sqrt{n+1}e_n(z)),f)\\
&=(\sqrt{n+1}e_n(z),T_{B_{\ba, r}}f)\\
&= \int_{\bD} \overline{\sqrt{n+1}e_n(z)} rf(rz+\ba) d\mu=\frac{r^{n+1}}{n! }f^{(n)}(\ba).
\end{align*}
Hence, the proposition follows from Lemma \ref{lem_kernel_derivative}.
\end{proof}

We examine dilations for $r \in (0,1]$.
By Lemma \ref{lem_translation_basis},
\begin{align*}
\rho_\cl(B_{0,r})(z^n)
=\frac{r^{n+1}}{(n+1)!} \pa_a^n e_a(z) \big|_{a=0}
= r^{n+1} z^n.
\end{align*}
Hence $\{z^n \in A^2(\bD)\}_{n \geq 0}$
forms an orthonormal basis consisting of simultaneous eigenvectors for the dilations.
Hence, we have (see \cite[Definition 1.32]{MLeft}):
\begin{lem}\label{lem_dilation}
The dilation $\rho_\cl(B_{0,r})$ defines a strongly continuous, self-adjoint, contractive monoid homomorphism
\begin{align*}
(0,1] \rightarrow \bB(A^2(\bD)),
\end{align*}
and $z^n$ is a simultaneous eigenvector satisfying $\rho_\cl(B_{0,r})z^n = r^{n+1}z^n$.
Moreover, $\rho_\cl(B_{0,r})$ is a Hilbert--Schmidt operator for $r<1$.
\end{lem}

Denote by $\End(\Sym A^2(\bD))$ the space of all linear endomorphisms of $\Sym A^2(\bD)$.
We extend $\rho_\cl:\CEc(1)\rightarrow \bB(A^2(\bD))$ to a monoid homomorphism
\begin{align*}
\rho_\cl:\CEc(1)\rightarrow \End(\Sym A^2(\bD))
\end{align*}
as follows.
 For $p\geq 1$ and $v_1\otimes \cdots \otimes v_p \in A^2(\bD)^{\otimes p}$, set
\begin{align*}
\rho_\cl(\phi)(v_1\otimes \cdots \otimes v_p)
=(\rho_\cl(\phi)v_1)\otimes \cdots \otimes (\rho_\cl(\phi)v_p),
\end{align*}
which is uniquely extend to a bounded linear map on $\Sym^p A^2(\bD)$.
For $p=0$ define $\rho_\cl(\phi)\va=\va$.

%The monoid homomorphism $\rho_\cl:\CEc(1) \rightarrow \bB(A^2(\bD))$ を monoid 準同型
%\begin{align*}
%\rho_\cl: \CEc(1) \rightarrow \End(\Sym A^2(\bD))
%\end{align*}
%へ以下のように拡張する。$p \geq 1$の場合は、
%\begin{align*}
%\rho_\cl(v_1\otimes \cdots \otimes v_p) = (\rho_\cl(\phi)v_1)\otimes \cdots \otimes (\rho_\cl(\phi)v_p)
%\end{align*}
%によって定め、$p=0$の場合は$\rho_\cl(\phi)\va=\va$によって定める。ここで$\End(\Sym A^2(\bD))$は$\Sym A^2(\bD)$上の線形写像全体。

In what follows, we introduce a quantum correction to the monoid homomorphism $\rho_\cl$ and construct a monoid homomorphism
\begin{align*}
\rho:\CEh(1) \rightarrow \End(\Sym A^2(\bD)).
\end{align*}

%よってLemma \ref{lem_kernel_derivative}より命題が従う。
%\end{proof}
%
%最後に$r \in (0,1]$に対する、dilation を調べる。
%Lemma \ref{lem_translation_basis}より
%\begin{align*}
%\rho_\cl(B_{0,r})(z^n) =\frac{r^{n+1}}{(n+1)!} \pa_a^n e_a(z) |_{a=0}=r^{n+1} z^n.
%\end{align*}
%よって$\{z^n \in A^2(\bD) \}_{n \geq 0}$
%は dilations に対する、同時固有ベクトルからなる正規直交基底である。Hence, we have:
%\begin{lem}\label{lem_dilation}
%Dilation $\rho_\cl(B_{0,r})$は strongly continuous, self-adjoint, contraction な monoid の準同型
%\begin{align*}
%(0,1] \rightarrow B(A^2(\bD))
%\end{align*}
%を定め、$z^n$は$\rho_\cl(B_{0,r})z^n=r^{n+1}z^n$を満たす同時固有ベクトルである。とくに$\rho_\cl(B_{0,r})$は$r<1$のとき Hilbert-Schdmit operator である。
%\end{lem}
%
%
%
%
%
%以下 monoid 準同型$\rho_\cl$に量子補正を加えて monoid 準同型
%\begin{align*}
%\rho:\CEh(1) \rightarrow \End(\Sym A^2(\bD))
%\end{align*}
%を構成する。
Let $\phi\in \CEh(1)$. Then, by definition, 
\begin{align}
F_\phi(z,w) \in A^2(\bD\times \bD)\cong A^2(\bD)\hotimes A^2(\bD).
\label{eq_F_L2}
\end{align}
Define a bounded linear map $\pa_{\phi}: A^2(\bD\times \bD) \rightarrow \C$ by
\begin{align*}
\pa_{\phi}\left(P(z,w)\right)= (F_{\J\phi},P)_B=\int_{\bD\times \bD} {F_\phi(\z,\w)}P(z,w) d\mu
%\label{def_partial_phi}
\end{align*}
for $P(z,w) \in A^2(\bD \times \bD)$ (see Proposition \ref{prop_sub_CEh}). 
%Define $\{d_{n,m}\}_{n,m>0}$ by
%\begin{align}
%F_\phi(z,w) &= \sum_{n,m \geq 0}\sqrt{(n+1)(m+1)}d_{n,m}z^n w^m\\
%&= \sum_{n,m \geq 0}d_{n,m}e_n(z) e_m(w),
%\label{eq_expansion_F}
%\end{align}
%where $\{e_n(z)e_m(w)\}_{n,m \geq 0}$ is an orthonormal basis of
%$A^2(\bD \times \bD)$.
%By Lemma \ref{lem_tensor_isomorphism}, this can equivalently be regarded as a
%bounded operator $A^2(\bD) \hotimes A^2(\bD) \rightarrow \C$.

%where $\{e_n(z)e_m(w)\}_{n,m \geq 0}$ は $A^2(\bD \times \bD)$の正規直交基底。
%補題\ref{lem_tensor_isomorphism}より これは有界作用素 、とも思える:
\begin{prop}\label{lem_partial_phi_dense}
For any $\phi \in \CEh(1)$, the bilinear map $\pa_{\phi}: A^2(\bD)\hotimes  A^2(\bD) \rightarrow \C$ satisfies
\begin{align*}
|\pa_{\phi}(P)|\leq \norm{F_{\J\phi}}_B \norm{P}_B
\end{align*}
for any $P \in A^2(\bD) \hotimes A^2(\bD)$. Moreover, for any $a,b\in \bD$,
\begin{align*}
\pa_{\phi}(E_a \otimes E_b) = {F_\phi(a,b)}.
\end{align*}
\end{prop}
\begin{proof}
Since $\pa_{\phi}(E_a,E_b)=(F_{\J\phi}(z,w),E_a(z)E_b(w))=\overline{(E_a(z)E_b(w),F_{\J\phi}(z,w))}=\overline{F_{\J\phi}(\ba,\bar{b})}=F_{\phi}(a,b)$.
%the sum $F_a(z) = \sum_{n} \overline{e_n(a)} e_n(z)$ converges in $L^2$,
%and $\pa_\phi$ is bounded, it follows that
%%Since the sum $F_a(z) =\sum_{n} \overline{e_n(a)}e_n(z)$は$L^2$収束するので、$\pa_\phi$が有界であることから
%\begin{align*}
%\pa_{\phi}(F_a,F_b) &= \sum_{n,m\geq 0}\overline{e_n(a)} \overline{e_m(b)}\overline{d_{n+1,m+1}}
%% &= \sum_{n,m\geq 0}
%%\sqrt{(n+1)(m+1)}\bar{a}^n \bar{b}^m\overline{d_{n+1,m+1}}
%=\overline{F_{\phi}(a,b)}.
%\end{align*}
\end{proof}

Let $\phi \in \CEh(1)$. Define $\pa_\phi: (A^2(\bD))^{\hotimes n} \rightarrow ( A^2(\bD))^{\hotimes n-2}$ by
\begin{align}
\pa_\phi(a_1 \hotimes \cdots \hotimes a_n)=
\sum_{1 \leq i<j \leq n} \pa_\phi(a_i \otimes a_j) \left(a_1 \hotimes \cdots   \hat{a}_i \cdots \hat{a}_j \cdots \hotimes a_n\right)
\label{eq_pa_phi_cont}
\end{align}
for $n \geq 2$ and $\pa_\phi=0$ if $n=0,1$.
Here $a_1\hotimes\cdots \hat{a}_i \cdots \hat{a}_j \cdots \hotimes a_n$ denotes the vector in $A^2(\bD)^{\otimes (n-2)}$ obtained by omitting $a_i(z)$ and $a_j(z)$.
Since $\pa_\phi$ is an operator of degree $-2$ and is locally nilpotent,
the exponential
\begin{align}
\exp(\pa_\phi): \bigoplus_{p \geq 0} A^2(\bD)^{\hotimes p} \rightarrow \bigoplus_{p \geq 0} A^2(\bD)^{\hotimes p}
\label{eq_exp_tensor}
\end{align}
is a well-defined linear map.
\begin{lem}
\label{lem_cont_sym}
For any $p \geq 2$ and $a_1,\dots,a_p \in A^2(\bD)$,
\begin{align*}
\pa_\phi (\bs^p(a_1\otimes \cdots \otimes a_p))= \bs^{p-2}(\pa_\phi(a_1\otimes \cdots \otimes a_p)).
\end{align*}
\end{lem}
\begin{proof}
Since $F_{\J(\phi)}(z,w) = F_{\J(\phi)}(w,z)$, $\pa_\phi(a \otimes b)=\pa_\phi(b \otimes a)$ for any $a,b \in A^2(\bD)$. Hence,
\begin{align*}
\pa_\phi (\bs^p(a_1\otimes \cdots \otimes a_p))
&=\frac{1}{p!} \sum_{\si \in S_p} \pa_\phi (a_{\si(1)} \otimes \cdots \otimes a_{\si(p)}))\\
&=\frac{1}{p!} \sum_{\si \in S_p}\sum_{i<j} \pa_\phi(a_{\si(i)},a_{\si(j)}) \left(a_{\si(1)} \otimes \cdots   \hat{a}_{\si(i)} \cdots \hat{a}_{\si(j)} \cdots \otimes a_p\right)\\
&=\bs^{p-2}(\pa_\phi(a_1\otimes \cdots \otimes a_p)).
\end{align*}
\end{proof}
Hence, the restriction of \eqref{eq_exp_tensor} on $\Sym A^2(\bD) \subset \bigoplus_{p \geq 0} A^2(\bD)^{\hotimes p}$
defines a linear map $\exp(\pa_\phi): \Sym A^2(\bD) \rightarrow \Sym A^2(\bD)$.
%Since $\pa_\phi$は次数$-2$の作用素であり、locally nilpotent であるから、
%$\exp(\pa_\phi):\Sym A^2(\bD) \rightarrow \Sym A^2(\bD)$が定義できる。
Set
\begin{align*}
\rho(\phi) = \rho_\cl(\phi)\circ \exp(\pa_\phi):\Sym A^2(\bD) \rightarrow \Sym A^2(\bD).
\end{align*}

\begin{thm}\label{thm_semigroup_quantum}
For any $\phi,\psi \in \CEh(1)$, $\rho(\phi \circ \psi) = \rho(\phi)\circ \rho(\psi)$ and $\rho(1_\bD)=\id$, that is, $\rho$ is a representation of the monoid $\CEh$ on $\Sym A^2(\bD)$. 
Moreover, for any $k \geq 0$, $\rho$ preserves $\Hf^k=\bigoplus_{p=0}^k \Sym^pA^2(\bD) \subset \Sym A^2(\bD)$ and for any $\phi \in \CEh(1)$,
\begin{align*}
\rho(\phi)|_{\Hf^k}: \Hf^k \rightarrow \Hf^k
\end{align*}
is a bounded linear operator.
Furthermore, the restriction $\rho|_{\mathrm{PSU}(1,1)}: \mathrm{PSU}(1,1) \rightarrow \mathbb{B}(\Hf^k)$ is a strongly continuous unitary representation.
%Moreover, for any $k \geq 0$, $\rho$は$\Hf^k \subset \Sym A^2(\bD)$を$\Hf^k$に送り、for any $\phi \in \CEh(1)$,
%\begin{align*}
%\rho(\phi)|_{\Hf^k}: \Hf^k \rightarrow \Hf^k
%\end{align*}
%は有界線形作用素である。さらに$\rho|_{\mathrm{PSU}(1,1)}: {\mathrm{PSU}(1,1)} \rightarrow B(H^k)$ は強連続な unitary 表現である。
\end{thm}
\begin{lem}\label{lem_partial_cocycle}
Let $\phi,\psi \in \CEh(1)$. Then,
\begin{align*}
\pa_{\psi}+\pa_{\phi} \circ (\rho_\cl(\psi)\hotimes\rho_\cl(\psi)) = \pa_{\phi \circ \psi}
\end{align*}
as linear maps $A^2(\bD)\hotimes A^2(\bD) \rightarrow \C$.
\end{lem}
\begin{proof}
Since the both sides are bounded linear operators, it suffices to show the identity on the subset $\{E_a \}_{a\in \bD} \subset A^2(\bD)$. By Lemma \ref{lem_partial_phi_dense}, for any $a,b\in \bD$
\begin{align*}
\left(\pa_{\psi}+\pa_{\phi} \circ (\rho_\cl(\psi)\hotimes\rho_\cl(\psi)) \right)(E_a,E_b)&={F_{\psi}(a,b)}+ {\psi'(a)}{\psi'(b)} {F_{\phi}(\psi(a),\psi(b))}.
\end{align*}
Hence, the assertion follows from Proposition \ref{prop_FG_equality}.
%\begin{align*}
%&F_{\psi}(a,b)+ \psi'(a){\psi'(b)}{F_{\phi}(\psi(a),\psi(b))}\\
%&=
%\left(\frac{\psi'(a)\psi'(b)}{(\psi(a)-\psi(b))^2}-\frac{1}{(a-b)^2}\right)
%+\psi'(a)\psi'(b)
%\left(\frac{\phi'(\psi(a))\phi'(\psi(b))}{(\phi(\psi(a))-\phi(\psi(b)))^2}-\frac{1}{(\psi(a)-\psi(b))^2}\right)\\
%&= F_{\phi \circ \psi}(a,b).
%\end{align*}
\end{proof}

\begin{proof}[proof of Theorem \ref{thm_semigroup_quantum})]
Since
\begin{align*}
\exp(\pa_\phi)\rho_\cl(\psi) =\rho_\cl(\psi) \exp({\pa_\phi \circ (\rho_\cl(\psi)\hotimes \rho_\cl(\psi))})
\end{align*}
as linear maps on $\mathrm{Sym}A^2(\bD)$, by Lemma \ref{lem_partial_cocycle}, we have
\begin{align*}
&\rho(\phi)\rho(\psi)=
\rho_\cl(\phi)e^{\pa_\phi} \rho_\cl(\psi) e^{\pa_\psi} 
%\rho_\cl(\phi)\rho_\cl(\psi)\rho_\cl(\psi)^{-1} e^{\pa_\phi} \rho_\cl(\psi) e^{\pa_\psi}\\
 =\rho_\cl(\phi)\rho_\cl(\psi) e^{\pa_\phi \circ (\rho_\cl(\psi)\hotimes \rho_\cl(\psi))}e^{\pa_\psi}\\
& =\rho_\cl(\phi\circ \psi) e^{\pa_\phi \circ (\rho_\cl(\psi)\hotimes \rho_\cl(\psi))+\pa_\psi}
=\rho_\cl(\phi\circ \psi) e^{\pa_{\phi \circ \psi}}=\rho(\phi \circ \psi).
\end{align*}
Here, we use $[\pa_\phi \circ (\rho_\cl(\psi)\hotimes \rho_\cl(\psi)),\pa_\psi]=0$ as linear maps on $\Sym A^2(\bD)$.
Strong continuity and unitarity follow immediately from Proposition~\ref{prop_cocycle_vanish} and Proposition~\ref{prop_monoid_classical}.
\end{proof}

\subsection{Construction of geometric wick contraction}\label{sec_wick}
Let $(\phi,\psi) \in \CEh(2)$. Then,
%Then, that is, there is an open subset $U \subset \C$ such that
%$\overline{\bD} \subset U$ and $\phi,\psi:U \rightarrow \C$ are injective holomorphic maps
%and  
%\begin{align}
%\phi(\bD), \psi(\bD) \subset \bD \text{ and }
%\overline{\phi(\bD)} \cap \overline{\psi(\bD)} =\emptyset.
%\label{eq_geometry}
%\end{align}
%Set
%\begin{align*}
%h_{\phi,\psi}(z,w) = \frac{\phi'(z)\psi'(w)}{(\phi(z)-\psi(w))^2}
%\end{align*}
%which is a holomorphic map on $\bD \times \bD$ 
%and square-integrable by \eqref{eq_geometry}. Hence,
\begin{align*}
G_{\phi,\psi}(z,w) \in A^2(\bD \times \bD).
\end{align*}
Define a bounded linear map $C_{\phi,\psi}: A^2(\bD\times \bD) \rightarrow \C$ by
\begin{align*}
C_{\phi,\psi}\left(P(z,w)\right)= (G_{\J\phi,\J\psi},P)_B=\int_{\bD\times \bD} {G_{\phi,\psi}(\z,\w)}P(z,w) d\mu
%\label{def_partial_phi}
\end{align*}
for $P(z,w) \in A^2(\bD \times \bD)$. 
Define $\{g_{n,m}\}_{n,m\geq 0}$ by
\begin{align}
G_{\phi,\psi}(z,w) =
 \sum_{n,m \geq 0}g_{n,m}z^n w^m.\label{eq_expansion}
\end{align}

Then, we have:
\begin{prop}\label{prop_G_dense}
Let $(\phi,\psi) \in \CEh(2)$. Then, there is a unique bounded linear map 
$C_{\phi,\psi}:A^2(\bD) \hotimes A^2(\bD) \rightarrow \C$
such that 
\begin{align*}
C_{\phi,\psi}(E_a \otimes E_b) = {G_{\phi,\psi}(a,b)}.
\end{align*}
for any $a,b \in \bD$. Moreover, 
$C_{\phi,\psi}$ satisfies
$C_{\phi,\psi}(\sqrt{n+1}e_n(z)\otimes \sqrt{m+1}e_m(w)) = {g_{n,m}}$ for any $n,m \geq 0$.
\end{prop}

\begin{prop}\label{prop_contraction_covariance}
Let $(\phi_1,\phi_2)\in \CEh(2)$. Then, the following properties hold as  bounded linear maps $A^2(\bD)\hotimes A^2(\bD) \rightarrow \C$:
\begin{enumerate}
\item
For any $\psi_1,\psi_2 \in \CEh(1)$,
\begin{align*}
C_{\phi_1 \circ \psi_1,\phi_2 \circ \psi_2} = C_{\phi_1, \phi_2} \circ (\rho_\cl(\psi_1) \otimes \rho_\cl(\psi_2)).
\end{align*}
\item
For any $h \in \CEh(1)$,
\begin{align*}
C_{h \circ \phi_1, h \circ \phi_2} = C_{\phi_1,\phi_2} + \pa_{h}\circ (\rho_\cl(\phi_1) \otimes \rho_\cl(\phi_2)).
\end{align*}
\end{enumerate}
\end{prop}
\begin{proof}
It suffices to show the equalities on the subset $\{E_a\}_{a\in \bD} \subset A^2(\bD)$.
(1) follows from Proposition \ref{prop_FG_equality} and 
\begin{align*}
C_{\phi_1,\phi_2}(\rho_\cl(\psi_1)E_a\otimes \rho_\cl(\psi_2) E_b) &= {\psi_1'(a)}{\psi_2'(b)}C_{\phi_1,\phi_2}(E_{\psi_1(a)}\otimes E_{\psi_2(b)})\\
&= 
{\frac{\psi_1'(a) \phi_1'(\psi_1(a)) \psi_2'(b) \phi_2'(\psi_2(b))}{(\phi_1(\psi_1(a))-\phi_2(\psi_2(b)))^2}}=C_{\phi_1\psi_1,\phi_2\psi_2}(E_a \otimes E_b).
\end{align*}
Similarly, (2) follows from
\begin{align*}
\pa_{h}\circ (\rho_\cl(\phi_1) \otimes \rho_\cl(\phi_2))(E_a \otimes E_b) &= {\phi_1'(a)\phi_2'(b)} \pa_h(E_{\phi_1(a)} \otimes E_{\phi_2(b)})\\
&= {\phi_1'(a)\phi_2'(b)}{\left(\frac{h'(\phi_1(a))h'(\phi_2(b))}{(h\circ\phi_1(a)-h\circ\phi_2(b))^2}   -\frac{1}{(\phi_1(a)-\phi_2(b))^2}  \right)}\\
&=C_{h \circ \phi_1, h \circ \phi_2}(E_a \otimes E_b) - C_{\phi_1, \phi_2}(E_a \otimes E_b) 
\end{align*}
\end{proof}

Let $a,b\in\bD$ and $1>r,s>0$ satisfy
\begin{align*}
B_r(a),B_s(b) \subset \bD
\qquad {B_r(a)} \cap {B_s(b)}=\emptyset.
\end{align*}
Then, $(B_{a,r}(z),B_{b,s}(z)) \in \CEc(2)$ (see \eqref{eq_def_B}) and
\begin{align}
G_{(B_{a,r}(z),B_{b,s}(z))}(z,w)= \frac{rs}{(rz+a-sw-b)^2}.\label{eq_h_trans}
\end{align}
\begin{prop}\label{prop_contraction_trans}
$G_{(B_{a,r}(z),B_{b,s}(z))}(z,w) \in A^2(\bD \times \bD)$ if and only if
\begin{align*}
\overline{B_r(a)} \cap \overline{B_s(b)}=\emptyset.
\end{align*}
Moreover, in this case $(B_{a,r}(z),B_{b,s}(z)) \in \CE(2)$, and for any $n,m \geq 0$,
\begin{align*}
C_{(B_{a,r}(z),B_{b,s}(z))}(\sqrt{n+1}e_n(z),\sqrt{m+1}e_m(z)) 
&= \frac{r^{n+1}}{n!} \frac{s^{m+1}}{m!} \pa_z^n\pa_w^m\frac{1}{(z-w)^2}\Bigl|_{z-w = a-b}.
\end{align*}
\end{prop}
\begin{proof}
The equivalence about square-integrability can be proved by exactly the same argument as in the proof of \cite[Proposition 2.26]{MLeft}.
The final claim follows from Proposition \ref{prop_G_dense}.
%square-integrable に関する同値性は \cite[]{MLeft}の証明と全く同様の議論で示せる。
%最後の主張は命題 \ref{prop_G_dense}より従う。
\end{proof}

\subsection{Construction of conformally flat disk 2-algebra}\label{sec_construction}
In this section, we construct a $\CEh$-algebra using Proposition \ref{prop_contraction_covariance}.
Many parts of the argument are parallel to \cite[Section 2.5]{MLeft}.
%多くの部分は\cite[Section ]{MLeft}とパラレルである.
Set $A^2(\bD)=\Hf$ and $A=\Sym(\Hf)$ for short.

Let $(g_1,g_2) \in \CEh(2)$.
Similarly to \eqref{eq_pa_phi_cont}, for $p,q \geq 0$ we define
\begin{align}
\hat{C}_{g_1,g_2}:\Hf^{\hotimes p} \hotimes \Hf^{\hotimes q}
\rightarrow
\Hf^{\hotimes p-1} \hotimes \Hf^{\hotimes q-1}
\label{eq_hat_C_def}
\end{align}
as follows.
If $p,q \geq 1$, then on the algebraic tensor product
$\Hf^{\otimes p}\otimes \Hf^{\otimes q}$ we set
\begin{align*}
\hat{C}_{g_1,g_2}&(v_1 \otimes \dots \otimes v_p,\,
w_1 \otimes \dots \otimes w_q)\\
&=
\sum_{\substack{p \geq i \geq 1\\ q \geq j \geq 1}}
{C}_{g_1,g_2}(v_i,w_j) (v_1 \otimes \dots \hat{v_i} \dots \otimes v_p)
\otimes
(w_1 \otimes \dots \hat{w_j} \dots \otimes w_q),
\end{align*}
for any $v_1 \otimes \dots \otimes v_p \in \Hf^{\otimes p}$ and
$w_1 \otimes \dots \otimes w_q \in \Hf^{\otimes q}$.

This operator is bounded and induces
a bounded linear map on the completions, as in \eqref{eq_hat_C_def}.
If $p=0$ or $q=0$, we set $\hat{C}_{g_1,g_2}=0$.
Similarly to Lemma \ref{lem_cont_sym}, 
the restriction of $\hat{C}_{g_1,g_2}$ induces a bounded linear map from $\Sym^p \Hf \hotimes \Sym^q(\Hf) \rightarrow \Sym^{p-1}\Hf \hotimes \Sym^{q-1}\Hf$ (see \cite[Lemma 2.32]{MLeft}).
Since $\hat{C}_{g_1,g_2}$ is a degree-lowering operator, for each $p,q \geq 0$ the action of
\begin{align*}
\exp(\hat{C}_{g_1,g_2}) = \sum_{k \geq 0} \frac{1}{k!} \hat{C}_{g_1,g_2}^k
\end{align*}
on $\Hf^p \otimes \Hf^q$ is a finite sum.
Therefore, $\exp(\hat{C}_{g_1,g_2}):A \otimes A \rightarrow A$ is well defined, and its restriction
to $\Hf^p \hotimes \Hf^q$ is a bounded linear operator.

%$\hat{C}_{g_1,g_2}$は次数を下げる作用素であるから各$p,q \geq 0$に対して、$\Sym^p \Hf \otimes \Sym^q \Hf$への
%\begin{align*}
%\exp(\hat{C}_{g_1,g_2}) = \sum_{k \geq 0} \frac{1}{k!} \hat{C}_{g_1,g_2}^k
%\end{align*}
%の作用は有限和になる。よって$\exp(\hat{C}_{g_1,g_2}):V \otimes V \rightarrow V$は well-defined であり、その
%$H^p \otimes H^q$への制限は bounded linear map である。

%For $(g_1,g_2) \in \CE(2)$, 線形写像
%\begin{align*}
%\hat{C}_{g_1,g_2}: V \otimes V \rightarrow V
%\end{align*}
%をその$\Sym^p \Hf \otimes \Sym^q \Hf$ with $p,q \geq 1$への制限が linear map $\Sym^p \Hf \otimes \Sym^p \Hf \rightarrow \Sym^{p-1}\Hf \otimes \Sym^{q-1} \Hf$で
%\begin{align*}
%\hat{C}_{g_1,g_2}(v_1 \otimes \dots \otimes v_p, w_1 \otimes \dots \otimes w_q)
%= \sum_{\substack{p \geq i \geq 1\\q \geq j \geq 1}} C_{g_1,g_2}(v_i,w_j)(v_1 \otimes \dots \hat{v_i} \dots \otimes v_p) \otimes  (w_1 \otimes \dots 
%\hat{w_j} \dots \otimes w_q)
%\end{align*}
%for any $v_1 \otimes \dots \otimes v_p \in \Sym^p \Hf$ and $w_1 \otimes \dots \otimes w_q \in \Sym^q \Hf$
%であり、$\Sym^p \Hf \otimes \Sym^q \Hf$ with $p$ または$q$がゼロの空間への制限は$0$写像であるように定義する。
%
%によって定まるものになるように定義する.
%Here $v_1 \otimes \dots \hat{v_i} \dots \otimes v_p$ denotes the vector obtained by omitting $v_i$.
%また$p$または$q$が$0$の場合は$\C=\Sym^0 \Hf$として、$\hat{C}_{g_1,g_2}=0$によって射を定める。

%ただし$\hat{v_i}$は$v_i$をのぞいたテンソルを表す。$\hat{C}_{g_1,g_2}$は有界な作用素の有限和であるから有界である。
%
%\begin{lem}\label{lem_cont_hat_cov}
%For any $h \in \fS$ and $a_1 \in \Sym^p H$, $a_2 \in \Sym^q H$,
%\begin{align*}
%(h \otimes h)\hat{C}_{g_1,g_2}(g_1a_1,g_2a_2) = \hat{C}_{hg_1,hg_2}(hg_1a_1,hg_2a_2).
%\end{align*}
%\end{lem}

Set $\rho_1=\rho: \CEh(1) \rightarrow \End A$.
Let $(g_1,\dots,g_n) \in \CEh(n)$ for $n \geq 2$.
For $i,j \in \{1,\dots,n\}$ with $i\neq j$, we denote by
$\exp(\hat{C}_{g_i,g_j}^{i,j})$ the operator on $A^{\otimes n}$ obtained by letting
$\exp(\hat{C}_{g_i,g_j})$ act on the $i$-th and $j$-th tensor factors.
Then, the operators $\exp(\hat{C}_{g_i,g_j}^{i,j})$ and
$\exp(\hat{C}_{g_k,g_l}^{k,l})$ commute with each other.
The operator
$\exp(\sum_{n \geq i > j \geq 1} \hat{C}_{g_i,g_j}^{i,j})$
is the composition of all these mutually commuting operators.
We also define $\bs:A^{\otimes n} \rightarrow A$ so that its restriction to
$\Sym^{p_1}\Hf \otimes \cdots \otimes\Sym^{p_n} \Hf$ coincides with the restriction of $\bs^{p_1+\dots+p_n}:\Hf^{\hotimes p_1+\dots+p_n} \rightarrow \Sym^{p_1+\dots+p_n}(\Hf)$.

We define a linear map
$\m_{g_1,\dots,g_n}: A^{\otimes n} \rightarrow A$
by the following composition:
\begin{align*}
A^{\otimes n}
\overset{\exp(\sum_{n \geq i > j \geq 1} \hat{C}_{g_i,g_j}^{i,j})}{\longrightarrow}
A^{\otimes n}
\overset{\rho_1(g_1)\otimes \dots \otimes \rho_1(g_n)}{\longrightarrow}
A^{\otimes n}
\overset{\bs}{\longrightarrow}
A.
\end{align*}
Finally, for $*=(\emptyset \rightarrow \bD) \in \CE(0)$, we define
\begin{align*}
\rho_0(*):\C \rightarrow A,\quad 1 \mapsto \va.
\end{align*}
\begin{thm}\label{thm_CF}
$(A,\m,\va)$ is a $\CEh$-algebra in $\Vect$ with the Hilbert space filtration $\Hf^k = \bigoplus_{p = 0}^k \Sym^p \Hf$ \cite[Definition 1.27]{MLeft}. Moreover, 
for any $k \geq 0$, the monoid representation $\rho_1: \CEh(1) \rightarrow \bB(\Hf^k)$ satisfies the conditions (U) and (D) in \cite[Definition 1.32]{MLeft}.
\end{thm}
\begin{proof}
We first verify that $A$ is a $\CEh$-algebra (see 
\cite[Definition 1.10]{MLeft}).
The $S_n$-invariance is clear. The equality $\rho_1(\id_\bD)=\id_A$ also follows from the definition.
Let $(g_1,\dots,g_n,g_{n+1}) \in \CEh(n+1)$ and $(h_1,\dots,h_m) \in \CEh(m)$.
We will show that
\begin{align}
\m_{g_1,\dots,g_n,g_{n+1}h_1,\dots,g_{n+1}h_m}=
\m_{(g_1,\dots,g_n,g_{n+1})}\circ_{n+1} \m_{(h_1,\dots,h_m)}.\label{eq_def_of_operad}
\end{align}
For simplicity, we write $\hat{C}_{g_i,g_j}^{i,j}$ as $\hat{C}_{g_i,g_j}$, and we write $\rho_1(g_i)$ (resp. $\rho_\cl(g_i)$) simply as $g_i$ (resp. $g_i^\cl$).
The case $m=0$ in which \eqref{eq_def_of_operad} holds is clear, since $\hat{C}$ and $\rho_1$ act trivially on $\va$.
If $m \geq 1$, then by Proposition \ref{prop_contraction_covariance} we have
\begin{align*}
&\m_{g_1,\dots,g_n,g_{n+1}h_1,\dots,g_{n+1}h_m}\\
&=\bs\left((g_1 \otimes \dots \otimes g_n \otimes g_{n+1}h_1 \otimes \dots g_{n+1} h_m)
e^{\sum \hat{C}_{g_i,g_j}}
e^{\sum \hat{C}_{g_i,g_{n+1}h_j}}
e^{\sum \hat{C}_{g_{n+1}h_i,g_{n+1}h_j}}\right)\\
&=\bs\left((g_1 \otimes \dots \otimes g_n \otimes g_{n+1}h_1 \otimes \dots g_{n+1} h_m)
e^{\sum \hat{C}_{g_i,g_j}}
e^{\sum \hat{C}_{g_i,g_{n+1}}(\id\otimes h_j^\cl)}
e^{\sum \hat{C}_{h_i,h_j}+\pa_{g_{n+1}}(h_i^\cl\otimes h_j^\cl)}\right)\\
&=\bs\left((g_1 \otimes \dots \otimes g_n \otimes g_{n+1} \otimes \dots \otimes g_{n+1})
(\id \otimes h_1 \otimes \dots \otimes h_m)
e^{\sum \hat{C}_{g_i,g_j}}
e^{\sum \hat{C}_{g_i,g_{n+1}}(\id\otimes h_j^\cl)}
e^{\sum \hat{C}_{h_i,h_j}+\pa_{g_{n+1}}(h_i^\cl\otimes h_j^\cl)}\right)\\
&=\bs\left((g_1 \otimes \dots \otimes g_n \otimes g_{n+1} \otimes \dots \otimes g_{n+1})
e^{\sum \hat{C}_{g_i,g_j}}
e^{\sum \hat{C}_{g_i,g_{n+1}}}e^{\sum \pa_{g_{n+1}}^{ij}}
(\id \otimes h_1 \otimes \dots \otimes h_m)
e^{\sum \hat{C}_{h_i,h_j}}\right)\\
&=\bs\left((g_1 \otimes \dots \otimes g_n \otimes g_{n+1} \otimes \dots \otimes g_{n+1})e^{\sum \pa_{g_{n+1}}^{ij}}
e^{\sum \hat{C}_{g_i,g_j}}
e^{\sum \hat{C}_{g_i,g_{n+1}}}\bs
(\id \otimes h_1 \otimes \dots \otimes h_m)
e^{\sum \hat{C}_{h_i,h_j}}\right)
\end{align*}
%ここで$\pa_{g_{n+1}}^{ij}$は$h_i,h_j$が作用する$A$の成分上の contraction であり、最後の行ではwe have used the fact that, for a linear operator symmetric in its inputs, one may insert the symmetrization operator $\bs$.
%$g_{n+1}=\rho_\cl(g_{n+1})e^{\pa_{g_{n+1}}}$であるから$g_{n+1}\otimes \cdots \otimes g_{n+1}$の$A^{m}$への作用は
%異なる成分の間の縮約$\pa_{g_{n+1}}$が存在しないことに注意すると、上記の式は
Here $\pa_{g_{n+1}}^{ij}$ denotes the contraction on the $A$-components on which $h_i$ and $h_j$ act. 
In the last line, we have used the fact that for a linear operator symmetric in its inputs, one may insert the symmetrization operator $\bs$.

The operator 
$g_{n+1}\otimes \cdots \otimes g_{n+1}$ acts on $A^{m}$ without inducing any contractions $\pa_{g_{n+1}}$ between different tensor factors. 
The above expression simplifies to
\begin{align*}
&=\bs\left((g_1 \otimes \dots \otimes g_n \otimes {g}_{n+1})e^{\sum \hat{C}_{g_i,g_j}}
e^{\sum \hat{C}_{g_i,g_{n+1}}}\bs
(\id \otimes h_1 \otimes \dots \otimes h_m)
e^{\sum \hat{C}_{h_i,h_j}}\right)\\
&=\rho_{g_1,\dots,g_n,g_{n+1}}\circ_{n+1} \rho_{h_1,\dots,h_m}.
\end{align*}
By Proposition \ref{prop_monoid_classical} and Proposition \ref{prop_cocycle_vanish},
\begin{align*}
\rho_1|_{\fG}=\rho_\cl|_{\fG}: \fG \rightarrow \bB(\Hf^k)
\end{align*}
is a strongly continuous unitary representation.
Moreover, by Lemma~\ref{lem_dilation} and Proposition~\ref{prop_cocycle_vanish},
\begin{align*}
(0,1] \rightarrow \bB(\Hf^k),\quad r\mapsto \rho_\cl(B_{0,r})
\end{align*}
is a strongly continuous, self-adjoint, contractive representation.
A straightforward computation shows that its trace is
\begin{align*}
\mathrm{tr}|_{A} \rho_\cl(B_{0,r})
= \prod_{n > 0} \frac{1}{(1-r^{n})},
\end{align*}
which converges absolutely for $0<r<1$.
Hence condition~(D) of \cite[Definition~1.32]{MLeft} follows.
%By Proposition \ref{prop_monoid_classical} and Proposition \ref{prop_cocycle_vanish}より、
%\begin{align*}
%\rho_1|_{\fG}=\rho_\cl|_{\fG}: \fG \rightarrow B(\Hf^k)
%\end{align*}
%は強連続な unitary 表現である。
%また Lemma \ref{lem_dilation} and Proposition \ref{prop_cocycle_vanish}より、
%\begin{align*}
%(0,1] \rightarrow B(\Hf^k),\quad r\mapsto \rho_\cl(r^D)
%\end{align*}
%は強連続, self-adjoint, contraction representation であり、簡単な計算によりその trace は
%\begin{align*}
%\mathrm{tr}|_{A} \rho_\cl(q^D)
%= \prod_{n > 0} \frac{1}{(1-q^{n})}\Bigl|_{|q|<1}
%\end{align*}
%converges absolutely for $|q|<1$. よって \cite[Definition 1.32]{MLeft}の(D)が従う。
\end{proof}

Define a linear isomorphism $\cN:A \rightarrow A$ by
\begin{align*}
\cN|_{\Sym^p \Hf} = \sqrt{p!}\,\id_{\Sym^p \Hf}.
%\bn^{p} &=  \bs^p,\\
%\bn^{p_1,\dots,p_n} &= \sqrt{ \frac{p_1! \cdots p_n!}{(p_1+\dots+p_n)!}} \bs^{p_1,\dots,p_n}.
\end{align*}
Then $\cN|_{H^k}$ is a bounded linear map for each $k \geq 0$, while $\cN$ itself is not bounded.
Using this linear isomorphism, we can endow $A$ with a $\CEh$-algebra structure by setting
\begin{align}
\rho_{g_1,\dots,g_n}^\cN
= \cN \circ \rho_{g_1,\dots,g_n} \circ (\cN^{-1}\otimes \cdots \otimes \cN^{-1}).\label{eq_sqrt}
\end{align}
%The $\CE$-algebra $(A,\rho^\cN,\va)$ を考える重要性は affine Heisenberg vertex algebra との比較で明らかになる (see also Remark in \cite{MLeft}).
The importance of considering the normalized $\CEh$-algebra $(A,\rho^\cN,\va)$ becomes apparent when comparing it with the affine Heisenberg vertex algebra (see also \cite[Remark 2.35]{MLeft}).
By Lemma \ref{lem_CE2}, a $\CEh$-algebra determines a $\CE$-algebra by restriction.
By Theorem \ref{thm_CF} and \cite[Theorem 3.5]{MLeft}, we have:
\begin{thm}\label{thm_CF_algebra}
The $\CE$-algebra $\Sym A^2(\bD)$ defines a symmetric monoidal functor $A:\Disk \rightarrow  \Ind\,\Hilb$.
Moreover, its left Kan extension along $\Disk \hookrightarrow \Embc$
defines a symmetric monoidal functor
\begin{align*}
%H(\bullet, \Sym A^2(\bD))=
\mathrm{Lan}_{\Sym A^2(\bD)}:\Embc \rightarrow \Ind\,\Hilb.
\end{align*}
\end{thm}
Thus we obtain invariants of two-dimensional Riemannian manifolds with coefficients in $\Sym A^2(\bD)$.
%よって$\Sym A^2(\bD)$を係数とする二次元リーマン多様体の不変量を得る。

Finally, we end this section by examining the complex conjugate of the $\CE$-algebra $\Sym A^2(\bD)$.
%An important quantity characterizing a $\CE$-algebra is the two-point correlation function, defined by
%$(\va, \rho_{B_{a,r},B_{b,s}}(e_0(z),e_0(z)))$
%for $a,b\in \bD$ and $1>r,s>0$, using 
%(the reason for considering this quantity will become clear in Section~\ref{sec_vertex}).
The following holds:
\begin{prop}\label{prop_conjugate}
For the $\CE$-algebra $\Sym A^2(\bD)$ and its complex conjugate,
\begin{align}
(\va, \rho_{B_{a,r},B_{b,s}}(e_0(z),e_0(z))) &= \frac{rs}{({a}-{b})^2},\label{eq_two_A}\\
(\va, \rho_{B_{a,r},B_{b,s}}^\J(e_0(z),e_0(z))) &= \frac{rs}{(\bar{a}-\bar{b})^2}\label{eq_two_A2}
\end{align}
for any $(B_{a,r},B_{b,s}) \in \CE(2)$ in Proposition \ref{prop_contraction_trans}.
\end{prop}
\begin{proof}
By Proposition \ref{prop_contraction_trans}, we have
\begin{align*}
(\va, \rho_{B_{a,r},B_{b,s}}(e_0(z),e_0(z))) 
&= C_{B_{a,r},B_{b,s}}(e_0(z),e_0(z)) = \frac{rs}{(a-b)^2},
\end{align*}
and by the definition of the complex conjugate (Definition~\ref{def_conjugate}),
\begin{align*}
(\va, \rho_{B_{a,r},B_{b,s}}^\J(e_0(z),e_0(z))) 
&= (\va, \rho_{\J B_{a,r},\J B_{b,s}}(e_0(z),e_0(z))) \\
&= (\va, \rho_{B_{\ba,r}, B_{\bar{b},s}}(e_0(z),e_0(z))) 
= \frac{rs}{(\bar{a}-\bar{b})^2}.
\end{align*}
\end{proof}

Readers familiar with vertex operator algebras will recognize that \eqref{eq_two_A}
coincides with the two-point correlation function of the affine Heisenberg vertex algebra, while \eqref{eq_two_A2}
coincides with that of 
%its complex conjugate, namely the 
anti-holomorphic affine Heisenberg vertex algebra \cite{M1}.
We examine this correspondence in more detail in the next section.

%
%
%最後に$\CE$代数$\Sym A^2(\bD)$の complex conjugate を調べてこの章を終える.
%$\CE$代数を特徴づける重要な量は二点相関関数であり、命題\ref{prop_contraction_trans}の$a,b\in \bD$ and $1>r,s>0$を用いて$(\va, \rho_{B_{a,r},B_{b,s}}(e_0(z),e_0(z)))$によって定義される (なぜこの量を考えるかはSection \ref{sec_vertex}で明らかになる)。
%このとき次が成り立つ:
%\begin{prop}\label{prop_conjugate}
%$\CE$代数$\Sym A^2(\bD)$およびその複素共役に対して,
%\begin{align}
%(\va, \rho_{B_{a,r},B_{b,s}}(e_0(z),e_0(z))) &= \frac{rs}{({a}-{b})^2},\label{eq_two_A}\\
%(\va, \rho_{B_{a,r},B_{b,s}}^\J(e_0(z),e_0(z))) &= \frac{rs}{(\bar{a}-\bar{b})^2}.\label{eq_two_A2}
%\end{align}
%\end{prop}
%\begin{proof}
%By Proposition \ref{prop_contraction_trans}, we have
%\begin{align*}
%(\va, \rho_{B_{a,r},B_{b,s}}(e_0(z),e_0(z))) 
%&= C_{B_{a,r},B_{b,s}}(e_0(z),e_0(z))) = \frac{rs}{(a-b)^2}
%\end{align*}
%and by definition of the complex conjugate, Definition \ref{def_conjugate},
%\begin{align*}
%(\va, \rho_{B_{a,r},B_{b,s}}^\J(e_0(z),e_0(z))) 
%&= (\va, \rho_{\J B_{a,r},\J B_{b,s}}(e_0(z),e_0(z))) \\
%&= (\va, \rho_{B_{\ba,r},\J B_{\bar{b},s}}(e_0(z),e_0(z))) = \frac{rs}{(\bar{a}-\bar{b})^2}.
%\end{align*}
%\end{proof}
%
%頂点作用素代数になじみがある読者は、\eqref{eq_two_A}が affine Heisenberg vertex algebra の二点相関関数に、\eqref{eq_two_A2}
%が その complex conjugate である 反正則な affine Heisenberg vertex algebra \cite{M1}の二点相関関数に等しいことが分かるだろう。
%こうした対応を次の章でより詳しく調べる。

\section{Relation with affine Heisenberg vertex algebra}\label{sec_relation}
%この章では $\CE$代数と vertex operator algebra の関係を調べる。我々は一般に unitary full vertex operator algebra に対して、その ind Hilbert space としての完備化として $\CE$代数が構成できると予想している。この章ではこうした対応関係を、 affine Heisenberg vertex algebra と $\Sym A^2(\bD)$の対応を見ることで確立する。Section \ref{sec_vertex}では affine Heisenberg vertex algebra の unitary 構造に関する基本的な事項をまとめる。
%Section \ref{sec_dictionary}では、affine Heisenberg vertex algebra と Bergman 空間の間のヒルベルト空間としての isometric isomorphism を用いて、
%代数構造を比較する。
%In this section, we study the relationship between $\CEh$-algebras and vertex operator algebras.
In this section, we establish the correspondence between the affine Heisenberg vertex operator algebra $M(0)$ and
the $\CEh$-algebra $\Sym A^2(\bD)$.
Based on this result, we expect that for any unitary full vertex operator algebra, one can construct a $\CEh$-algebra as its ind-Hilbert space completion.

In Section~\ref{sec_vertex}, we review the basic facts about the unitary structure of the affine Heisenberg vertex operator algebra.
In Section~\ref{sec_dictionary}, we give an isometric isomorphism of Hilbert spaces between the completion of the unitary affine Heisenberg vertex algebra and $\Sym A^2(\bD)$ and compare their algebraic structures.

\subsection{Review on affine Heisenberg vertex algebra}\label{sec_vertex}
We begin by briefly recalling the definition of the affine Heisenberg vertex algebra.
For the definition of a vertex operator algebra and its basic properties, see \cite{FLM,FB,FHL,LL,Kac}.
Let 
\begin{align*}
\fh = \bigoplus_{n \in \Z}\C h \otimes t^n \oplus \C c
\end{align*}
be a Lie algebra which satisfies the following commutator relations:
\begin{align*}
[h \otimes t^n,h \otimes t^m]=n\,\delta_{n+m,0}c
\end{align*}
and $c$ is a center.
Let $M(0)$ be the induced representation of $\fh$ given by
\begin{align*}
h \otimes t^n\cdot \va &=0, \qquad \text{ for any }n \geq 0,\\
c \cdot \va &=\va.
\end{align*}
Denote by $h(n) \in \End M(0)$ the action of $h \otimes t^n$ on $M(0)$.
As a vector space, $M(0)$ is spanned by the basis $\{h(-k_1)h(-k_2)\dots h(-k_n)\va\}_{n \geq 0, k_1 \geq k_2 \geq \dots \geq k_n \geq 1}$. Then, one can define a unique vertex algebra structure on $M(0)$ such that $\va \in M(0)$ is the vacuum vector and
 the linear map
\begin{align*}
Y(-,z):M(0) &\rightarrow \End M(0)[[z^\pm]],\\
a &\mapsto Y(a,z) = \sum_{n\in\Z}a(n)z^{-n-1},
\end{align*}
called a {\it vertex operator}, satisfies
%$(M(0),Y(-,z))$ is a vertex algebra with the vacuum vector $\va \in M(0)$ and 
\begin{align*}
Y(h(-1)\va,z) = \sum_{n \in \Z} h(n)z^{-n-1} \in \End M(0)[[z^\pm]].
\end{align*}
Set
\begin{align}
h^+(z) &=\sum_{n \geq 0}h(n)z^{-n-1},\\
h^-(z)&=\sum_{n \geq 0} h(-n-1)z^{n}.
\end{align}
This gives a decomposition of the vertex operator $Y(h(-1)\va,z)$ into its creation operator $h^-(z)$ and annihilation operator $h^+(z)$.
%これは頂点作用素$Y(h(-1)\va,z)$の生成演算子$h^-(z)$と消滅演算子$h^+(z)$への分解を与える。
Set $\om = \ft h(-1)h(-1)\va$ and define $L(n)\in \End M(0)$ by $Y(\om,z) = \sum_{n \in \Z} L(n)z^{-n-2}$.
Then,
\begin{align}
L(n) = \ft \left(\sum_{k \geq 0} h(n-k)h(k) + \sum_{k<0} h(k)h(n-k)\right)\label{eq_Ln}
\end{align}
satisfies the Virasoro commutator relation:
\begin{align*}
[L(n),L(m)]=(n-m) L(n+m) + \frac{n^3-n}{12} \delta_{n+m,0}.
\end{align*}
The pair of the vertex algebra $M(0)$ and $\om$ is called the {\it affine Heisenberg vertex operator algebra} \cite{FLM}.
For $p>0$, let $M_p(0)$ be the subspace of $M(0)$ spanned by
\begin{align*}
M_p(0) = \{h(-k_1)\dots h(-k_p)\va\}_{k_1 \geq \dots \geq k_p \geq 1}
\end{align*}
and set $M_0(0) = \C\va$. 
In physics, $p$ is referred to as the {\it number of particles} in free field theory.
%$p$は物理において自由場理論の粒子数と呼ばれる。
Set
\begin{align*}
M(0)^p = \bigoplus_{k \geq 0}^p M_k(0),
\end{align*}
which is a filtration on $M(0)$.
The space $M(0)$ carries the structure of a $\Z_{\geq 0}$-graded vector space via the action of $L(0)$:
\begin{align*}
M(0)_n = \{v \in M(0) \mid L(0)v = n v\}.
\end{align*}
Here 
$L(0)h(-k_1)\dots h(-k_p)\va = (k_1+\dots+k_p)h(-k_1)\dots h(-k_p)\va$.
In particular, $M(0)_0 = \C \va$ and $M(0)_1 = \C h(-1)\va$.
By \eqref{eq_Ln} and $h(0)=0$, we have:
%$M(0)$は$L(0)$の作用によって$\Z_{\geq 0}$-graded vector space の構造を持つ:
%\begin{align*}
%M(0)_n = \{v \in M(0) \mid L(0)v =n v\},
%\end{align*}
%ここで $L(0)h(-k_1)\dots h(-k_p)\va = (k_1+\dots+k_p)h(-k_1)\dots h(-k_p)\va$. とくに$M(0)_0 =\R \va$ and $M(0)_1 = \R h(-1)\va$である。
%\eqref{eq_Ln}と$h(0)=0$より、we have:
\begin{prop}\label{prop_particle_number}
Let $p \geq 0$. Then,
\begin{align*}
L(n) M_p(0) \subset 
\begin{cases}
M_p(0) & (\text{ if }n =-1,0,1)\\
M_p(0)\oplus M_{p-2}(0) & (\text{ if }n \geq 2)\\
M_p(0)\oplus M_{p+2}(0) & (\text{ if }n \leq -2).
\end{cases}
\end{align*}
%If $n=\{-1,0,1\}$, then $L(n)$ は $M_p(0)$を$M_p(0)$に送る。一方で、$n \geq 2$ (resp. )ならば、$L(n)$は$M_p(0)$ を$M_p(0) \oplus M_{p-2}(0)$ に送り、$n \leq -2$ならば$$
\end{prop}

\begin{rem}
%命題\ref{prop_particle_number}にあるように一般にVirasoro の作用は粒子数を保存しない。
%$n \geq 2$ における二次微分 \eqref{eq_Ln}は調和コサイクルによる二次微分 \eqref{eq_pa_phi_cont}に対応する。
%また $\{L(n)\}_{n=-1,0,1}$が粒子数を保存することは$\{L(n)\}_{n=-1,0,1}$が$\mathrm{PSL}_2\C =\Conf^+(S^2)$を生成し、このような変換に対しては Proposition \ref{prop_cocycle_vanish}によって、$\rho_\cl =\rho_1$となることが対応する。こうした事実の
%distribution の正規化と Green 関数の境界条件からくる解釈は\cite{Mfactorization}を参照されたい。
As stated in Proposition \ref{prop_particle_number}, the Virasoro algebra do not preserve the particle number in general.
In fact, for $n \geq 2$, the second-order differential operator \eqref{eq_Ln} corresponds to the second-order differential operator \eqref{eq_pa_phi_cont} arising from the harmonic cocycle.
Moreover, the fact that $\{L(n)\}_{n=-1,0,1}$ preserves the particle number corresponds to the fact that $\{L(n)\}_{n=-1,0,1}$ generates $\mathrm{PSL}_2\C = \Conf^+(S^2)$, and for such transformations Proposition~\ref{prop_cocycle_vanish} implies that $\rho_1=\rho_\cl$.
For an interpretation of these in terms of the boundary conditions of the Green's function, see \cite{Mfactorization}.
\end{rem}

%この補題は後で見るように、調和関数によって定義される二次微分
%\begin{align*}
%\pa_{H_\phi}:P^p(S^\cl(\bD)) \rightarrow P^{p-2}(S^\cl(\bD))
%\end{align*}
%と関係している。実際、$\Conf(S^2)\cong \mathrm{PSL}_2\C$は$L(-1),L(0),L(1)$で生成されるため、$\phi \in \Conf(S^2)$のとき、$H_\phi=0$となることが、補題の粒子数の保存と対応する。

It is well-known that $M(0)$ admits a unique symmetric bilinear form $(-,-):M(0) \otimes M(0) \rightarrow \C$ such that: (see for example \cite{FHL,Li})
\begin{itemize}
\item
$(\va,\va)=1$.
\item
For any $a,b,c \in M(0)$,
\begin{align}
(Y(a,z)b,c) = (b,Y(e^{zL(1)}(-z^{-2})^{L(0)}a,z^{-1})c)
\label{eq_def_inv}
\end{align}
holds as formal power series.
\end{itemize}
Define an anti-linear map $\theta:M(0) \rightarrow M(0)$ by
\begin{align*}
\theta(h(-k_1)\dots h(-k_n)\va)= (-1)^n h(-k_1)\dots h(-k_n)\va
\end{align*}
for any $k_1 \geq \dots \geq k_n \geq 1$.
Then, $\theta$ is a vertex operator algebra anti-linear automorphism
 which satisfies $\theta^2=\id_{M(0)}$.
Define sesquilinear form $(-,-)_M :M(0) \otimes M(0) \rightarrow \C$ defined by 
$( a, b )_M = (\theta(a),b)$ for $a,b \in M(0)$.
Then, $(-,-)_M$ is positive-definite \cite{DL}.

\begin{rem}\label{rem_positive}
A (full) vertex operator algebra equipped with a nondegenerate invariant bilinear form \eqref{eq_def_inv} and an anti-linear involution $\theta$ such that the associated sesquilinear form is positive definite is called a unitary (full) vertex operator algebra \cite{DL,CKLW,AMT}.
The affine Heisenberg vertex algebra is an example of a unitary vertex operator algebra.
%このような非退化な不変双線形形式\eqref{eq_def_inv}と involution $\theta$を持ち、それによって定義される sesqui-linear form が正定値になるようなVOA はunitary vertex operator algebra と呼ばれる\cite{DL,CKLW}. affine Heisenberg vertex algebra は unitary VOA の一例である。
\end{rem}

%We simply write $M(0)$ for $M(0)_\C$.
Let $\overline{M(0)^p}$ denote the Hilbert space completion of $M(0)^p$ with respect to the inner product $(-,-)_M$.
Note that by \eqref{eq_def_inv}, $M_p(0)$ and $M_q(0)$ are mutually orthogonal if $p \neq q$.
The following proposition can be generalized into any unitary full vertex operator algebra (see  also \cite{AMT}).
\begin{prop}\label{prop_vertex_conv}
Let $v \in M(0)^p$ and $w \in M(0)^q$.
Then, for any $\zeta \in \bD \setminus \{0\}$, the formal power series
\begin{align*}
Y(v,z)w = \sum_{n \in \Z} v(n)w z^{-n-1}
\end{align*}
converges in norm $(-,-)_M$ upon substituting $z=\zeta$, and hence defines an element of the Hilbert space $\overline{M(0)^{p+q}}$, that is,
$\sum_{n \in \Z} \norm{v(n)w}_M^2 |\zeta|^{-2n-2} <\infty$ holds.
\end{prop}

\begin{proof}
Recall that a vector $v \in M(0)$ is called {\it quasi-primary} of weight $\Delta \in \Z_{\geq 0}$ if $L(1)v=0$ and $L(0)v=\Delta v$.
The underlying vector space of any unitary VOA is spanned by $\{L(-1)^n v\}_{n \geq 0}$, where $v$ ranges over all quasi-primary vectors (see  \cite{DLM}).
By linearity, it suffices to prove the proposition for 
$Y(\frac{1}{l!}L(-1)^l v,z)w$ with $l \geq 0$, $v \in M(0)_\Delta$, and $w \in M(0)_{\Delta'}$ such that $v$ is quasi-primary.
Since
\begin{align*}
Y\left(\frac{1}{l!}L(-1)^l v,z\right)w
= \frac{\pa_z^l}{l!} Y(v,z)w
= \sum_{n \in \Z} \binom{-n-1}{l} v(n)w z^{-n-1-l},
\end{align*}
and $v(n)w \in M(0)_{\Delta+\Delta'-n-1}$,
it suffices to show that
\begin{align*}
\sum_{n \in \Z} \left(\binom{-n-1}{l}\right)^2 
\norm{v(n)w}_M^2 
|\zeta|^{-2n-2l-2} < \infty.
\end{align*}
Let $z_1,z_2$ be formal variables.
By \eqref{eq_def_inv},
\begin{align*}
&\frac{1}{(l!)^2}(Y(L(-1)^l v,z_1)w, Y(L(-1)^l v,z_2)w)_M\\
&=\frac{1}{(l!)^2} \pa_{z_1}^l \pa_{z_2}^l (Y(v,z_1)w, Y(v,z_2)w)_M\\
&=\frac{1}{(l!)^2} \pa_{z_1}^l \pa_{z_2}^l 
(w,Y( e^{L(1)z_1}(-z_1^{-2})^{L(0)} \theta(v),z_1^{-1}) Y(v,z_2)w)_M\\
&=\frac{1}{(l!)^2} \pa_{z_1}^l \pa_{z_2}^l 
(-z_1^{-2})^{\Delta}
(w,Y(\theta(v),z_1^{-1}) Y(v,z_2)w)_M\\
&=\frac{1}{l!}(-1)^{\Delta}
\sum_{n \in \Z}\binom{-n-1}{l}
\pa_{z_1}^l  z_1^{-2\Delta} z_2^{-n-l-1}
(w,Y(\theta(v),z_1^{-1}) v(n)w)_M\\
&=(-1)^{\Delta}
\sum_{n \in \Z}\left(\binom{-n-1}{l}\right)^2
z_1^{-n-l-1} z_2^{-n-l-1}
(w,\theta(v)(2\Delta-n-2)v(n)w)_M.
\end{align*}
It is a well-known fact that this formal power series coincides with the expansion of a rational polynomial in $\C[z_1^\pm,z_2^\pm,(1-z_1z_2)^{-1}]$ and converges absolutely in the region $|z_1|^{-1}>|z_2|$ (see, for example, \cite{FB,LL}).
Taking into account the anti-linearity in the first variable of $(-,-)_M$, we obtain
\begin{align*}
&\frac{1}{(l!)^2}(Y(L(-1)^l v,z_1)w, Y(L(-1)^l v,z_2)w)_M\Bigl|_{z_1=z_2=\zeta}\\
&=(-1)^{\Delta}
\sum_{n \in \Z}\left(\binom{-n-1}{l}\right)^2
z_1^{-n-l-1} z_2^{-n-l-1}
(w,\theta(v)(2\Delta-n-2)v(n)w)_M \Bigl|_{z_1=\overline{\zeta},z_2=\zeta}\\
&=(-1)^{\Delta}
\sum_{n \in \Z}\left(\binom{-n-1}{l}\right)^2
|\zeta|^{-2n-2l-2}
(w,\theta(v)(2\Delta-n-2)v(n)w)_M,
\end{align*}
which converges absolutely since $|\zeta|^{-1}>1>|\zeta|$.
Comparing coefficients, we obtain
\begin{align*}
\norm{v(n)w}_M^2
= (-1)^{\Delta}(w,\theta(v)(2\Delta-n-2)v(n)w)_M,
\end{align*}
and the proposition follows.
\end{proof}

\subsection{Relation between $\CE$-algebra and affine Heisenberg vertex algebra}\label{sec_dictionary}
In this section, we discuss the relationship between $M(0)$ and the $\CEh$-algebra $\Sym A^2(\bD)$.

Note that by an easy computation, we have
\begin{align*}
(h(-k-1)^n \va,h(-k-1)^n \va)_M= n! (k+1)^n 
\end{align*}
for any $k \geq 0$ and $n>0$.
In particular, $\norm{h(-k-1)\va}_M= \sqrt{k+1}$, and we would like to identify this vector with $\sqrt{k+1}e_k(z) = (k+1)z^k \in A^2(\bD)$.
Set
\begin{align*}
h_j(z) = \sqrt{j+1}e_j(z) \in A^2(\bD).
\end{align*}

Motivated by this observation, for any $p \geq 0$, define a linear map
\begin{align*}
\Psi: M_p(0) \rightarrow \Sym^p(A^2(\bD))
\end{align*}
by, for any $n \geq 1$, $p= k_1+\dots+k_n$, and $i_1>i_2>\dots>i_n \geq 0$,
\begin{align*}
\Psi(h(-i_1-1)^{k_1}\dots h(-i_n-1)^{k_n})
= \sqrt{p!}\,\bs^p\bigl((h_{i_1}(z))^{\otimes k_1}\otimes \cdots \otimes (h_{i_n}(z))^{\otimes k_n}\bigr).
\end{align*}
Here $\bs^p$ denotes the symmetrization operator (the orthogonal projection \eqref{eq_symetrizer})
$A^2(\bD)^{\hotimes p} \rightarrow \Sym^p A^2(\bD)$.
Taking the direct sum over $p$, we obtain a linear map
\begin{align*}
\Psi: M(0) \rightarrow \Sym(A^2(\bD)).
\end{align*}

%
%この章では$M(0)$と$\CE$代数$\Sym A^2(\bD)$の関係を議論する。
%
%Note that by an easy computation, we have
%\begin{align*}
%(h(-k-1)^n \va,h(-k-1)^n \va)= n! (k+1)^n 
%\end{align*}
%for any $k \geq 0$ and $n>0$. とくに$\norm{h(-k-1)\va}= \sqrt{k+1}$であり、我々はこのベクトルを$\sqrt{k+1}e_k(z) = (k+1)z^k \in A^2(\bD)$と同定したい。そこでSet
%\begin{align*}
%h_j(z) = \sqrt{j+1}e_j(z) \in A^2(\bD).
%\end{align*}
%このような観察に動機付けられて for any $p \geq 0$, define a linear map
%\begin{align*}
%\Psi: M_p(0) \rightarrow \mathrm{Sym}^p(A^2(\bD))
%\end{align*}
%by, for any $n \geq 1$, $p= k_1+\dots+k_n$ and $i_1>i_2>\dots>i_n \geq 0$,
%\begin{align*}
%\Psi(h(-i_1-1)^{k_1}\dots h(-i_n-1)^{k_n}) = \sqrt{p!}\bs^p((h_{i_1}(z))^{\otimes k_1}\otimes \cdots \otimes (h_{i_n}(z))^{\otimes k_n} ).
%\end{align*}
%ここで$\bs^p$は対称化作用素(直行射影 \eqref{eq_symetrizer}) $A^2(\bD)^{\hotimes p} \rightarrow \Sym^p A^2(\bD)$.
%これらの直和により$\Psi: M(0) \rightarrow \mathrm{Sym}(A^2(\bD))$を定義する。
\begin{lem}\label{lem_Psi}
The linear map $\Psi:M(0) \rightarrow \mathrm{Sym}(A^2(\bD))$ is isometric with respect to $(-,-)_M$ and $(-,-)_A$.
%Moreover, the image of the map $\Psi$ は$\CE$代数の $\mathrm{Sym}(A^2(\bD))$ の algebraic core である。
\end{lem}
\begin{proof}
Let $n,m \geq 1$, $p= k_1+\dots+k_n = l_1+\dots+l_m$ and $i_1>i_2>\dots>i_n \geq 0$,  $j_1>j_2>\dots>j_m \geq 0$.
By \eqref{eq_def_inv}, we have
\begin{align*}
( &h(-i_1-1)^{k_1}\dots h(-i_n-1)^{k_n}\va,h(-j_1-1)^{l_1}\dots h(-j_m-1)^{l_m}\va)_M\\
 &=
( \va, h(i_1+1)^{k_1}\dots h(i_n+1)^{k_n}h(-j_1-1)^{l_1}\dots h(-j_m-1)^{l_m}\va )_M\\
&= \begin{cases}
k_1! (i_1+1)^{k_1} k_2!(i_2+1)^{k_2} \dots k_n! (i_n+1)^{k_n},
 & (n=m, i_s=j_s, k_s=l_s \text{ for all }s=1,\dots,n) \\
0, & \text{(otherwise)}.
\end{cases}
\end{align*}
Noting that $(e_i(z),e_j(z))_A = \delta_{i,j}$, we have
\begin{align*}
\norm{\bs(e_{i_1}(z)^{\otimes k_1}\otimes \cdots \otimes e_{i_n}(z)^{\otimes k_n} )}_A^2
&=\frac{k_1!k_2!\dots k_n!}{p!}
\end{align*}
Hence, $\Psi$ is an isometry.
\end{proof}
In particular, the isometry $\Psi$ induces isometric isomorphisms of Hilbert spaces
\begin{align*}
\overline{M_p(0)} \overset{\cong}{\rightarrow} \Sym^p A^2(\bD)
\quad\text{ and }\quad
\overline{M(0)^p} \overset{\cong}{\rightarrow} \Hf^p.
\end{align*}
By Proposition \ref{prop_vertex_conv}, for any $\zeta \in \bD$ and $v \in M(0)^p$, we have
$Y(v,z)\va \Bigl|_{z=\zeta} \in \overline{M(0)^p}$.
%とくにThe isometry $\Psi$はヒルベルト空間の間の isometric isomorphism
%\begin{align*}
%\overline{M_p(0)} \overset{\cong}{\rightarrow} \Sym^p A^2(\bD) \quad\text{ and }\quad
%\overline{M(0)^p} \overset{\cong}{\rightarrow} \Hf^p
%\end{align*}
%を誘導する。
%Proposition \ref{prop_vertex_conv}より任意の$\zeta \in \bD$と$v\in M(0)^p$に対して、$Y(v,z)\va \Bigl|_{z=\zeta} \in \overline{M(0)^p}$が成り立つ。
\begin{lem}\label{lem_vertex_kernel}
For any $\zeta \in \bD$ and $n \geq 0$,
\begin{align*}
\Psi(Y(h(-n-1)\va,z) \va |_{z=\zeta}) = \Psi\left(\frac{1}{n!} \pa_z^n h^-(z)\va |_{z=\zeta}\right) = \frac{1}{n!}\pa_{\zeta}^n E_{\zeta}(z) \in A^2(\bD).
\end{align*}
\end{lem}
\begin{proof}
Since
\begin{align*}
Y(L(-1)^n h(-1)\va,z) \va |_{z=\zeta}= L(-1)^n \exp(\zeta L(-1))h(-1)\va = \sum_{k \geq 0} \frac{\zeta^{k}}{k!} L(-1)^{k+n}h(-1)\va
\end{align*}
and $L(-1)^{k+n}h(-1)\va = (k+n)!h(-k-n-1)\va$, 
\begin{align*}
\Psi(Y(L(-1)^n h,z) \va |_{z=\zeta}) = 
\sum_{k \geq 0} 
\frac{(k+n+1)!}{k!}z^{k+n}\zeta^{k}=\pa_\zeta^n E_\zeta(z).
\end{align*}
%Hence, the assertion follows from $h(-n-1)\va = \frac{1}{n!}L(-1)^n h$.
\end{proof}

Let $\zeta \in \bD$ with $\zeta \neq 0$. Then, there is $r,s>0$ such that
\begin{align}
\overline{B_r(\zeta)} \cap \overline{B_s(0)} =\emptyset\qquad \text{and}\qquad B_r(\zeta),B_s(0) \subset \bD.
\label{eq_sep_z}
\end{align}
Then, $(B_{\zeta,r}, B_{0,s}) \in \CE(2)$ by Proposition \ref{prop_contraction_trans}.
%\eqref{eq_def_B}.
\begin{thm}\label{thm_identification}
Let $\zeta \in \bD$ with $\zeta \neq 0$
and $1>r,s >0$ satisfy \eqref{eq_sep_z}.
Then, for the normalized $\CEh$-algebra \eqref{eq_sqrt}, for any $v_1,v_2 \in M(0)$,
\begin{align}
\m_{(B_{\zeta,r}, B_{0,s})}^\cN(\Psi(v_1),\Psi(v_2))= \Psi(Y(r^{L(0)} v_1,z)s^{L(0)} v_2 |_{z=\zeta})
\label{eq_comparison_thm}
\end{align}
as elements in $\Sym \,A^2(\bD)$.
\end{thm}
The following Lemma follows immediately from Proposition \ref{prop_contraction_trans}:
%Proposition \ref{prop_contraction_trans}より直ちに次が分かる:
\begin{lem}\label{lem_two_point}
For any $n,m \geq 0$,
\begin{align*}
\hat{C}_{(B_{\zeta,r}, B_{0,s})}(h_n(z), h_m(z)) = (-1)^n r^{n+1}s^{m+1}\frac{(n+m+1)!}{n!m!} \zeta^{-n-m-2}.
\end{align*}
\end{lem}
%\begin{proof}
%Since 
%\begin{align*}
%h_{(T_\zeta r^D, s^D)}(z,w) &= \frac{rs}{(\zeta+rz-sw)^2} = rs\sum_{n,m \geq 0}c_{n,m}(rz)^n (sw)^m\\
%& = rs\sum_{n,m \geq 0} \frac{1}{\sqrt{(n+1)(m+1)}}c_{n,m} r^n e_n(z) s^m e_m(w),
%\end{align*}
%%By Lemma \ref{lem_conjugate_cor}, 
%we have
%\begin{align*}
%r^{-n-1}s^{-m-1}&\hat{C}_{(T_\zeta r^D, s^D)}(\sqrt{n+1}e_n(z), \sqrt{m+1}e_m(z)) 
% =c_{n,m}\\
% &= \frac{\pa_z^n}{n!}\frac{\pa_w^m}{m!} \frac{1}{(\zeta+z-w)^2}\Biggl|_{z=w=0}=(-1)^n \frac{(n+m+1)!}{n!m!}\zeta^{-n-m-2}.
%\end{align*}
%\end{proof}

\begin{proof}[proof of Theorem \ref{thm_identification}]
Let $n,m \geq 1$ and $i_1 \geq \dots i_n \geq 0$ and $j_1 \geq \dots \geq j_m \geq 0$.
Set
\begin{align*}
\Delta = (i_1+1) +\dots+ (i_n+1) \qquad\text{and}\qquad \Delta' = (j_1+1) +\dots+ (j_m+1)
\end{align*}
and $[n]=\{1,2,\dots,n\}$, and for a subset $S \subset [n]$ denote its complement $[n] \setminus S$ by $S'$.
Set
\begin{align*}
C_z^{i,j} = \left[\frac{\pa_z^i}{i!} h^+(z), h(-j-1)\right] =(-1)^i \frac{(i+j+1)!}{i! j!} z^{-i-j-2}.
\end{align*}
%It is important to note that $C_z^{i,j}|_\zeta$ coincides with Lemma \ref{lem_two_point}.
Since $h(-s-1)\va = \frac{1}{s!} L(-1)^{s} h$, it follows from the definition of $M(0)$ that
\begin{align*}
&Y(r^{L(0)}h(-i_1-1) \cdots h(-i_n-1) \va,z) s^{L(0)}h(-j_1-1)\cdots h(-j_m-1)\va\\
&=r^{\Delta}s^{\Delta'}\sum_{S \subset [n]}\prod_{p \in S'}\left(\frac{\pa_z^{i_p}}{i_p!} h^-(z)\right) \prod_{s\in S} 
\left(\frac{\pa_z^{i_s}}{i_s!} h^+(z)\right)h(-j_1-1)\cdots h(-j_m-1)\va\\
&=r^{\Delta}s^{\Delta'}\sum_{S \subset [n]}\sum_{T:S \hookrightarrow [m]}
\left(\prod_{s \in S} C_z^{i_s,j_{T(s)}} \right)
\left(\prod_{p \in S'}\frac{\pa_z^{i_p}}{i_p!} h^-(z)\right) \left(\prod_{q \in [m] \setminus T(S)} h(-j_q-1) \right)\va.
\end{align*}
Here $T:S \hookrightarrow [m]$ ranges over all injections from $S$ to $[m]$.
Moreover, $S$ ranges over subsets with $\#S \leq m$; otherwise the second expression vanishes.
By Lemma \ref{lem_vertex_kernel},
\begin{align*}
&\Psi(Y(r^{L(0)} h(-i_1-1) \cdots h(-i_n-1) \va,z)s^{L(0)} h(-j_1-1)\cdots h(-j_m-1)\va|_{z=\zeta})\\
&=r^{\Delta}s^{\Delta'}\sum_{S \subset [n]}\sum_{T:S \hookrightarrow [m]}
\left(\prod_{s \in S} C_\zeta^{i_s,j_{T(s)}} \right)
\sqrt{(n+m-2\#S)!}\bs^{n+m-2\#S}\left(\bigotimes_{p \in S'}\frac{\pa_z^{i_p}}{i_p!} E_\zeta(z) \otimes \bigotimes_{q \in [m] \setminus T(S)} 
h_{j_q}(z) \right).
\end{align*}

On the other hand, using Proposition~\ref{prop_cocycle_vanish}, we compute
\begin{align*}
&\m_{(B_{\zeta,r}, B_{0,s})}^\cN (\Psi(h(-i_1-1) \cdots h(-i_n-1) \va),\Psi(h(-j_1-1)\cdots h(-j_n-1)\va))\\
&=\m_{(B_{\zeta,r}, B_{0,s})}^\cN(\sqrt{n!}\bs^n(h_{i_1}(z) \otimes \cdots \otimes h_{i_n}(z) ),
\sqrt{m!}\bs^m(h_{j_1}(z) \otimes \cdots \otimes h_{j_m}(z) )\\
&=\sqrt{n!m!}\cN \m_{(B_{\zeta,r}, B_{0,s})}(\cN^{-1}
\bs^n(h_{i_1}(z) \otimes \cdots \otimes h_{i_n}(z) ),\cN^{-1}\bs^m(h_{j_1}(z) \otimes \cdots \otimes h_{j_m}(z) ))\\
&=\cN \m_{(B_{\zeta,r}, B_{0,s})}(\bs^n(h_{i_1}(z) \otimes \cdots \otimes h_{i_n}(z) ),\bs^m(h_{j_1}(z) \otimes \cdots \otimes h_{j_m}(z) ))\\
&=\cN \bs
\left(\rho_\cl(B_{\zeta,r})\otimes \rho_\cl(B_{0,s})\right)
 \exp( \hat{C}_{(B_{\zeta,r}, B_{0,s})})(\bs^n(h_{i_1}(z) \otimes \cdots \otimes h_{i_n}(z) ),\bs^m(h_{j_1}(z) \otimes \cdots \otimes h_{j_m}(z) ))\\
&=\sum_{k=0}^{\mathrm{min} \{n,m\}} \frac{\sqrt{(n+m-2k)!}}{k!} \bs^{n+m-2k}\left(\rho_\cl(B_{\zeta,r})\otimes \rho_\cl(B_{0,s})\right)\\
&\times
\hat{C}_{(B_{\zeta,r}, B_{0,s})}^k (\bs^n(h_{i_1}(z) \otimes \cdots \otimes h_{i_n}(z) ),\bs^m(h_{j_1}(z) \otimes \cdots \otimes h_{j_m}(z) ))\\
&=\sum_{k=0}^{\mathrm{min} \{n,m\}} \frac{\sqrt{(n+m-2k)!}}{k!} \bs^{n+m-2k}\left(\rho_\cl(B_{\zeta,r})\otimes \rho_\cl(B_{0,s})\right)\\
&\times \hat{C}_{(B_{\zeta,r}, B_{0,s})}^k (h_{i_1}(z) \otimes \cdots \otimes h_{i_n}(z),h_{j_1}(z) \otimes \cdots \otimes h_{j_m}(z))
\end{align*}
By Lemma \ref{lem_translation_basis}, Lemma \ref{lem_vertex_kernel}, and Lemma \ref{lem_dilation}, we have
\begin{align*}
\rho_1(B_{\zeta,r})(h_i(z)) &=  \frac{r^{i+1}}{i!}\pa_\zeta^i E_\zeta(z),\\
\rho_1(B_{0,s})(h_j(z)) &= s^{j+1} h_j(z).
\end{align*}
Since there are $k!$ contractions in $\hat{C}^k$ specified by $S$ and $T$, the claim follows from Lemma \ref{lem_two_point}.

\end{proof}

\begin{cor}\label{cor_simple}
The $\CE$-algebra $\Sym A^2(\bD)$ is simple.
\end{cor}
\begin{proof}
Set
\begin{align*}
I = \{v \in M(0) \mid (\va, Y(a,z)v )_M =0 \text{ for any }a\in M(0)\}.
\end{align*}
By definition of vertex algebra, it is easy to show that $I$ is an ideal of $M(0)$.
Since $M(0)$ is a simple vertex algebra, $I=0$. By Theorem \ref{thm_identification}, $(\va, Y(a,z)v )_M =0$ if and only if
\begin{align*}
(\va, \rho_{(B_{\zeta,r},B_{0,s})(\Psi(a),\Psi(v))})_A=0.
\end{align*}
Hence, by \cite[Corollary 3.11]{MLeft}, $N_{\text{2-pt}}^\alg(\Sym A^2(\bD))=0$ and $\Sym A^2(\bD)$ is simple.
\end{proof}

In general, for any dimension $d \geq 2$, the algebraic core of a $\CEf_d$-algebra \cite[Definition 1.34]{MLeft} can, under suitable conditions, be endowed with an algebraic structure analogous to a (full) vertex algebra.
Physically, this corresponds to the fact that the conformal field theory admits an operator product expansion and satisfies the bootstrap equations.
We call such a structure a {\it full vertex $d$-algebra}.
The definition of a full vertex $d$-algebra and its relationship with $\CEf_d$-algebras will be discussed in \cite{Mvertex}.

\appendix


\begin{thebibliography}{100}

%\bibitem[ABD]{ABD}
%T.~Abe, G. ~Buhl and C.~Dong, Rationality, regularity, and C2-cofiniteness, Trans. Amer. Math. Soc., 356, 2004, ({\bf8}).
%


 \bibitem[AF1]{AF1}
 {D.~Ayala and J.~Francis},
  {Factorization homology of topological manifolds},
{Journal of Topology}, {\bf 8}, {(4)},
 {1045--1084}, {2015}.

\bibitem[AF2]{AF2}
{D.~Ayala and J.~Francis},
The cobordism hypothesis, arXiv:{1705.02240}.



\bibitem[AGT]{AGT} 
M. S.~Adamo, L.~ Giorgetti, and Y.~Tanimoto, Wightman fields for two dimensional conformal field theories with pointed representation category. Comm. Math. Phys., {\bf404}, (3), 1231--1273, 2023.

\bibitem[AMT]{AMT}
M.S.~Adamo, Y.~Moriwaki and Y.~Tanimoto,
Osterwalder-Schrader axioms for unitary full vertex operator algebras, arXiv:2407.18222.
%
%
%\bibitem[Bar]{Bar}
%D.~Bar-Natan, On associators and the Grothendieck-Teichmuller group. I, Selecta Math. (N.S.), 4, 1998, ({\bf2}), 183--212.
%




\bibitem[Bo1]{B1} R.E.~Borcherds,
Vertex algebras, {K}ac-{M}oody algebras, and the {M}onster,
Proc. Nat. Acad. Sci. U.S.A.,
{\bf83}, 1986, {(10)},
{3068--3071}.



%
%\bibitem[Bu]{Bu}
%G.~Buhl, A spanning set for VOA modules, J. Algebra, 254, 2002, ({\bf1}).
%
\bibitem[Bo2]{B2}
{R.E.~Borcherds},
{Monstrous moonshine and monstrous {L}ie superalgebras},
{Invent. Math.}, {\bf109}, {1992}, {(2)}, 405--444.

\bibitem[Br]{Bru}
{D.~Bruegmann},
{Vertex Algebras and Costello-Gwilliam Factorization Algebras}, {2020},
{arXiv:2012.12214}.


\bibitem[BD]{BD}
{A.~Beilinson and V.~Drinfeld},{Chiral Algebras},
 {American Mathematical Society},
{AMS Colloquium Publications},
 {51}, {2004}.


\bibitem[BKLR]{BKLR}
{M.~Bischoff, Y.~Kawahigashi, R.~Longo and K-H.~Rehren},
{Phase Boundaries in Algebraic Conformal QFT},
{Communications in Mathematical Physics},
{\bf342}, {(1)}, {1--45}, {2016}.





\bibitem[BPZ]{BPZ}
{A. A. ~Belavin, A. M.~Polyakov, A. B.~ Zamolodchikov},
{Infinite conformal symmetry in two-dimensional quantum field theory},
{Nuclear Phys. B}, {\bf241}, {1984}, {(2)}, {333--380}.



%\bibitem[BBP]{BBP}
%T.~Brendle, N.~Broaddus and A.~Putman,
%The mapping class group of connect sums of $S^2 \times S^1$, Trans. Amer. Math. Soc. {\bf376}, 2023, 2557--2572.



\bibitem[CG1]{CG1}
{K.~Costello and O.~Gwilliam},
{Factorization Algebras in Quantum Field Theory. Volume 1},
{New Mathematical Monographs},
{\bf31}, {Cambridge University Press}, {2016}.

\bibitem[CG2]{CG2}
{K.~Costello and O.~Gwilliam},
{Factorization Algebras in Quantum Field Theory. Volume 2},
{New Mathematical Monographs},
{\bf41}, {Cambridge University Press}, {2021}.





%\bibitem[CG]{CG17}
%S.N.~Curry, A.R.~Gover, An Introduction to Conformal Geometry and Tractor Calculus, with a view to Applications in General Relativity, Asymptotic Analysis in General Relativity, London Math Soc Lecture Note Series,  Cambridge University Press, 2018, 86--170.

\bibitem[CKLW]{CKLW}
S.~Carpi, Y.~Kawahigashi, R.~Longo, and M.~Weiner, 
From vertex operator algebras to conformal nets and back, Mem. Amer. Math. Soc., {\bf254}, (1213), 2018.

\bibitem[CN]{CN}
{J. H.~Conway and S. P.~Norton},
{Monstrous Moonshine},
{Bull. London Math. Soc.}, {\bf11},
{(3)}, {308--339}, {1979}.

\bibitem[CS]{CS}
{D.~Calaque and C.~Scheimbauer}, 
{A note on the $(\infty,n)$-category of cobordisms}, {Algebraic \& Geometric Topology}, {\bf19}, {2019}, {533--655}.



%
%\bibitem[CKLR]{CKLR}
%T.~Creutzig, S.~Kanade, A.R.~Linshaw and D.~Ridout,
%Schur-Weyl duality for Heisenberg cosets. Transform. Groups, 24 ({\bf2}) 2019, 301--354.
%
%
%\bibitem[CKM]{CKM22}
%T.~Creutzig, S.~Kanade and R.~McRae. Gluing vertex algebras, Adv. Math., 396:108174, 2022.
%
%
%\bibitem[DM]{DM}
%C.~Dong and G.~Mason, Rational vertex operator algebras and the effective central charge. Int. Math. Res. Not., ({\bf 56}), 2004, 2989--3008.

%
%\bibitem[Dr1]{Dr1}
% {V.G.~Drinfeld},
%{Quasi-{H}opf algebras}, {Algebra i Analiz},
%{\bf 1}, {1989}, {(6)}, {114--148}.
%
%\bibitem[Dr2]{Dr2}
%{V.G.~Drinfeld},
%{On quasitriangular quasi-{H}opf algebras and on a group that
%              is closely connected with {${\rm Gal}(\overline{\bf Q}/{\bf
%              Q})$}}, {Algebra i Analiz},
%{\bf 2}, {1990}, {(4)}, {149--181}.


%\bibitem[Dai]{Dai}
%X.~ Dai, Eta invariant and conformal cobordism,
%Annals of Global Analysis and Geometry 27, 333--340.

\bibitem[DGT]{DGT}
%@article{DamioliniGibneyTarasca2021,
{C.~Damiolini, A.~Gibney and N.~Tarasca},
{Conformal blocks from vertex algebras and their connections on $\overline{\mathcal{M}}_{g,n}$},
{Geometry \& Topology}, {\bf25}, {(5)}, {2021}.

\bibitem[DKRV]{DKRV}
F.~David, A.~Kupiainen, R.~Rhodes, and V.~Vargas,
{Liouville Quantum Gravity on the {R}iemann Sphere},
{Communications in Mathematical Physics}, {342},
 {(3)}, {869--907}, {2016}.

\bibitem[DL]{DL}
C.~Dong and X.~Lin, Unitary vertex operator algebras. J. Algebra, 397, 2014, 252--277.


\bibitem[DLM]{DLM}
{C.~Dong, Z.~Lin, and G.~Mason},
{On vertex operator algebras as {$\mathfrak{sl}_2$}-modules},
{Pacific Journal of Mathematics},
{\bf178}, {1}, {51--77}, {1997}.

%
%\bibitem[EPPRSV]{EPPRSV}
%{S.~El-Showk, M.F.~Paulos, D.~Poland, S.~Rychkov, D.~Simmons-Duffin and A.~Vichi}, {Solving the 3d {I}sing model with the conformal bootstrap
%              {II}. {$c$}-minimization and precise critical exponents}, {J. Stat. Phys.},
%{\bf 157}, {2014}, {(4)-(5)}, {869--914}.

%\bibitem[Fr]{Fr}
%B.~Fresse, Homotopy of operads and Grothendieck-Teichmuller groups. Part 1, Mathematical 
%Surveys and Monographs, 217, American Mathematical Society, Providence, RI, 2017.

\bibitem[FB]{FB}
{E.~Frenkel and D.~Ben-Zvi},
{Vertex algebras and algebraic curves}, {Mathematical Surveys and Monographs}, {\bf88}, {Second},
{American Mathematical Society, Providence, RI}, {2004}.

\bibitem[FLM]{FLM}
I. ~Frenkel, J. ~Lepowsky, and A. ~Meurman, 
{Vertex operator algebras and the {M}onster}, {Pure and Applied Mathematics},
{\bf134}, {Academic Press, Inc., Boston, MA}, {1988}.


%\bibitem[FGG]{FGG}
%{S.~Ferrara and A.F.~ Grillo and R.~Gatto},
%{Tensor representations of conformal algebra and conformally
%covariant operator product expansion}, {Ann. Physics},
%{\bf{76}}, {1973}, {161--188}.
%
\bibitem[FHL]{FHL} I. ~Frenkel, Y. ~Huang and J. ~Lepowsky,
On axiomatic approaches to vertex operator algebras and modules, Mem. Amer. Math. Soc., {\bf104}, 1993, (494).

%\bibitem[FMS]{FMS}
%{P.~Di Francesco, P.~ Mathieu and D.~S\'{e}n\'{e}chal},
%{Conformal field theory}, {Graduate Texts in Contemporary Physics}, {Springer-Verlag, New York}, {1997}.
%
%\bibitem[FR]{FR}
%S.~F\"{o}rste, D.~Roggenkamp, Current-current deformations of conformal field theories, and WZW models, J. High Energy Phys. 5 (2003).

\bibitem[FRS]{FRS}
J.~Fuchs, I.~Runkel and C.~Schweigert, Conformal correlation functions, Frobenius algebras and triangulations, Nucl. Phys. {\bf624}
2002, 452--468.
%
%\bibitem[GKRV]{GKRV1}
%C.~Guillarmou, A.~Kupiainen, R.~Rhodes, and V.~Vargas,
%{Conformal bootstrap in {L}iouville theory},
%{Acta Mathematica}, {\bf233},
% {(1)}, {33--194}, {2024}.
% 
 
\bibitem[GW]{GW}
L.~G{\aa}rding and A.S.~Wightman,
{Fields as Operator-valued Distributions in Relativistic Quantum Theory}, {Arkiv f{\"o}r Fysik},
{\bf28}, {129--189}, {1964}.

 
 
%
%\bibitem[Go]{G}
%{P.~Goddard},
%{Meromorphic conformal field theory},
%{Infinite-dimensional {L}ie algebras and groups ({L}uminy-{M}arseille, 1988)},
%{Adv. Ser. Math. Phys.}, {\bf{7}}, {556--587}, {1989}.
%

%\bibitem[Gi]{GilM}
%O.~Gil-Medrano, On the Yamabe Problem concerning the compact locally conformally flat manifolds,
%Journal of Functional Analysis, {\bf66}, (1), 1986, 42--53.

%\bibitem[Hu1]{H1}
%Y.-Z.~Huang, A theory of tensor products for module categories for a
%vertex operator algebra, IV, J. Pure Appl. Alg., {\bf100}, (1995), 173--216.
%\bibitem[Hu2]{H2}
%Y.-Z.~Huang, Virasoro vertex operator algebras, (nonmeromorphic)
%operator product expansion and the tensor product theory, J. Alg., {\bf182}, (1996), 201--234.
%\bibitem[Hu3]{H3}
%Y.-Z. Huang, Vertex operator algebras and the Verlinde conjecture, Comm. Contemp. Math., {\bf10}, 2008, 103--154.
%\bibitem[Hu4]{H4}
%Y.-Z.~Huang, Rigidity and modularity of vertex tensor categories, Comm. Contemp. Math. {\bf10}, 2008, 871--911.
%
%\bibitem[Hu]{H2} Y.-Z. Huang, Differential equations and intertwining operators, Commun. Contemp. Math., {\bf7}, 2005, (3), 375--400.
%

% 
\bibitem[HK]{HK}
{Y.-Z.~Huang, L.~Kong},
{Full field algebras}, {Comm. Math. Phys.},
{\bf272}, {2007}, {(2)}, {345--396}.

%\bibitem[Ha]{Hatcher}
%A.~Hatcher, On the Diffeomorphism Group of $S^1 \times S^2$, Proc. AMS, Vol. 83, (2), 1981, pp. 427--430.
%
%\bibitem[Hum]{Hum}
%J. E.~Humphreys,
%Introduction to Lie algebras and representation theory, Graduate texts in mathematics, {\bf9}, Springer-Verlag, 1973.

\bibitem[HKZ]{HKZ}
{H.~Hedenmalm, B.~Korenblum, and K.~Zhu},
{Theory of Bergman Spaces}, {Graduate Texts in Mathematics},
{\bf199}, {Springer}, {New York, NY}, {2000}.




%\bibitem[HST]{HST}
%H.~Hohnhold, S.~Stolz and P.~Teichner, From minimal geodesics to supersymmetric field theories, A celebration of the mathematical legacy of Raoul Bott, 207274, CRM Proc. Lecture Notes, 50, Amer. Math. Soc., Providence, RI, 2010.


\bibitem[HT1]{HT1}
{A.G.~Henriques and J.E.~Tener},
{The Segal-Neretin semigroup of annuli}, {arXiv:2410.05929}.

\bibitem[HT2]{HT2}
{A.G.~Henriques and J.E.~Tener},
{Integrating positive energy representations of the Virasoro algebra}, {arXiv:2506.08684}.

\bibitem[Ka]{Kac}
{V.G.~Kac}, {Vertex Algebras for Beginners},
{University Lecture Series},
{\bf10}, {(2)}, {1998}.
 
%\bibitem[HK1]{HK1}
%{Y.-Z.~Huang, L.~Kong},
%{Full field algebras}, {Comm. Math. Phys.},
%{\bf272}, {2007}, {(2)}, {345--396}.
%
%\bibitem[HK2]{HK2}
%{Y.-Z.~Huang, L.~Kong}, Open-string vertex algebra, category and operad. Commun. Math. Phys.,
%{\bf 250}, 2004, 433--471.
% 
% \bibitem[HL]{HL}
%Y.Z. Huang and J. Lepowsky, Tensor products of modules for a vertex operator algebra and vertex
%tensor categories, Progr. Math., 123, 349--383.

%\bibitem[HL1]{HL1}
%Y.-Z.~Huang and J.~Lepowsky, A theory of tensor products for module categories for a vertex operator algebra, I, Selecta Mathematica (New Series), {\bf1}, 1995, 699--756.
%\bibitem[HL2]{HL2}
%Y.-Z.~Huang and J.~Lepowsky, A theory of tensor products for module categories for a vertex operator algebra, II, Selecta Mathematica (New Series), {\bf1}, 1995, 757--786.
%\bibitem[HL3]{HL3}
%Y.-Z.~Huang and J.~Lepowsky, A theory of tensor products for module categories for a vertex operator algebra, III, J. Pure Appl. Alg.,
%{\bf100}, 1995, 141--171.
%
%\bibitem[HM]{HM23}
%G.~H\"{o}hn, S.~M\"{o}ller, Classification of Self-Dual Vertex Operator Superalgebras of Central Charge at Most 24,
%arXiv:2303.17190 [math.QA].


%\bibitem[Id]{Id}
%{N.~Idrissi},{Swiss-cheese operad and {D}rinfeld center},
% {Israel J. Math.}, {\bf221}, {2017},
%{(2)}, {941--972}.
%
%\bibitem[Kui1]{Kuiper}
%N. H.~Kuiper, On conformally flat spaces in the large, Ann. of Math., {\bf 50}, 1949, 916--924.
%
%
%\bibitem[Kui2]{Kuiper2}
%N. H. ~Kuiper, On Compact Conformally Euclidean Spaces of Dimension >2, Ann. of Math.,
%{\bf 52}, (2), 1950, 478--490.

%\bibitem[Kul]{Kul}
%{R. S.~ Kulkarni}, {On the principle of uniformization}, {Journal of Differential Geometry}, {\bf13},
%{1978}, {109--138}.
%%
%
%\bibitem[KS]{KS}
%M.~Kashiwara and P.~Schapira, Categories and Sheaves, Grundlehren der Mathematischen Wissenschaften 332, 
%Springer (2006).

%
%\bibitem[Ko1]{Ko1}
%L.~Kong, Open-closed field algebras, Comm. Math. Phys, {\bf280}, 2008, (1), 207--261.
%%
%\bibitem[Ko2]{Ko2}
%L.~Kong, Full field algebras, operads and tensor categories, Adv. Math, {\bf213}
%(2007) 271--340.
%
%\bibitem[Ko2]{Ko2}
%L.~Kong, Cardy condition for open-closed field algebras, Comm. Math. Phys, {\bf283}, 2008, 25--92.

%
%\bibitem[KL]{KL}
%{D.~Kazhdan and G.~Lusztig},
%{Tensor structures arising from affine {L}ie algebras. {I},
%              {II}}, {J. Amer. Math. Soc.},
%{\bf 6}, {1993}, {(4)}, {905--947, 949--1011}.

%
%\bibitem[KRV]{KRV}
%A.~Kupiainen, R.~Rhodes, and V.~Vargas, 
%{Integrability of {L}iouville theory: proof of the {DOZZ} formula},
%{Annals of Mathematics}, {191},
%{(1)}, {81--166}, {2020}.


%\bibitem[KS1]{KS1}
%{H.~Kajiura and J.~Stasheff}, {Homotopy algebras inspired by classical open-closed string
%              field theory}, {Comm. Math. Phys.}, {\bf263}, {2006}, {(3)}, {553--581}.
%
%\bibitem[KS2]{KS2}
%{H.~Kajiura and J.~Stasheff}, {Open-closed homotopy algebra in mathematical physics}, {J. Math. Phys.},
% {\bf47}, {2006}, {(2)}.
% 
%\bibitem[KS3]{KS3}
%A.~Klimyk and K.~ Schmüdgen, Quantum groups and their representations, Texts and Monographs in Physics, Springer-Verlag, Berlin, (1997)
%
%\bibitem[Lam69]{Lam69}
%J.~Lambek, Deductive systems and categories. {II}. Standard constructions and closed categories,
%{Category {T}heory, {H}omology {T}heory and their
%              {A}pplications, {I} ({B}attelle {I}nstitute {C}onference,
%              {S}eattle, {W}ash., 1968, {V}ol. {O}ne)}, {Lecture Notes in Math.},
%{\bf 86},{1969}, {76--122}.


%
%\bibitem[Li1]{Li}
%H. ~Li, {Symmetric invariant bilinear forms on vertex operator algebras}, {J. Pure Appl. Algebra}, {\bf96}, {1994}, {(3)}, {279--297}.
%
%\bibitem[Li2]{Li2}
% {H.~Li,},
%{Local systems of vertex operators, vertex superalgebras and
%              modules}, {J. Pure Appl. Algebra}, {\bf 109}, {1996}, {(2)}, {143--195}.


\bibitem[Lu1]{Lurie2}
 {J.~Lurie}, {Higher Algebra}, {online version}, {December 22, 2025}.
 
 \bibitem[Lu2]{Lurie3}
 {J.~Lurie}, {On the Classification of Topological Field Theories},
{Current Developments in Mathematics, 2008}, {129--280},
 {2009}, {International Press of Boston}.

%
%\bibitem[Lu]{Lu2}
%{G.~Lusztig},
%{Monodromic systems on affine flag manifolds},
%{Proc. Roy. Soc. London Ser. A},
%{\bf 445}, {1994}, {(1923)}, {231--246};
%Erratum, {Proc. Roy. Soc. London Ser. A}, {\bf 450}, 1995, 731-732.


%\bibitem[La1]{Laudenbach}
%F.~Laudenbach, Sur les 2-sph\`eres d’une vari\'et\'e de dimension 3, Ann. of Math. (2) 97 (1973),
%57--81.
%\bibitem[La2]{Laudenbach2}
%F.~Laudenbach, Topologie de la dimension trois: homotopie et isotopie, Soci\'et\'e Math\'ematique
%de France, Paris, 1974.

%\bibitem[Lio]{Liouville}
%J.~Liouville, Extension au cas des trois dimensions de la questions du trace geographic, in Application de l'Anlyse a la Geometrie, Paris, 1850, 609--616. 

\bibitem[Li]{Li}
H. ~Li, {Symmetric invariant bilinear forms on vertex operator algebras}, {J. Pure Appl. Algebra}, {\bf96}, {1994}, {(3)}, {279--297}.

\bibitem[LL]{LL}
J. ~Lepowsky and H. ~Li, {Introduction to vertex operator algebras and their representations}, {Progress in Mathematics},
{\bf227}, {Birkh\"{a}user Boston, Inc., Boston, MA}, {2004}.

\bibitem[LM]{LM}
M.~L\"{u}scher and G.~Mack,
Global conformal invariance in quantum field theory, Comm. Math. Phys., 41, 203--234, 1975.

%\bibitem[LS]{LS}
%C. ~Lam and H. ~Shimakura, {Quadratic spaces and holomorphic framed vertex operator algebras of central charge 24}, {Proc. Lond. Math. Soc. (3)}, {\bf{104}}, {2012}, {(3)}, {540--576}.

%\bibitem[Ma]{Ma} 
%J.~Maldacena, {The large {$N$} limit of superconformal field theories and
%supergravity}, {Adv. Theor. Math. Phys.}, 
%{\bf2}, {1998}, {2}, {231--252}.
%
%\bibitem[Mac]{Mac} 
%S.~MacLane, Categories for the working mathematician, Graduate Texts in Mathematics, Vol. 5, Springer-Verlag, New York-Berlin, 1971, ix+262.
%
%\bibitem[Mas]{Mas} 
%G.~Mason, Lattice subalgebras of strongly regular vertex operator algebras,
%Conformal Field Theory, Automorphic Forms and Related Topics, Springer, 2014, 31--53 (arXiv:1110.0544v1
%[math.QA]).
%
%\bibitem[MMS]{MMS} 
%{M.~Markl, S.~ Shnider, and J.~Stasheff},
%{Operads in algebra, topology and physics},
%{Mathematical Surveys and Monographs},{\bf96}, {2002}, {x+349}.
%
%\bibitem[Mi1]{Mi1} 
%A.~Milas, Weak modules and logarithmic intertwining operators for vertex operator algebras, Contemp. Math., 297, 201--225.
%\bibitem[Mi2]{Mi2} 
%A.~ Milas, Logarithmic intertwining operators and vertex operators, Comm. Math. Phys., 277, 2008,
%({\bf2}), 497--529.


%\bibitem[Ma]{Matsumoto}
%S.~Matsumoto, Foundations of Flat Conformal Structure,
%Adv. Stud. Pure Math., {\bf 20}, Aspects of Low Dimensional Manifolds, 1992, 167--261.

\bibitem[Mo1]{M1} 
Y.~Moriwaki, Two-dimensional conformal field theory, full vertex algebra and current-current deformation, Adv. Math, 427, (2023).
%
%\bibitem[Mo2]{M2} Y.~Moriwaki, Full vertex algebra and bootstrap -- consistency of four point functions in 2d CFT,
%arXiv:2006.15859 [math.QA].
%
%\bibitem[Mo4]{M4} Y.~Moriwaki, Code conformal field theory and framed algebra, arXiv:2104.10094.
%
%\bibitem[Mo5]{M5}
%Y.~Moriwaki, 
%{Quantum coordinate ring in {WZW} model and affine vertex
%              algebra extensions}, {Selecta Math. (N.S.)}, {\bf 28},
%{2022}, {(4)}, {Paper No. 68, 49}.
%
%\bibitem[Mo2]{M6}
%Y.~Moriwaki, 
%{Vertex operator algebra and parenthesized braid operad}, arXiv:2209.10443.
%
%\bibitem[Mo2]{M8}
%Y.~Moriwaki, 
%{Consistency of operator product expansions of Boundary 2d CFT and Swiss-Cheese operad}, arXiv:2410.02648.

\bibitem[Mo2]{MLeft}
Y.~Moriwaki, Conformally flat factorization homology in Ind-Hilbert spaces and Conformal field theory, arXiv:2602.08729.


\bibitem[Mo3]{Mfactorization}
Y.~Moriwaki, 
Prefactorization algebras for the conformal Laplacian: Central charge and Hilbert Fock space, arXiv:2602.17549.

\bibitem[Mo4]{Mvertex}
Y. Moriwaki, On functorial conformal field theory and higher dimensional full vertex algebras, to appear.





%\bibitem[Mo7]{M7} 
%Y.~Moriwaki, Genus of vertex algebras and mass formula, Math. Z., {\bf 299}, (2021), 1473 -- 1505.
%
%\bibitem[MS1]{MS1}
%G. ~Moore and N. ~Seiberg, Polynomial equations for rational conformal field theories, Phys. Lett. {\bf 212}, 1988, 451--460.
\bibitem[MS]{MS}
G. ~Moore and N. ~Seiberg, Classical and quantum conformal field theory, Comm. Math. Phys. {\bf 123}, 1989, 177--254.

\bibitem[Ni]{Nis}
Y.~Nishinaka, Factorization envelopes and enveloping vertex algebras, arXiv:2512.07635.

%
%\bibitem[NT]{NT}
%K.~ Nagatomo and A.~Tsuchiya, Conformal field theories associated to regular chiral vertex operator
%algebras. I. Theories over the projective line, Duke Math. J., 128, 2005, ({\bf 3}), 393--471.
%
%\bibitem[OS1]{OS1}
% {K.~Osterwalder and R.~Schrader},
%{Axioms for {E}uclidean {G}reen's functions},
%{Comm. Math. Phys.}, {\bf 31}, {1973}, {83--112}.
%\bibitem[OS2]{OS2}
% {K.~Osterwalder and R.~Schrader},
%Axioms for Euclidean Green’s functions. II. Comm. Math. Phys., {\bf42}, 1975, 281--305.

%\bibitem[PS]{PS}
%{M. E. ~Peskin and D. V.~Schroeder},
%{An introduction to quantum field theory},
%{Addison-Wesley Publishing Company, Advanced Book Program, Reading, MA}, {1995}.

%\bibitem[Ploc]{Polc}
%J. Polchinski, String theory. Vol. I, Cambridge Monographs on Mathematical Physics, Cambridge University Press, Cambridge, 1998.

%\bibitem[Sc]{Sch08}
%M.~Schottenloher, A mathematical introduction to conformal field theory, second ed., Lecture Notes in Physics, {\bf759}, Springer-Verlag, Berlin, 2008.


\bibitem[RSS]{RSS}
%@article{RadnellSchippersStaubach2010,
{D.~Radnell, E.~Schippers and W.~Staubach},
{A Hilbert manifold structure on the Weil–Petersson class Teichmüller space of bordered Riemann surfaces},
{Communications in Contemporary Mathematics},
{\bf17}, {(4)}, {2015}.
%
%{The {T}eichm\"uller space of a bordered {R}iemann surface as a complex {H}ilbert manifold},
%{Journal of Functional Analysis},
%{\bf259}, {(10)}, {2653--2674}, {2010}.


\bibitem[Sc]{Sc}
{C.I.~Scheimbauer},
{Factorization Homology as a Fully Extended Topological Field Theory},
{2014}, {ETH Z{\"u}rich}, {Doctoral Thesis (Dr.\ sc.\ ETH Zurich)}.

\bibitem[Se]{Segal}
G.~Segal, The definition of conformal field theory. Topology, geometry and quantum field theory, 421--577, London Math. Soc. Lecture Note Ser., 308, Cambridge Univ. Press, Cambridge, 2004.

%\bibitem[Sh]{Shen}
%Y.~Shen, On Grunsky operator, Sci. China Ser., A 50(12), 1805--1817 (2007).
%
%\bibitem[SW]{SW}
%E. M.~Stein and G.~Weiss, Introduction to Fourier Analysis on Euclidean Spaces, Princeton Mathematical Series, {\bf32}, Princeton University Press, 1971.

\bibitem[SS]{SS}
{E.~Schippers and W.~Staubach},
{Weil--Petersson Teichm\"uller Theory of Surfaces of Infinite Conformal Type},
{In the Tradition of Thurston III}, {Springer}, {2024}, {169--247}.

%\bibitem[ST]{ST}
%S.~Stolz and P.~Teichner, Supersymmetric field theories and generalized cohomology, Mathematical foundations of quantum field theory and perturbative string theory., {\bf83}, Proc. Sympos. Pure Math. Amer. Math. Soc., Providence, RI, 2011, 279--340.


\bibitem[TT]{TT}
{L.A.~Takhtajan and L-P.~Teo},
{Weil-Petersson Metric on the Universal Teichm{\"u}ller Space}, {Memoirs of the American Mathematical Society},
 {183}, {861}, {2006}.
 
 
 \bibitem[TUY]{TUY}
{A.~Tsuchiya and K.~Ueno and Y.~Yamada},
{Conformal field theory on universal family of stable curves with gauge symmetries},
{Advanced Studies in Pure Mathematics}, {\bf19},
 {459--566}, {1989}.


%
%\bibitem[Ta]{Ta}
%D. E. Tamarkin, Formality of chain operad of little discs, Lett. Math. Phys., 66, 2003, ({\bf1-2}), 65--72.

%
%\bibitem[RRTV]{RRTV}
%R.~Rattazzi, V.S.~Rychkov, E.~Tonni and A.~Vichi,
%{Bounding scalar operator dimensions in 4{D} {CFT}},
%{J. High Energy Phys.}, {2008}, {\bf12}.
%
%
\bibitem[Vi]{Vic}
%@article{Vicedo2026FullUEVA,
{B.~Vicedo}, {Full universal enveloping vertex algebras from factorisation}, {Annales Henri Poincar\'e}, {2026}.


%\bibitem[Wh1]{Wh1}
%J. H. C.~Whitehead, On Certain Sets of Elements in a Free Group, Proc. London Math. Soc., (2) 41, 1936, no. 1, 48--56.
%\bibitem[Wh2]{Wh2}
%J. H. C.~Whitehead, On equivalent sets of elements in a free group, Ann. of Math. (2) 37, 1936, 4, 782--800.

%\bibitem[We]{We}
%{S.~Weinberg}, {The quantum theory of fields. {V}ol. {II}}, {Modern applications}, {Cambridge University Press, Cambridge}, {1996}.
%
%\bibitem[ZZ]{ZZ}
%{A.~Zamolodchikov and Al.~Zamolodchikov},
%{Conformal bootstrap in {L}iouville field theory},
%{Nuclear Phys. B}, {\bf477}, {1996}, {(2)}, {577--605}.
%


\bibitem[Zh]{Zh}
{Y.~Zhu}, {Modular invariance of characters of vertex operator algebras},
{J. Amer. Math. Soc.}, {\bf9}, {1996}, {(1)}, {237--302}.
\end{thebibliography}
\end{document}